%% file: main.tex
\documentclass[final]{article}
\usepackage[colorlinks = true, allcolors = blue]{hyperref}
\usepackage[leqno]{amsmath}
\usepackage{mathtools, booktabs, etoolbox}
\usepackage{amssymb, amsfonts, stmaryrd, graphicx, 
cleveref, physics, tikz, tcolorbox, nicefrac, doi}
\usepackage{subcaption, graphicx, float}
\usepackage[numbers,sort]{natbib}
\usepackage[english]{babel}
\usepackage{accents}
\usepackage[letterpaper, top = 2cm,bottom = 2cm, left = 3cm, right = 3cm, marginparwidth = 1.75cm]{geometry}
\usepackage[sc]{mathpazo}
\usepackage[parfill]{parskip}
\usepackage{algorithm}
\usepackage{algpseudocode}
\newcommand{\fullplot}[1]{%
    \makebox[\linewidth][c]{\includegraphics[width=1.08\linewidth]{#1}}%
}
\input{macros}

\input{theoremsetup}

\title{Multigrid Preconditioning for FEEC using Mass-Lumping and Transforming Smoothers}
\date{\today}
\author{Radovan Dabeti\'c\thanks{SAM, D-MATH, ETH Zurich, CH-8092 Zürich, Switzerland, \texttt{rdabetic@sam.math.ethz.ch}}}
\begin{document}

\maketitle

\begin{abstract}
    For PDEs naturally posed in the de Rham complex, structure-preserving mixed and saddle-point finite element discretizations typically produce indefinite linear systems. We propose a multigrid preconditioning framework that combines mass-lumped (explicitly invertible) FEEC mass matrices with transforming smoothers that map the operator to a block form with positive definite diagonal blocks, enabling Gauss-Seidel-type relaxation on the transformed system. Under mild $h$-uniform norm-equivalence assumptions (and for trivial topology), we prove stability of the mass-lumped systems, and by extension spectral equivalence between the mass-lumped and original FEEC operators, which motivates using multigrid cycles designed for the mass-lumped operators as preconditioners for the consistent FEEC systems. While our primary focus is on algorithmic design rather than formal convergence theory, extensive numerical experiments on the Hodge-Dirac operator, mixed Hodge-Laplacians, and a magnetostatics saddle-point system in 2D and 3D demonstrate the robustness of the approach.
\end{abstract}

\section{Introduction}\label{section:introduction}
As elucidated in \cite{CW18}, solving the linear system arising from a finite-element discretization of the mixed formulation of the Hodge-Laplacian in $H_0(\CURL)$ requires special care due to the indefiniteness and saddle-point form of the system. In \cite{CW18}, a multigrid approach using a variable V-cycle on the Schur complement is given and under some assumptions, it is proven that the condition number of the preconditioned system is $\mathcal{O}(1)$.

In this article, we propose an alternative multigrid strategy that utilizes \emph{transforming smoothers} on a mass-lumped system, in which the mass matrices possess explicit inverses. Our approach is motivated by the work in \cite{Dabetic2025}, where spectral equivalence was shown between the systems obtained from FEEC and Discrete Exterior Calculus (DEC) discretizations of the Hodge–Dirac operator, and multigrid schemes based on transforming smoothers were observed to be efficient in practice.

To illustrate the mechanism of transforming smoothers, consider the concrete example of the mixed formulation of the Hodge-Laplacian on 1-forms. Formally, the system has the form (in vector proxies)
\begin{equation}\label{eq:hcurl_lap_strong_mixed}
    \begin{pmatrix}
        -I & -\DIV \\
        \GRAD & \CURL\CURL
    \end{pmatrix}.
\end{equation}
On the discrete level, we get a symmetric, but indefinite, linear system, hence smoothers like Gauss-Seidel or Jacobi cannot be expected to work.

At the continuous level, this operator can be (at least formally) transformed as follows:
\begin{equation}\label{eq:intro_transforming_smoother_example}
    \begin{pmatrix}
        -I & -\DIV \\
        \GRAD & \CURL\CURL
    \end{pmatrix}
    \begin{pmatrix}
        -I & -\DIV \\
        \GRAD & I
    \end{pmatrix}
    =
    \begin{pmatrix}
        I - \DIV\GRAD & 0 \\
        -\GRAD & \CURL\CURL - \GRAD\DIV
    \end{pmatrix}
\end{equation}
If the discrete system mimics this behavior, we can use this transformed system for smoothening, as the transformed system is lower-block-triangular and the diagonal blocks are positive definite.

The necessity of mass-lumping arises naturally when attempting to replicate this continuous composition at the discrete level. Let $V_h$ denote the chosen finite element space, and let $V_h'$ be its dual space. The standard Galerkin discretizations of the system and the transformation yield discrete operators $\mathsf{A}_h: V_h \to V_h'$ and $\mathsf{T}_h: V_h \to V_h'$, respectively. Consequently, the direct composition $\mathsf{A}_h \mathsf{T}_h$ is ill-defined: the range of $\mathsf{T}_h$ lies in the dual space $V_h'$, whereas the domain of $\mathsf{A}_h$ is $V_h$. 

To bridge these spaces, we must interpose the discrete Riesz isomorphism $\mathsf{R}_h: V_h' \to V_h$, yielding the well-defined composition $\mathsf{A}_h \mathsf{R}_h \mathsf{T}_h$. In the matrix-vector setting, evaluating this exact discrete Riesz map is equivalent to multiplying by the inverse of the consistent mass matrix. Because the inverse of the mass matrix is dense in general, computing the exact composition is prohibitively expensive. This necessitates approximating the Riesz map by replacing the consistent mass matrix with a diagonal one. The mass-lumping will (in general) not lead to a consistent discretization, but it may still be useful as a preconditioner, as will be motivated in \Cref{section:discrete}.

We will employ this approach to solve problems for the Hodge-Dirac operator, Hodge-Laplacian and a saddle-point problem arising from Maxwell's equations (the magnetic potential in the magnetostatic case in the Coulomb gauge), which all fit into the framework of FEEC.

In \Cref{section:framework}, we introduce notation relating to differential forms and state the problems of interest.

In \Cref{section:discrete}, the problems on the discrete level are formulated and it is proven that under relatively mild conditions on the mass-lumping, we have spectral equivalence of the mass-lumped and original problems.

In \Cref{section:mg}, the multigrid algorithms and smoothing procedures are described in detail. Finally, we show some empirical results in \Cref{section:results}, demonstrating that this is a valid approach and that it merits further investigation.

\section{Continuous Framework}\label{section:framework}
As we will be mostly interested in problems involving \emph{essential} boundary conditions, we will also work with spaces involving these boundary conditions. However, most results can be carried over to the case of natural boundary conditions with minor modifications.

Let $\Omega\subset\R^n$ be a bounded, Lipschitz, polytopal, and, for the sake of simplicity, \emph{topologically trivial} domain.

\begin{remark}
    Though we only deal with the case of trivial topologies, it is straightforward to extend the stability arguments (using \Cref{proposition:dec_poincare}) for the discrete systems to non-trivial topologies, once we impose the solutions be orthogonal to harmonic forms.
\end{remark}

\subsection{Differential Forms}
Let $\XLambdak[][k]$ denote the space of smooth $k$-forms on $\Omega$.
% \ footnote{The notation in this manuscript is largely adopted from \cite{GUP25}.}
As in \cite{GUP25} the exterior derivative operators are denoted by
$\d[k]: \XLambdak[][k]\to\XLambdak[][k + 1]$, $0\leq k<n$, the (Euclidean) Hodge star
operators by $\star_k$ and the codifferential operators by
$\delta_k := (-1)^k\star_{k - 1}^{-1} \d[n-k]\star_k: \XLambdak[][k]\to\XLambdak[][k -
1], k = 1, \dots, n$. The Hodge star operators induce inner products on $\XLambdak[][k]$.
\begin{definition}
  \label{def:1}
  The $L^2$ inner product on two $k$-forms $\omega$ and $\mu$ is given by 
  \begin{gather*}
    \inner{\omega}{\mu}_{L^2\XLambdak[][k]} := \int_\Omega\omega\wedge\star\mu.
  \end{gather*}
\end{definition}

We denote by $\XLambdak := \bigoplus_{k = 0}^n \XLambdak[][k]$ the exterior
algebra of (smooth) differential forms on $\Omega$ and write 
\begin{gather}
  \label{ddd}
  \d :=
  \begin{pmatrix}
    0 &  & \\
    \d[0] & 0 & \\
    & \d[1] & 0 & \\
    && \ddots & \ddots
  \end{pmatrix}, \quad \delta :=
  \begin{pmatrix}
    0 & \delta_1 & \\
    & 0 & \delta_2 \\
    & & 0 & \ddots \\
    && & \ddots
  \end{pmatrix}
\end{gather}
for the exterior derivative and codifferential on $\XLambdak$. We equip
$\XLambdak$ with the natural Hilbert space structure by combining the inner products
from \Cref{def:1}. For $\mathfrak{u} \equiv (u_0, \dots, u_n), \mathfrak{v} \equiv (v_0, \dots, v_n)\in\XLambdak$
we set
\begin{gather*}
  \inner{\mathfrak{u}}{\mathfrak{v}}_{\XLambdak[L^2]} := \sum_{k = 0}^n
  \inner{u_k}{v_k}_{\XLambdak[L^2][k]}\;.
\end{gather*}
Write $L^2\XLambdak:= \bigoplus_{k = 0}^n L^2\XLambdak[][k]$, where
$L^2\XLambdak[][k]$ is the space of square-integrable $k$-forms, i.e.\ $k$-forms with
coefficients in $L^2(\Omega)$.

Also refer to \cite[Section 6.2.6]{Arnold2018}, where Sobolev spaces of differential forms
are introduced. Let
\begin{gather*}
  \XLambdak[H][k] := \left\{ u_k\in \XLambdak[L^2][k]: \d[k] u_k\in
    \XLambdak[L^2][k + 1]\right\}\;,
\end{gather*}
and define
$\mathring{V} := \bigoplus_{k = 0}^{n-1}\mathring{H}\XLambdak[][k] \oplus
L^2_*\XLambdak[][n]$, where
\begin{gather}
  L^2_*\XLambdak[][n] := \left\{v\in L^2\XLambdak[][n]:
  \int_\Omega v = 0 \right\}\;,
\end{gather}
and $\mathring{H}\XLambdak[][k]$ is the
space of functions in $H\XLambdak[][k]$ with vanishing trace on $\partial\Omega$,
see \cite[Section 6.2.6]{Arnold2018}. Also let $\XLambdak[H], \XLambdak[H^*]$ be the domain of
$\d$ and $\delta$ respectively, see also \cite[Section 6.2.6]{Arnold2018}.

\subsection{Hodge-Dirac Operator}

\begin{definition}
  \label{def:D}
  The Hodge-Dirac operator is $\DD := \d + \delta$ with domain of definition 
  $\mathcal{D}(\DD) := \mathring{H}\Lambda(\Omega)\cap H^*\Lambda(\Omega)$, where the domain of $\d$ is
  $\mathring{H}\Lambda(\Omega)$ and that of $\delta$ is $H^*\Lambda(\Omega)$. 
\end{definition}

Taking the cue from \cite{Leopardi2016} we put the focus on the following boundary
value problem\footnote{Note that the boundary conditions are included implicitly in the
domain of the operator.} for the Dirac operator: Given $\mathfrak{f} = (f_0, \dots, f_n) \in L^2\XLambdak$, seek
$\mathfrak{u}\in\mathcal{D}(\DD)\cap(\ker\DD)^\perp, \mathfrak{p}\in\ker\DD$ such that
\begin{equation}\label{eq:strong_form}
  \DD\mathfrak{u} + \mathfrak{p} = \mathfrak{f}.
\end{equation}
Corollary 8 of \cite{Leopardi2016} tells us the well-posedness of the following weak form of
\eqref{eq:strong_form}: Given $\mathfrak{f}\in L^2\XLambdak$, seek
$\mathfrak{u}\in\mathring{H}\XLambdak, \mathfrak{p}\in\ker\DD$ such that
\begin{align}\label{eq:var}
  \begin{aligned}
    \inner{\d\mathfrak{u}}{\mathfrak{v}}_{\XLambdak[L^2]} +
    \inner{\mathfrak{u}}{\d\mathfrak{v}}_{\XLambdak[L^2]} +
    \inner{\mathfrak{p}}{\mathfrak{v}}_{\XLambdak[L^2]} &
    = \inner{\mathfrak{f}}{\mathfrak{v}}_{\XLambdak[L^2]} &&
    \forall\mathfrak{v}\in\mathring{H}\Lambda(\Omega), \\
    \inner{\mathfrak{u}}{\mathfrak{v}}_{\XLambdak[L^2]} & = 0 &&
    \forall\mathfrak{v}\in\ker\DD.
  \end{aligned}
\end{align}
As we are working with a domain with trivial topology, $\ker\DD$ (the space of
harmonic forms) is trivial (see \cite[Section 4.3]{Arnold2018} for more information)
except for constant $n$-forms, i.e.\ $\ker\DD|_{\mathring{V}} = \{0\}$, so that we can
consider the following simpler problem: Given $\mathfrak{f}\in L^2\XLambdak$ with
$\int_\Omega f_n = 0$\footnote{Or given the general case, we can recover
  $\mathfrak{p}$ by taking the mean of the $n$-form in $\mathfrak{f}$ and then
  subtract the mean to get a suitable right-hand side.}, seek
$\mathfrak{u}\in\mathring{V}$ such that
\begin{gather}
  \label{eq:dirac_weak}
  \inner{\d\mathfrak{u}}{\mathfrak{v}}_{\XLambdak[L^2]} +
  \inner{\mathfrak{u}}{\d\mathfrak{v}}_{\XLambdak[L^2]} =
  \inner{\mathfrak{f}}{\mathfrak{v}}_{\XLambdak[L^2]}
  \quad\forall\mathfrak{v}\in\mathring{V}.
\end{gather}

\begin{remark}
  \label{rem:dopf}
  For $n=3$ in physics the Dirac operator (or more specifically, the Dirac Hamiltonian of a free particle) is usually expressed by means of Pauli matrices as a
  differential operator acting on functions $\Omega\to\mathbb{C}^{4}$. Identifying $\mathbb{C}^{4}$ with
  $\mathbb{R}^{8}$ in a suitable way, this is exactly the differential operator of
  \Cref{def:D} as is explained in \cite[Section~2]{CHR18} and \cite[Section~1.1]{FRI23}.
\end{remark}

\subsection{Hodge-Laplacian}
\begin{definition}
    The Hodge-Laplacian on $k$-forms is 
    \[
        \mathsf{L}^k u_k := \left(\delta_{k + 1}\d[k] + \d[k - 1]\delta_k\right) u_k,
    \] where we understand $\delta_{k + 1}, \d[k - 1]$ to be zero if $0 \leq k \pm 1 \leq n$ is not the case. The domain is $$\mathcal{D}(\mathsf{L}^k) := \left\{u_k\in\XLambdak[\mathring{H}][k] \cap\XLambdak[{H}^*][k]: \d[k]u_k\in\XLambdak[{H}^*][k + 1] \land \delta_ku_k\in\XLambdak[\mathring{H}][k - 1]\right\}.$$
\end{definition}
The problem of interest (in strong form) is: given $f\in\XLambdak[L^2][k]$ (or $\XLambdak[L_*^2][n]$ in the case of $k = n$), find $u\in\mathcal{D}(\mathsf{L}^k)$ such that $\mathsf{L}^k u = f$.

We will use the mixed weak formulation from \cite[Chapter 4.4]{Arnold2018}, where well-posedness and stability are also established.

Explicitly: given $f\in\XLambdak[L^2][k]$ (or $\XLambdak[L_*^2][n]$ in the case of $k = n$), seek $\sigma\in\XLambdak[\mathring{H}][k - 1], u\in\XLambdak[\mathring{H}][k]$ (or $u\in\XLambdak[L_*^2][n]$ in the case of $k = n$) such that
\begin{equation}
\label{eq:laplace_weak}
\begin{aligned}
    \inner{u}{\d[k-1]\tau}_{\XLambdak[L^2][k]} - \inner{\sigma}{\tau}_{\XLambdak[L^2][k - 1]} & = 0 & \forall\tau\in\XLambdak[\mathring{H}][k - 1]\\
    \inner{\d[k-1]\sigma}{v}_{\XLambdak[L^2][k]} + \inner{\d[k]u}{\d[k]v}_{\XLambdak[L^2][k + 1]} & = \inner{f}{v}_{\XLambdak[L^2][k]} & \forall v\in\XLambdak[\mathring{H}][k].
\end{aligned}
\end{equation}
We have omitted the harmonic forms here, as we are working with a trivial topology.

\subsection{Magnetostatics}
This problem is motivated by Maxwell's equations (cf.\ \cite[Equation~5]{CW18}): for a given $f\in\XLambdak[L^2][1]$ with zero divergence (in the distributional sense), seek $u\in\XLambdak[\mathring{H}][1], \sigma\in\XLambdak[\mathring{H}][0]$:
\begin{equation}
    \label{eq:magnetostatics_weak}
    \begin{aligned}
        \inner{u}{\d[0]\tau}_{\XLambdak[L^2][1]} & = 0 \quad&&\forall\tau\in\XLambdak[\mathring{H}][0] \\
        \inner{\d[1]u}{\d[1]v}_{\XLambdak[L^2][2]} + \inner{\d[0]\sigma}{v}_{\XLambdak[L^2][1]} & = \inner{f}{v}_{\XLambdak[L^2][1]} \quad&&\forall v\in\XLambdak[\mathring{H}][1]
    \end{aligned}
\end{equation}

\begin{proposition}\label{proposition:magnetostatics_stability}
    \eqref{eq:magnetostatics_weak} is well-posed and stable.
\end{proposition}
\begin{proof}
    This follows from the Babu\v{s}ka-Brezzi theorem, see \cite[Theorem 4.2.3]{BB13}. The coercivity and inf-sup condition follow trivially from the Poincar\'e inequality and the fact that we are working with a trivial topology.
\end{proof}

\section{Discrete Setting}\label{section:discrete}
Let $\mathcal{T}$ be a member of a \emph{uniformly shape-regular} sequence $\left( \mathcal{T}_{h}\right)$ of simplicial meshes of $\Omega$ indexed by their meshwidths $h$, which tend to zero.

We adopt the notation $a \lesssim b$ if $a \leq C b$ holds with a constant $C$ that does not depend on $h$, and $a\sim b$ if $a\lesssim b \;\land\; b\lesssim a$.

For the discretization, we rely on lowest order Whitney-forms $\whitneyForms[k]\subset\XLambdak[H][k]$, or more specifically the subspace with boundary conditions $\ringWhitneyForms[k]\subset\XLambdak[\mathring{H}][k]$, see \cite{Arnold2018} for more. We also introduce the exterior algebra of lowest order Whitney forms (with boundary conditions) as $\ringWhitneyForms := \bigoplus_{k = 0}^{n - 1} \ringWhitneyForms[k] \oplus \whitneyForms[n, *]$, where $\whitneyForms[n, *]$ designates the discrete $n$-forms with zero mean over $\Omega$. To simplify notation, we write $\ringWhitneyForms[n] := \whitneyForms[n, *]$.

\begin{definition}
    The operator $\d[k]$ restricted to $\ringWhitneyForms[k]$ is denoted by $\dh[k]$ and $\d$ restricted to $\ringWhitneyForms$ by $\dh$.
\end{definition}

\subsection{Mass-Lumping}
In \cite{Hiptmair2006}, an abstract approach to establishing spectral equivalence of problems arising from bilinear forms fulfilling an inf-sup condition is elucidated. If we do mass-lumping and the mass-lumped and original bilinear forms are both bounded $h$-uniformly in the Sobolev norms and if they fulfill an inf-sup inequality, we may employ similar means to prove spectral equivalence.

\begin{remark}
    We will see that the preconditioner we provide leads to an asymmetric linear system, so we will have to resort to a solver for generic linear systems like \texttt{GMRES}.
    It should be emphasized that spectral equivalence alone is \emph{insufficient} to establish that \texttt{GMRES} converges in an $h$-independent number of iterations for a prescribed tolerance. 
    However, it nevertheless offers at least some motivation for investigating the convergence numerically.
\end{remark}

Therefore, what we intend to show in this section is that under reasonable assumptions on the mass-lumping, we can obtain a Poincar\'e inequality, which we can use for establishing stability and spectral equivalence of the finite-element and mass-lumped problems.

\begin{remark}
    On tensor-product meshes using quadrilaterals, consistent mass-lumping can be performed (e.g.\ trapezoidal; essentially finite-differences) to arrive at a diagonal mass matrix, and the approach outlined here can even be implemented in a matrix-free fashion. This approach is known in literature under \emph{distributive relaxation}, cf.\ \cite{BD79}, \cite[Chapter 11.3]{Hackbusch1985}.
\end{remark}

\begin{definition}
    The Riesz isomorphism on $\ringWhitneyForms[k]$ equipped with the inner product $\inner{\cdot}{\cdot}_{\XLambdak[L^2]}$ is defined as $$\rieszFE[k]: \hilbertSpaceFE[k]'\to\hilbertSpaceFE[k],$$ and the one of $\hilbertSpaceFE$ by $\rieszFE$.

    The mass-lumped inner product on $\ringWhitneyForms[k]$ is denoted by $\innerc{\cdot}{\cdot}[k]$, the corresponding Riesz isomorphism is denoted by $$\rieszDEC[k]: \hilbertSpaceDEC[k]'\to\hilbertSpaceDEC[k],$$ and the one of $\hilbertSpaceDEC$ by $\rieszDEC$. The induced norms are denoted by $\normm{\cdot}[k]$ on $\ringWhitneyForms[k]$ and by $\normm{\cdot}$ on $\ringWhitneyForms$.

    Furthermore, define $\innerc{\cdot}{\cdot}[\XLambdak[H][k]] := \innerc{\cdot}{\cdot}[k] + \innerc{\d[k]\cdot}{\d[k]\cdot}[k]$, the induced norm $\normm{\cdot}[\XLambdak[H][k]]$, and $\innerc{\cdot}{\cdot}[\XLambdak[{H}]], \normm{\cdot}[\XLambdak[{H}]]$ analogously.
\end{definition}
\begin{definition}\label{definition:iso_decfe}
    Via the equality of sets, we define isomorphisms taking us from the original space to the mass-lumped one:
    \[
        \IsoFEDEC[k]: \hilbertSpaceFE[k]\to\hilbertSpaceDEC[k], \qquad \IsoSobolevFEDEC[k]: \hilbertSpaceFE[k][H]\to\hilbertSpaceDEC[k][\XLambdak[H][k]],
    \] and analogously for $\IsoFEDEC, \IsoSobolevFEDEC$.
\end{definition}
\begin{assumption}[$h$-uniform Norm Equivalence]\label{assumption:norm_eq}
  The norms $\normm{\cdot}[k]$ and $\norm{\cdot}_{\XLambdak[L^2][k]}$ on the finite-element spaces are $h$-uniformly equivalent. More
  precisely, there exist constants $c_-, c_+ > 0$ independent of $h$, such that
    \[
        c_- \normm{\IsoFEDEC[k] u_h}[k] \leq \norm{u_h}_{\XLambdak[L^2][k]} \leq c_+ \normm{\IsoFEDEC[k] u_h}[k]\quad\forall u_h\in \ringWhitneyForms[k],\ 0\leq k\leq n.
    \] 
    This is equivalent to the operator norms of $\IsoFEDEC[k]$ and $\IsoFEDEC$ (and their inverses) being $h$-uniformly bounded from above and below.

    Note that this also implies that the operator norms of $\IsoSobolevFEDEC[k]$ and $\IsoSobolevFEDEC$ (and their inverses) are $h$-uniformly bounded from above and below.
\end{assumption}

\begin{definition}
    The discrete adjoint of $\dh[k]$ in $\hilbertSpaceDEC$ is denoted by $\deltaML[k + 1]$. The adjoint of $\dh$ is defined analogously and denoted by $\deltaML$. Similarly, let $\deltaFE[k], \deltaFE$ denote the discrete adjoints of $\dh[k]$ and $\dh$, but in $\hilbertSpaceFE$.
\end{definition}

\begin{definition}
    Let $\perpFEEC$ denote orthogonality w.r.t.\ the original inner product, and $\perpML$ w.r.t.\ the mass-lumped one. Then define
    \begin{align*}
        \cycles[k]      &\mathrel{:=} \ker\dh[k],
        &\qquad \boundaries[k]  &\mathrel{:=} \ran\dh[k - 1], \\
        \harmonics[k]   &\mathrel{:=} \cycles[k] \cap \boundaries[k, \perpFEEC],
        &\qquad \harmonicsML[k] &\mathrel{:=} \cycles[k] \cap \boundaries[k, \perpML].
    \end{align*}
\end{definition}

\begin{lemma}[Discrete Hodge-Decomposition]\label{lemma:dec_hodge}
    We have the orthogonal decompositions
    \[
    \begin{alignedat}{3}
        \ringWhitneyForms[k] &= \boundaries[k]\; &\oplus_{\mathmakebox[\widthof{$\perpFEEC$}][c]{\perpFEEC}} &\mathmakebox[\widthof{$\cycles[k,\perpFEEC]$}][l]{\cycles[k, \perpFEEC]}\; &\oplus_{\mathmakebox[\widthof{$\perpFEEC$}][c]{\perpFEEC}} &\harmonics[k], \\
        \ringWhitneyForms[k] &= \boundaries[k]\; &\oplus_{\mathmakebox[\widthof{$\perpFEEC$}][c]{\perpML}}   &\mathmakebox[\widthof{$\cycles[k,\perpFEEC]$}][l]{\cycles[k, \perpML]}\;   &\oplus_{\mathmakebox[\widthof{$\perpFEEC$}][c]{\perpML}}   &\harmonicsML[k].
    \end{alignedat}
    \] Moreover, $\dim\harmonicsML[k] = \dim\harmonics[k]$.
\end{lemma}
\begin{proof}
    This is a trivial consequence of the fact that $\dh[k]\dh[k - 1] \equiv 0 \equiv \deltaML[k - 1]\deltaML[k] \equiv \delta^h_{k - 1}\delta^h_k$, that $\deltaML$ (or $\delta^h_k$) and $\dh$ are adjoint to one another in the respective inner products. See also \cite[Section 2.1]{Leopardi2016} and \cite[Chapter 5.2.2]{Arnold2018}.
\end{proof}

\begin{proposition}[Poincar\'e Inequality in Mass-Lumped Norms]\label{proposition:dec_poincare}
    For $u_h\in\cycles[k, \perpML]$ it holds that
    \[
        \normm{u_h}[k] \lesssim \normm{\dh[k] u_h}[k + 1].
    \]
\end{proposition}
\begin{proof}
    Let us decompose $u_h$ into a \emph{FEEC-orthogonal} decomposition
    \begin{equation}\label{eq:feec_hodge}
        u_h = \dh[k - 1] \alpha_h + \beta_h + \gamma_h,
    \end{equation}
    where $\alpha_h \in\cycles[k - 1, \perpFEEC], \beta_h \in\cycles[k, \perpFEEC]$ and $\gamma_h\in\harmonics[k]$.

    We now decompose $\alpha_h, \beta_h$ and $\gamma_h$ using a \emph{mass-lumped-orthogonal} decomposition, i.e.\ 
    \begin{alignat}{2}
      \alpha_h &= \dh[k-2] \bar{\alpha}_h + \tilde{\alpha}_h + \alpha_{h,0},
      &\qquad \bar{\alpha}_h \in \cycles[k-2, \perpML],\; \alpha_{h,0} \in \harmonicsML[k-1],\nonumber\\
      \beta_h  &= \dh[k-1] \bar{\beta}_h + \tilde{\beta}_h + \beta_{h,0},
      &\qquad \bar{\beta}_h \in \cycles[k-1, \perpML],\; \tilde{\beta}_h \in \cycles[k, \perpML],\; \beta_{h,0} \in \harmonicsML[k],\label{eq:beta_h_decomp}\\
      \gamma_h &= \dh[k-1] \bar{\gamma}_h + \tilde{\gamma}_h + \gamma_{h,0},
      &\qquad \bar{\gamma}_h \in \cycles[k-1, \perpML],\; \tilde{\gamma}_h \in \cycles[k, \perpML],\; \gamma_{h,0} \in \harmonicsML[k].\nonumber
    \end{alignat}
    We note that due to $\dh[k]\gamma_h = 0$ (it is harmonic) and $\tilde{\gamma}_h \in \cycles[k, \perpML]$, we may imply that $\tilde{\gamma}_h = 0$. We thus get the following form of the mass-lumped-orthogonal Hodge-decomposition of $u_h$:
    \begin{equation}\label{eq:ml_orth_hodge_decomp_uh}
        u_h = \dh[k - 1]\left[ \tilde{\alpha}_h + \bar{\beta}_h + \bar{\gamma}_h \right] + \left[ \tilde{\beta}_h \right] + \left[ \beta_{h, 0} + \gamma_{h, 0} \right].
    \end{equation}
    Recall that by the assumption of the lemma, we have $u_h\in\cycles[k, \perpML]$, so by uniqueness of the decomposition, we follow $u_h = \tilde{\beta}_h$.
    The rest essentially follows from applying the FEEC Poincar\'e inequality (see \cite[Theorem 5.2]{Arnold2018} and also \cite{Leopardi2016}) to $\beta_h$, or more precisely
    \begin{align*}
        \normm{u_h}[k] & = \normm{\tilde{\beta}_h}[k] \lesssim \normm{\beta_h}[k] && [\text{Orthogonality of \eqref{eq:beta_h_decomp}}] \\
        & \lesssim \formNorm{\beta_h}[L^2][k] \leq \formNorm{\beta_h}[H][k] && [\text{\Cref{assumption:norm_eq}}] \\
        & \lesssim \formNorm{\dh[k]\beta_h}[L^2][k + 1] && [\text{FEEC Poincar\'e \cite[Theorem 5.2]{Arnold2018}}] \\
        & = \formNorm{\dh[k]u_h}[L^2][k + 1] && [\text{\Cref{eq:feec_hodge}}] \\
        & \lesssim \normm{\dh[k]u_h}[k + 1] && [\text{\Cref{assumption:norm_eq}}],
    \end{align*} which concludes the proof.
\end{proof}

\subsection{Hodge-Dirac Operator}
\subsubsection{Discrete Problem}
In the discrete case, for the Dirac operator we seek $\mathfrak{u}_h\in\ringWhitneyForms$ such that
\begin{equation}\label{eq:dirac_feec_discrete_problem}
  \diracBFEEC{\mathfrak{u}_h}{\mathfrak{v}_h} := \inner{\d\mathfrak{u}_h}{\mathfrak{v}_h}_{\XLambdak[L^2]} +
  \inner{\mathfrak{u}_h}{\d\mathfrak{v}_h}_{\XLambdak[L^2]} =
  \inner{\mathfrak{f}}{\mathfrak{v}_h}_{\XLambdak[L^2]}
  \quad\forall\mathfrak{v}_h\in\ringWhitneyForms.
\end{equation} For 
\[
    \diracBFEEC{\cdot}{\cdot}: \hilbertSpaceFE\times\hilbertSpaceFE\to\R,
\] let $\diracGalerkinOpFEEC: \hilbertSpaceFE\to\hilbertSpaceFE'$ denote the induced operator and define 
\[ 
    \diracOpFEEC := \rieszFE\diracGalerkinOpFEEC: \hilbertSpaceFE\to\hilbertSpaceFE.
\]
Equivalently, $\diracOpFEEC$ can be written as
\[
    \diracOpFEEC = \d + \deltaFE.
\]
When replacing the $L^2$ by the Sobolev norms, we can also regard the bilinear form as
\[
    \diracBFEEC{\cdot}{\cdot}: \hilbertSpaceFE[][H]\times\hilbertSpaceFE[][H]\to\R,
\] in which case we denote the induced operator by $\diracSobolevOpFEEC: \hilbertSpaceFE[][H]\to\hilbertSpaceFE[][H]'$.
\begin{remark}
    By choosing a basis and instantiating $\diracGalerkinOpFEEC$ as an explicit matrix, we retrieve the usual Galerkin matrix using the chosen basis of the finite element space.
\end{remark}

Applying mass-lumping (replacing the inner products), we get the corresponding bilinear form
\begin{equation}\label{eq:dirac_dec_bilinear_form}
\diracBDEC{\mathfrak{u}_h}{\mathfrak{v}_h} := \innerc{\d\mathfrak{u}_h}{\mathfrak{v}_h} +
  \innerc{\mathfrak{u}_h}{\d\mathfrak{v}_h},
\end{equation}
 and the induced linear map $\diracGalerkinOpDEC: \hilbertSpaceDEC\to\hilbertSpaceDEC'$. Analogously, we define the endomorphism $\diracOpDEC := \rieszDEC\diracGalerkinOpDEC$. Equivalently,
\[
    \diracOpDEC = \dh + \deltaML.
\]
We denote the induced map of the bilinear form on the mass-lumped discrete Sobolev space by $\diracSobolevOpDEC: \hilbertSpaceDEC[][\XLambdak[H]]\to\hilbertSpaceDEC[][\XLambdak[H]]'$.

Using the Poincar\'e inequality, we can prove stability for the mass-lumped problem.
\begin{proposition}[{{Inf-Sup Inequality, \cite[Theorem~6]{Leopardi2016}}}]\label{lemma:infsup_dec}
    For all nonzero $\mathfrak{u}_h\in\ringWhitneyForms$, there exists a nonzero $\mathfrak{v}_h\in\ringWhitneyForms$ satisfying
    \[
        \normm{\mathfrak{u}_h}[\XLambdak[H]]\normm{\mathfrak{v}_h}[\XLambdak[H]] \lesssim \diracBDEC{\mathfrak{u}_h}{\mathfrak{v}_h}.
    \]
\end{proposition}
\begin{proof}
    The proof is almost the same as the one from \cite[Theorem~6]{Leopardi2016}, but with the norms replaced by the mass-lumped variants. We will go through it for completeness' sake.
    
    Take the Hodge decomposition of $$\mathfrak{u}_h = \d \mathfrak{r}_h + \mathfrak{w}_h$$ with $\mathfrak{r}_h\in\left(\ker\dh\right)^\perp$ and define $\mathfrak{v}_h := \mathfrak{r}_h + \d\mathfrak{u}_h$.
    
    The Poincar\'e inequality from \Cref{proposition:dec_poincare} together with the orthogonality of the Hodge decomposition from \Cref{lemma:dec_hodge} yields
    \begin{equation}\label{eq:infsup_v}
    \begin{aligned}
        \normm{\mathfrak{v}_h}[\XLambdak[H]] & \leq \normm{\mathfrak{r}_h}[\XLambdak[H]] + \normm{\d\mathfrak{u}_h}[\XLambdak[H]] \\
        & \leq c_p\normm{\d \mathfrak{r}_h} + \normm{\d\mathfrak{u}_h} \\
        & \lesssim \normm{\mathfrak{u}_h}[\XLambdak[H]],
    \end{aligned}
    \end{equation} where $c_p \geq 1$ is given through the Poincar\'e inequality and we used that $\normm{\d\mathfrak{u}_h}[\XLambdak[H]] = \normm{\d\mathfrak{u}_h}$. Note that $c_p \geq 1$ has to hold, as we are estimating $\normm{\mathfrak{v}_h}[\XLambdak[H]]$ by its exterior derivative.
    
    Substituting this into the bilinear form and once again using the Hodge decomposition and Poincar\'e inequality gives
    \begin{align*}
        \diracBDEC{\mathfrak{u}_h}{\mathfrak{v}_h} & = \normm{\d\mathfrak{u}_h}^2 + \innerc{\mathfrak{u}_h}{\d\mathfrak{r}_h} \\
        & = \normm{\d\mathfrak{u}_h}^2 + \innerc{\d\mathfrak{r}_h}{\d\mathfrak{r}_h} \\
        & = \frac{1}{2}\normm{\d\mathfrak{u}_h}^2 + \frac{1}{2}\normm{\d\mathfrak{w}_h}^2 + \normm{\d\mathfrak{r}_h}^2 \\
        & \geq \frac{1}{2}\normm{\d\mathfrak{u}_h}^2 + \frac{1}{2 c_p^2}\normm{\mathfrak{w}_h}^2 + \normm{\d\mathfrak{r}_h}^2 \\
        & \geq \frac{1}{2 c_p^2} \normm{\mathfrak{u}_h}[\XLambdak[H]]^2,
    \end{align*} where the last inequality follows from $c_p \geq 1$. Combining this with \eqref{eq:infsup_v} yields the statement.
\end{proof}

\subsubsection{Spectral Equivalence}
\begin{lemma}[{{FEEC Inf-Sup Inequality, \cite[Theorem~10]{Leopardi2016}}}]\label{lemma:feec_stab}
    For all nonzero $u_h\in\ringWhitneyForms$, there exists a nonzero $v_h\in\ringWhitneyForms$ satisfying
    \[
        \norm{u_h}_{\XLambdak[H]}\norm{v_h}_{\XLambdak[H]} \lesssim \diracBFEEC{u_h}{v_h}.
    \]
\end{lemma}
\begin{lemma}[$h$-uniform Boundedness of Bilinear Forms and Inverses]\label{lemma:h_uniform_opnorms}
    It holds that
    \[\begin{array}{ccc}
        \norm{\diracSobolevOpFEEC}_{\hilbertSpaceFE[][H]\to\hilbertSpaceFE[][H]'}
        &,&
        \norm{\diracSobolevOpFEEC^{-1}}_{\hilbertSpaceFE[][H]'\to\hilbertSpaceFE[][H]},\\
        \norm{\diracSobolevOpDEC}_{\hilbertSpaceDEC[][{\XLambdak[H]}]\to\hilbertSpaceDEC[][{\XLambdak[H]}]'}
        &,&
        \norm{\diracSobolevOpDEC^{-1}}_{\hilbertSpaceDEC[][{\XLambdak[H]}]'\to\hilbertSpaceDEC[][{\XLambdak[H]}]}
    \end{array}\]
    are all bounded from above by constants independent of the mesh-width.
\end{lemma}
\begin{proof}
    For the FEEC problem, boundedness of the bilinear form yields
    \[
        \diracBFEEC{u_h}{v_h} \leq 2\norm{u_h}_{\XLambdak[H]}\norm{v_h}_{\XLambdak[H]}\quad\forall u_h, v_h \in \ringWhitneyForms,
    \]
    hence
    \[
        \norm{\diracSobolevOpFEEC}_{\hilbertSpaceFE[][H]\to\hilbertSpaceFE[][H]'} \leq 2.
    \]
    Moreover, \Cref{lemma:feec_stab} implies the inf-sup condition for $\diracBFEEC{\cdot}{\cdot}$ on $\hilbertSpaceFE[][H]$, and therefore
    \[
        \norm{\diracSobolevOpFEEC^{-1}}_{\hilbertSpaceFE[][H]'\to\hilbertSpaceFE[][H]} \lesssim 1.
    \]

    For the mass-lumped problem, we analogously have
    \[
        \diracBDEC{u_h}{v_h} \leq 2\normm{u_h}[\XLambdak[H]]\normm{v_h}[\XLambdak[H]]\quad\forall u_h, v_h \in\ringWhitneyForms,
    \]
    which implies
    \[
        \norm{\diracSobolevOpDEC}_{\hilbertSpaceDEC[][{\XLambdak[H]}]\to\hilbertSpaceDEC[][{\XLambdak[H]}]'} \lesssim 1.
    \]
    Furthermore, \Cref{lemma:infsup_dec} yields
    \[
        \norm{\diracSobolevOpDEC^{-1}}_{\hilbertSpaceDEC[][{\XLambdak[H]}]'\to\hilbertSpaceDEC[][{\XLambdak[H]}]} \lesssim 1,
    \]
    concluding the proof.
\end{proof}

\begin{theorem}[Spectral Equivalence to FEEC, cf.\ \cite{Hiptmair2006}]\label{theorem:spec_eq_dirac}
    Let
    \[
        \kappa\left(\mathsf{A}\right) := \norm{\mathsf{A}}_{\hilbertSpaceFE[][H]\to\hilbertSpaceFE[][H] }\,\norm{\mathsf{A}^{-1}}_{\hilbertSpaceFE[][H]\to\hilbertSpaceFE[][H]}
    \]
    denote the condition number of an automorphism $\mathsf{A}: \hilbertSpaceFE[][H]\to\hilbertSpaceFE[][H]$ and $\rho(\mathsf{A})$ the spectral radius. Then
    \begin{gather*}
        \kappa\left( \left[ \IsoSobolevFEDEC[*] \diracSobolevOpDEC \IsoSobolevFEDEC \right]^{-1} \diracSobolevOpFEEC \right) \lesssim 1, \\
        \rho\left( \left[ \IsoSobolevFEDEC[*] \diracSobolevOpDEC \IsoSobolevFEDEC \right]^{-1} \diracSobolevOpFEEC \right) \lesssim 1, \quad
        \rho\left( \diracSobolevOpFEEC^{-1} \IsoSobolevFEDEC[*] \diracSobolevOpDEC \IsoSobolevFEDEC \right) \lesssim 1.
    \end{gather*}
\end{theorem}

\begin{proof} 
    Let $\mathsf{T} := \left[ \IsoSobolevFEDEC[*] \diracSobolevOpDEC \IsoSobolevFEDEC \right]^{-1} \diracSobolevOpFEEC$. We first note that due to \Cref{assumption:norm_eq} and the submultiplicativity of operator norms, we have
    \begin{align*}
        \norm{\mathsf{T}}_{\hilbertSpaceFE[][H]\to\hilbertSpaceFE[][H]} 
        &\lesssim \norm{\diracSobolevOpDEC^{-1}}_{\hilbertSpaceDEC[][{\XLambdak[H]}]'\to\hilbertSpaceDEC[][{\XLambdak[H]}]} 
        \norm{\diracSobolevOpFEEC}_{\hilbertSpaceFE[][H]\to\hilbertSpaceFE[][H]'}
    \end{align*}
    and
    \begin{align*}
        \norm{\mathsf{T}^{-1}}_{\hilbertSpaceFE[][H]\to\hilbertSpaceFE[][H]} 
        &\lesssim \norm{\diracSobolevOpDEC}_{\hilbertSpaceDEC[][{\XLambdak[H]}]\to\hilbertSpaceDEC[][{\XLambdak[H]}]'} 
        \norm{\diracSobolevOpFEEC^{-1}}_{\hilbertSpaceFE[][H]'\to\hilbertSpaceFE[][H]}.
    \end{align*}
    The estimate on the condition number is then a trivial consequence of \Cref{lemma:h_uniform_opnorms}. The eigenvalues can be bounded from above by the operator norms, so the estimates on $\rho(\mathsf{T}), \rho\left(\mathsf{T}^{-1}\right)$ follow immediately, which concludes the proof.
\end{proof}

\subsection{Hodge-Laplacian}

For \eqref{eq:laplace_weak}, we get: seek $\sigma_h\in\ringWhitneyForms[k - 1], u_h\in\ringWhitneyForms[k]$ such that
\begin{equation}\label{eq:laplace_feec_discrete_problem}
\begin{array}{rclcll}
    \inner{u_h}{\d[k - 1]\tau_h}_{\XLambdak[L^2][k]} & - & \inner{\sigma_h}{\tau_h}_{\XLambdak[L^2][k - 1]} & = & 0, & \forall\tau_h\in\ringWhitneyForms[k - 1]\\
    \inner{\d[k-1]\sigma_h}{v_h}_{\XLambdak[L^2][k]} & + & \inner{\d[k]u_h}{\d[k]v_h}_{\XLambdak[L^2][k + 1]} & = & \inner{f}{v_h}_{\XLambdak[L^2][k]}, & \forall v_h\in\ringWhitneyForms[k].
\end{array}
\end{equation}
To ease the notation, we introduce the spaces
\begin{equation}\label{eq:laplace_product_spaces}
\begin{aligned}
    \laplaceDiscWFEEC[k]   &:= \hilbertSpaceFE[k-1]\times\hilbertSpaceFE[k], \quad & \laplaceDiscWSobolevFEEC[k] &:= \hilbertSpaceFE[k-1][H]\times\hilbertSpaceFE[k][H], \\
    \laplaceDiscWDEC[k]   &:= \hilbertSpaceDEC[k-1]\times\hilbertSpaceDEC[k], \quad & \laplaceDiscWSobolevDEC[k] &:= \hilbertSpaceDEC[k-1][{\XLambdak[H][k-1]}]\times\hilbertSpaceDEC[k][{\XLambdak[H][k]}].
\end{aligned}
\end{equation}
The associated bilinear form on $\laplaceDiscWFEEC[k] \times \laplaceDiscWFEEC[k]$ is
\begin{equation}\label{eq:laplace_feec_bilinear_form}
  \begin{aligned}
    \laplaceBFEEC{(\sigma_h, u_h)}{(\tau_h, v_h)}[k]
    :=\;& \inner{\sigma_h}{\tau_h}_{\XLambdak[L^2][k - 1]}
    - \inner{u_h}{\d[k - 1]\tau_h}_{\XLambdak[L^2][k]} \\
    &- \inner{\d[k - 1]\sigma_h}{v_h}_{\XLambdak[L^2][k]}
    - \inner{\d[k]u_h}{\d[k]v_h}_{\XLambdak[L^2][k + 1]} .
  \end{aligned}
\end{equation}
For the bilinear form $\laplaceBFEEC{\cdot}{\cdot}[k]: \laplaceDiscWFEEC[k] \times \laplaceDiscWFEEC[k] \to \R$, let 
\begin{equation}\label{eq:laplace_galerkin_op_feec}
    \laplaceGalerkinOpFEEC[k]: \laplaceDiscWFEEC[k] \to (\laplaceDiscWFEEC[k])'
\end{equation}
denote the induced operator and define 
\begin{equation}\label{eq:laplace_op_feec_def}
    \laplaceOpFEEC[k] := \rieszFE[k]\laplaceGalerkinOpFEEC[k]: \laplaceDiscWFEEC[k] \to \laplaceDiscWFEEC[k].
\end{equation}
Equivalently,
\[
    \laplaceOpFEEC[k]
    =
    \begin{pmatrix}
        -\Id & \deltaFE[k] \\
        \dh[k-1] & \deltaFE[k+1]\dh[k]
    \end{pmatrix}.
\]
When replacing the $L^2$ norms by the Sobolev norms, we view the bilinear form on $\laplaceDiscWSobolevFEEC[k] \times \laplaceDiscWSobolevFEEC[k] \to \R$, in which case we denote the induced operator by 
\begin{equation}\label{eq:laplace_sobolev_op_feec}
    \laplaceSobolevOpFEEC[k]: \laplaceDiscWSobolevFEEC[k] \to (\laplaceDiscWSobolevFEEC[k])'.
\end{equation}

Applying mass-lumping (replacing the inner products), we get the corresponding bilinear form $\laplaceBDEC{\cdot}{\cdot}[k]: \laplaceDiscWDEC[k] \times \laplaceDiscWDEC[k] \to \R$ and the induced linear map 
\begin{equation}\label{eq:laplace_galerkin_op_dec}
    \laplaceGalerkinOpDEC[k]: \laplaceDiscWDEC[k] \to (\laplaceDiscWDEC[k])'.
\end{equation}
Analogously, we define the endomorphism $\laplaceOpDEC[k] := \rieszDEC[k]\laplaceGalerkinOpDEC[k]$. Equivalently,
\[
    \laplaceOpDEC[k]
    =
    \begin{pmatrix}
        -\Id & \deltaML[k] \\
        \dh[k-1] & \deltaML[k+1]\dh[k]
    \end{pmatrix}.
\]
We denote the induced map of the bilinear form on the mass-lumped discrete Sobolev space by
\begin{equation}\label{eq:laplace_sobolev_op_dec}
    \laplaceSobolevOpDEC[k]: \laplaceDiscWSobolevDEC[k] \to (\laplaceDiscWSobolevDEC[k])'.
\end{equation}

\begin{lemma}[{Inf-Sup Inequality for FEEC, \cite[Theorem 5.4]{Arnold2018}}]\label{lemma:laplace_feec_stability}
    $\laplaceBFEEC{\cdot}{\cdot}[k]$ fulfills an inf-sup condition with a constant that is independent of the mesh-width.
\end{lemma}
\begin{lemma}[{Inf-Sup Inequality for the Mass-Lumped System}]\label{lemma:laplace_dec_stability}
    $\laplaceBDEC{\cdot}{\cdot}[k]$ fulfills an inf-sup condition with a constant that is independent of the mesh-width.
\end{lemma}
\begin{proof}
    For the proof, we refer to \cite[Theorem 4.9]{Arnold2018}. The steps in the proof are the same as for FEEC, as we have a Poincar\'e inequality by \Cref{proposition:dec_poincare} and a Hodge-decomposition by \Cref{lemma:dec_hodge}.
\end{proof}

\begin{theorem}[Spectral Equivalence to FEEC]\label{theorem:spec_eq_laplace}
    Let
    \[
        \kappa\left(\mathsf{A}\right) := \norm{\mathsf{A}}_{\laplaceDiscWSobolevFEEC[k]\to \laplaceDiscWSobolevFEEC[k]}\norm{\mathsf{A}^{-1}}_{\laplaceDiscWSobolevFEEC[k]\to \laplaceDiscWSobolevFEEC[k]}
    \]
    denote the condition number of an automorphism $\mathsf{A}: \laplaceDiscWSobolevFEEC[k]\to \laplaceDiscWSobolevFEEC[k]$ and $\rho(\mathsf{A})$ the spectral radius. Then
    \begin{gather*}
        \kappa\left( \left[ \IsoSobolevFEDEC[k,*] \laplaceSobolevOpDEC[k] \IsoSobolevFEDEC[k] \right]^{-1} \laplaceSobolevOpFEEC[k] \right) \lesssim 1, \\
        \rho\left( \left[ \IsoSobolevFEDEC[k,*] \laplaceSobolevOpDEC[k] \IsoSobolevFEDEC[k] \right]^{-1} \laplaceSobolevOpFEEC[k] \right) \lesssim 1, \quad
        \rho\left( \left[ \laplaceSobolevOpFEEC[k] \right]^{-1} \IsoSobolevFEDEC[k,*] \laplaceSobolevOpDEC[k] \IsoSobolevFEDEC[k] \right) \lesssim 1.
    \end{gather*}
\end{theorem}
\begin{proof}
    The proof follows the same arguments as in \Cref{theorem:spec_eq_dirac}, as we have stability bounds from \Cref{lemma:laplace_feec_stability} and \Cref{lemma:laplace_dec_stability} and will be skipped for the sake of brevity, see also \cite{Hiptmair2006}.
\end{proof}

\subsection{Magnetostatics}
\Cref{eq:magnetostatics_weak} in the discrete case becomes: seek $\sigma_h\in\ringWhitneyForms[0], u_h\in\ringWhitneyForms[1]$ such that
\begin{equation}
    \label{eq:discrete_magnetostatics}
    \begin{aligned}
        \inner{u_h}{\d[0]\tau_h}_{\XLambdak[L^2][1]} & = 0 \quad&&\forall\tau_h\in\ringWhitneyForms[0] \\
        \inner{\d[0]\sigma_h}{v_h}_{\XLambdak[L^2][1]} + \inner{\d[1]u_h}{\d[1]v_h}_{\XLambdak[L^2][2]} & = \inner{f}{v_h}_{\XLambdak[L^2][1]} \quad&&\forall v_h\in\ringWhitneyForms[1].
    \end{aligned}
\end{equation}

To ease the notation, we introduce the product spaces for the magnetostatics problem (corresponding to $k=1$):
\begin{equation}\label{eq:mag_product_spaces}
\begin{aligned}
    \magDiscWFEEC[1]   &:= \hilbertSpaceFE[0]\times\hilbertSpaceFE[1], \quad & \magDiscWSobolevFEEC[1] &:= \hilbertSpaceFE[0][H]\times\hilbertSpaceFE[1][H], \\
    \magDiscWDEC[1]   &:= \hilbertSpaceDEC[0]\times\hilbertSpaceDEC[1], \quad & \magDiscWSobolevDEC[1] &:= \hilbertSpaceDEC[0][{\XLambdak[H][0]}]\times\hilbertSpaceDEC[1][{\XLambdak[H][1]}].
\end{aligned}
\end{equation}

The associated bilinear form on $\magDiscWFEEC[1] \times \magDiscWFEEC[1]$ is
\begin{equation}\label{eq:mag_feec_bilinear_form}
  \begin{aligned}
    \magBFEEC{(\sigma_h, u_h)}{(\tau_h, v_h)}[1]
    :=\;& \inner{u_h}{\d[0]\tau_h}_{\XLambdak[L^2][1]} \\
    &+ \inner{\d[0]\sigma_h}{v_h}_{\XLambdak[L^2][1]}
    + \inner{\d[1]u_h}{\d[1]v_h}_{\XLambdak[L^2][2]}.
  \end{aligned}
\end{equation}
For the bilinear form $\magBFEEC{\cdot}{\cdot}[1]: \magDiscWFEEC[1] \times \magDiscWFEEC[1] \to \R$, let 
\begin{equation}\label{eq:mag_galerkin_op_feec}
    \magGalerkinOpFEEC[1]: \magDiscWFEEC[1] \to (\magDiscWFEEC[1])'
\end{equation}
denote the induced operator and define 
\begin{equation}\label{eq:mag_op_feec_def}
    \magOpFEEC[1] := \rieszFE[1]\magGalerkinOpFEEC[1]: \magDiscWFEEC[1] \to \magDiscWFEEC[1].
\end{equation}
Equivalently,
\[
    \magOpFEEC[1]
    =
    \begin{pmatrix}
        0 & \deltaFE[1] \\
        \dh[0] & \deltaFE[2]\dh[1]
    \end{pmatrix}.
\]
When replacing the $L^2$ norms by the Sobolev norms, we view the bilinear form on $\magDiscWSobolevFEEC[1] \times \magDiscWSobolevFEEC[1] \to \R$, in which case we denote the induced operator by 
\begin{equation}\label{eq:mag_sobolev_op_feec}
    \magSobolevOpFEEC[1]: \magDiscWSobolevFEEC[1] \to (\magDiscWSobolevFEEC[1])'.
\end{equation}

Applying mass-lumping (replacing the inner products), we get the corresponding bilinear form $\magBDEC{\cdot}{\cdot}[1]: \magDiscWDEC[1] \times \magDiscWDEC[1] \to \R$ and the induced linear map 
\begin{equation}\label{eq:mag_galerkin_op_dec}
    \magGalerkinOpDEC[1]: \magDiscWDEC[1] \to (\magDiscWDEC[1])'.
\end{equation}
Analogously, we define the endomorphism $\magOpDEC[1] := \rieszDEC[1]\magGalerkinOpDEC[1]$, or equivalently,
\[
    \magOpDEC[1]
    =
    \begin{pmatrix}
        0 & \deltaML[1] \\
        \dh[0] & \deltaML[2]\dh[1]
    \end{pmatrix}.
\]
We denote the induced map of the bilinear form on the mass-lumped discrete Sobolev space by
\begin{equation}\label{eq:mag_sobolev_op_dec}
    \magSobolevOpDEC[1]: \magDiscWSobolevDEC[1] \to (\magDiscWSobolevDEC[1])'.
\end{equation}

\begin{lemma}[Discrete Stability]\label{lemma:magnetostatics_discrete_stability}
    We have well-posedness and stability of \eqref{eq:discrete_magnetostatics}, i.e.\ a unique solution $(\sigma_h, u_h)$ which also fulfills
    \[
        \norm{u_h}_{\XLambdak[H][1]} + \norm{\sigma_h}_{\XLambdak[H][0]} \lesssim C\norm{f}_{\XLambdak[L^2][1]}.
    \]

    Furthermore, we have a similar result for the solution of the mass-lumped system $(\tilde{\sigma}_h, \tilde{u}_h)$, where
    \[
        \normm{\tilde{u}_h}[\XLambdak[H][1]] + \normm{\tilde{\sigma}_h}[\XLambdak[H][0]] \lesssim \norm{f}_{\XLambdak[L^2][1]}.
    \]
\end{lemma}
\begin{proof}
    The proof follows the same argument as in \Cref{proposition:magnetostatics_stability}, as we have the discrete Poincar\'e inequality for the FEEC system from \cite[Theorem 5.2]{Arnold2018} and for the mass-lumped system from \Cref{proposition:dec_poincare}.
\end{proof}
\begin{remark}\label{remark:magnetostatics_sobolev_stability}
    We considered the case where the right-hand-side is given by $\inner{f}{\cdot}_{\XLambdak[L^2][1]}$, but the same argument also gives us stability if the right-hand-side is any bounded linear functional on $\XLambdak[H][1]$, in which case the norm of the linear functional appears in the stability bounds of \Cref{lemma:magnetostatics_discrete_stability} instead of $\norm{f}_{\XLambdak[L^2][1]}$.
\end{remark}

\begin{theorem}[Spectral Equivalence to FEEC]\label{theorem:spec_eq_mag}
    Let
    \[
        \kappa\left(\mathsf{A}\right) := \norm{\mathsf{A}}_{\magDiscWSobolevFEEC[1]\to \magDiscWSobolevFEEC[1]}\norm{\mathsf{A}^{-1}}_{\magDiscWSobolevFEEC[1]\to \magDiscWSobolevFEEC[1]}
    \] 
    denote the condition number of an automorphism $\mathsf{A}: \magDiscWSobolevFEEC[1]\to \magDiscWSobolevFEEC[1]$ and $\rho(\mathsf{A})$ the spectral radius. Then
    \begin{gather*}
        \kappa\left( \left[ \IsoSobolevFEDEC[1,*] \magSobolevOpDEC[1] \IsoSobolevFEDEC[1] \right]^{-1} \magSobolevOpFEEC[1] \right) \lesssim 1, \\
        \rho\left( \left[ \IsoSobolevFEDEC[1,*] \magSobolevOpDEC[1] \IsoSobolevFEDEC[1] \right]^{-1} \magSobolevOpFEEC[1] \right) \lesssim 1, \quad
        \rho\left( \left[ \magSobolevOpFEEC[1] \right]^{-1} \IsoSobolevFEDEC[1,*] \magSobolevOpDEC[1] \IsoSobolevFEDEC[1] \right) \lesssim 1.
    \end{gather*}
\end{theorem}
\begin{proof}
    We have stability of the bilinear forms by \Cref{lemma:magnetostatics_discrete_stability}\footnote{Technically speaking, we have stability for right-hand-sides that are linear functionals on the Sobolev space, see \Cref{remark:magnetostatics_sobolev_stability}.} and continuity, the proof is analogous to \Cref{theorem:spec_eq_dirac} and \Cref{theorem:spec_eq_laplace} and will be skipped for the sake of brevity.
\end{proof}

\section{Multigrid}\label{section:mg}
To avoid repeating derivations for the Hodge-Dirac, Hodge-Laplacian, and magnetostatics problems, we introduce a unified generic notation. We will develop a multigrid method for the \emph{mass-lumped systems} using transforming smoothers, which will then serve as a preconditioner for the standard FEEC problem (see \Cref{subsection:preconditioning}).

Let $\genericDiscWFEEC$ and $\genericDiscWDEC$ denote the finite element spaces for the standard FEEC and mass-lumped discretizations, respectively. They are connected by a canonical isomorphism $\genericIso : \genericDiscWFEEC \to \genericDiscWDEC$ (with dual $\genericIso^*$) whose operator norms, along with its inverse, are $h$-uniformly bounded. The associated bilinear forms $\genericBFEEC{\cdot}{\cdot}$ and $\genericBDEC{\cdot}{\cdot}$ induce the standard and mass-lumped Galerkin operators mapping to their respective dual spaces:
\[
    \genericGalerkinOpFEEC : \genericDiscWFEEC \to \genericDiscWFEEC' \quad \text{and} \quad \genericGalerkinOpDEC : \genericDiscWDEC \to \genericDiscWDEC'.
\]
Using the corresponding Riesz isomorphisms $\genericRieszFE : \genericDiscWFEEC' \to \genericDiscWFEEC$ and $\genericRieszDEC : \genericDiscWDEC' \to \genericDiscWDEC$, we can define the endomorphisms $\genericOpFEEC := \genericRieszFE \genericGalerkinOpFEEC$ and $\genericOpDEC := \genericRieszDEC \genericGalerkinOpDEC$.

To formulate a multi-grid algorithm, we need a hierarchy of nested triangulations $\mathcal{T}_{h_0}, \mathcal{T}_{h_1}, \dots, \mathcal{T}_{h_L}$ of the underlying domain. The grid levels are indexed by $\ell = 0, 1, \dots, L$, where $\ell = 0$ corresponds to the coarsest grid and $\ell = L$ represents the finest target resolution. Each level is characterized by a mesh-width $h_\ell$, satisfying $h_0 > h_1 > \dots > h_L$. This sequence induces a nested hierarchy of FEM spaces:
\[
    \genericDiscWFEEC[][h_0] \subset \genericDiscWFEEC[][h_1] \subset \dots \subset \genericDiscWFEEC[][h_L] = \genericDiscWFEEC.
\]

In what follows, we define the multigrid algorithm explicitly for the mass-lumped endomorphism $\genericOpDEC$. One could also work directly with $\genericGalerkinOpDEC$, and doing so requires adjusting the transfer operators and smoothers we are going to define below. However, both formulations remain algebraically equivalent. The reason we like to work with endomorphisms is to define the transforming smoothers, as elucidated in \Cref{section:introduction}.

\subsection{Preconditioning the FEEC Problems}\label{subsection:preconditioning}
To see how a multigrid method developed for the mass-lumped endomorphism $\genericOpDEC$ is applied to solve the original FEEC Galerkin system, consider the target linear system
\begin{equation}\label{eq:feec_target_system}
    \genericGalerkinOpFEEC u = g,
\end{equation}
where $g \in \genericDiscWFEEC'$ is a given right-hand side. We can construct a preconditioner $\mathsf{B} : \genericDiscWFEEC' \to \genericDiscWFEEC$ by using the inverse of the mass-lumped system:
\begin{equation}\label{eq:preconditioner_ideal_action}
    \mathsf{B} \approx (\genericIso^* \genericGalerkinOpDEC \genericIso)^{-1} = \genericIso^{-1} (\genericGalerkinOpDEC)^{-1} (\genericIso^*)^{-1} = \genericIso^{-1} (\genericOpDEC)^{-1} \genericRieszDEC(\genericIso^*)^{-1}.
\end{equation}

Replacing the exact inverse $(\genericOpDEC)^{-1}$ with the action of a multigrid cycle $\mathsf{M}_{\textnormal{MG}} \approx (\genericOpDEC)^{-1}$, we obtain the action of the preconditioner for the original FEEC problem:
\begin{equation}\label{eq:feec_preconditioner_action}
    \mathsf{B} := \genericIso^{-1} \mathsf{M}_{\textnormal{MG}} \genericRieszDEC (\genericIso^*)^{-1}.
\end{equation}

Evaluating the preconditioner for a given fine-grid residual $r \in \genericDiscWFEEC'$ therefore requires three steps:
\begin{enumerate}
    \item Map the FEEC residual to the mass-lumped dual space via $(\genericIso^*)^{-1}$, and apply the mass-lumped Riesz isomorphism $\genericRieszDEC$ to obtain the primal vector $\genericRieszDEC (\genericIso^*)^{-1} r \in \genericDiscWDEC$. Because $\genericRieszDEC$ corresponds to the inverse of a diagonal mass matrix, this operation is explicit and cheap.
    \item Apply a multigrid cycle $\mathsf{M}_{\textnormal{MG}}$.
    \item Map back to the original FEEC space using the inverse isomorphism $\genericIso^{-1}$.
\end{enumerate}

In practice, $\genericIso$ will correspond to an identity matrix (we use the same basis for $\genericDiscWFEEC$ and $\genericDiscWDEC$), so on the levels of vectors and matrices, it is a no-op.

Note that even if the multigrid method were to yield a self-adjoint operator, the resulting preconditioned operator is indefinite. Standard Krylov subspace methods, such as Preconditioned MINRES (or Preconditioned Conjugate Gradient (PCG), but our problem is fundamentally indefinite, so PCG is not applicable in any case), fundamentally require the preconditioner to be positive definite to induce a valid inner product. Because our preconditioner lacks this property, these algorithms are not applicable. Consequently, we are forced to use solvers for generic, non-symmetric and indefinite linear systems like GMRES to solve the preconditioned system.

\subsection{Transfer Operators}\label{subsection:transfer_operators}
To formulate the coarse-grid correction scheme, we define appropriate transfer operators between adjacent levels in our grid hierarchy. Consider a fine grid level $\ell$ with mesh-width $h_\ell$ and its consecutive coarser level $\ell-1$ with mesh-width $h_{\ell-1}$. 

Because the spaces are nested, we have a canonical prolongation operator $\prolong[h_{\ell-1}] : \genericDiscWDEC[][h_{\ell-1}] \to \genericDiscWDEC[][h_\ell]$. Its dual operator, mapping between the respective dual spaces, is denoted by the adjoint $\prolong[h_{\ell-1}]^* : (\genericDiscWDEC[][h_\ell])' \to (\genericDiscWDEC[][h_{\ell-1}])'$. 

Let $\genericGalerkinOpDEC[][h_\ell]$ and $\genericGalerkinOpDEC[][h_{\ell-1}]$ denote the mass-lumped Galerkin operators assembled on the fine and coarse spaces, respectively. For the problem $\genericGalerkinOpDEC[][h_\ell] u = g$ with residual $r := g - \genericGalerkinOpDEC[][h_\ell] u \in (\genericDiscWDEC[][h_\ell])'$, a standard coarse-grid correction step is given by
\begin{equation}\label{eq:standard_cgc}
    u \mapsto u + \prolong[h_{\ell-1}] \left(\genericGalerkinOpDEC[][h_{\ell-1}]\right)^{-1} \prolong[h_{\ell-1}]^* r.
\end{equation}

Our primary objective, however, is to construct a multigrid scheme acting directly on the endomorphisms. Consider the fine-grid system $\genericOpDEC[][h_\ell] u = f$ for $f \in \genericDiscWDEC[][h_\ell]$. This is algebraically equivalent to the dual problem $\genericGalerkinOpDEC[][h_\ell] u = (\genericRieszDEC[h_\ell])^{-1} f$. Substituting $(\genericRieszDEC[h_\ell])^{-1} f$ for $g$ and using the relation $\genericGalerkinOpDEC[][h_\ell] = (\genericRieszDEC[h_\ell])^{-1}\genericOpDEC[][h_\ell]$ into \eqref{eq:standard_cgc} yields:
\begin{align*}
    u \mapsto& \; u + \prolong[h_{\ell-1}] \left(\genericGalerkinOpDEC[][h_{\ell-1}]\right)^{-1} \prolong[h_{\ell-1}]^* \left( (\genericRieszDEC[h_\ell])^{-1} f - (\genericRieszDEC[h_\ell])^{-1} \genericOpDEC[][h_\ell] u \right) \\
    =& \; u + \prolong[h_{\ell-1}] \left(\genericGalerkinOpDEC[][h_{\ell-1}]\right)^{-1} \prolong[h_{\ell-1}]^* (\genericRieszDEC[h_\ell])^{-1} \left(f - \genericOpDEC[][h_\ell] u \right).
\end{align*}
By using $\left(\genericGalerkinOpDEC[][h_{\ell-1}]\right)^{-1} = \left(\genericOpDEC[][h_{\ell-1}]\right)^{-1} \genericRieszDEC[h_{\ell-1}]$, we get:
\begin{equation}\label{eq:endomorphism_cgc}
    u \mapsto u + \prolong[h_{\ell-1}] \left(\genericOpDEC[][h_{\ell-1}]\right)^{-1} \left[ \genericRieszDEC[h_{\ell-1}] \prolong[h_{\ell-1}]^* (\genericRieszDEC[h_\ell])^{-1} \right] \left(f - \genericOpDEC[][h_\ell] u \right).
\end{equation}

We define
\begin{equation}\label{eq:modified_restriction}
    \restrict[h_\ell] := \genericRieszDEC[h_{\ell-1}] \prolong[h_{\ell-1}]^* (\genericRieszDEC[h_\ell])^{-1} : \genericDiscWDEC[][h_\ell] \to \genericDiscWDEC[][h_{\ell-1}],
\end{equation}
which will act like the restriction, but for endomorphisms instead of Galerkin operators.

\subsection{Transforming Smoothers}\label{subsection:smoothers}
The main challenge which remains to be tackled is the smoothers, as the operators in question are unsuitable for a smoother like Gauss-Seidel or Jacobi. 

Instead, we employ transforming smoothers, where the idea is to transform the system into one where we can apply a known smoothing procedure, and then transform back accordingly, see \cite{Wittum1990} for a more rigorous description of this topic.

In our case, we seek invertible left- and right-transformations $\transformL[h_\ell], \transformR[h_\ell]$ such that the resulting transformed operator
\[
    \transformedSystem[h_\ell] := \transformL[h_\ell] \genericOpDEC[][h_\ell] \transformR[h_\ell],
\] 
is amenable to classical point-smoothing techniques.

Since
\[
    \transformR[h_\ell] \transformedSystem[h_\ell]^{-1} \transformL[h_\ell] = \left(\genericOpDEC[][h_\ell]\right)^{-1},
\] 
we can translate an easily invertible $\transformedSystemSplit[h_\ell]$ (the inverse of which is supposed to mimic\footnote{A smoother need not necessarily be a good approximation of the inverse.} $\transformedSystem[h_\ell]^{-1}$) into an efficient iterative method for the original endomorphism $\genericOpDEC[][h_\ell]$: given a right-hand-side $f$ and an initial guess $u$, a step in the transforming smoother is explicitly defined as
\begin{equation}\label{eq:transforming_smoother_step_level}
    u \mapsto u + \transformR[h_\ell] \transformedSystemSplit[h_\ell]^{-1} \transformL[h_\ell] \left( f - \genericOpDEC[][h_\ell] u \right).
\end{equation}

For the problems at hand, we chose r-transformations, meaning $\transformL = \textnormal{Id}$, listed in \Cref{tab:transformations}, to transform the systems into a suitable shape. As the systems are block-triangular and the off-diagonal blocks are of lower-order, an appropriate Gauss-Seidel splitting\footnote{In this abstract setting, Gauss-Seidel can be interpreted as a successive subspace correction scheme induced by a decomposition of the underlying space into one-dimensional subspaces spanned by the elements of a specified basis.} of $\transformedSystem[h_l]$ may be used as a smoother\footnote{Alternative approaches could also be employed for the diagonal blocks in the transformed system, such as damped Jacobi, Chebyshev accelerated versions, or ILU \cite{Wittum1989} (or ICHOL, since these operators are self-adjoint and positive, and can therefore be symmetrized by multiplication with the diagonal mass-lumped mass matrix). However, in our experiments we restricted ourselves to Gauss–Seidel.}.

\begin{table}[ht]
    \centering
    \begin{tabular}{c|ccc}
        & Hodge-Dirac & Hodge-Laplace & Magnetostatics \\
        \toprule
        $\transformR$  & $\diracOpDEC \equiv \dh + \deltaML$ &
        $
            \begin{pmatrix}
                -I & \deltaML[k] \\
                \dh[k - 1] & I
            \end{pmatrix}
        $ &
        $
            \begin{pmatrix}
                0 & \deltaML[1] \\
                \dh[0] & I
            \end{pmatrix}
        $ \\
        $\transformedSystem$ & $\textnormal{blockdiag}\left(\dh[k - 1]\deltaML[k] + \deltaML[k + 1]\dh[k]\right)$ &
        $
            \begin{pmatrix}
                I + \deltaML[k]\dh[k - 1] & 0 \\
                -\dh[k - 1] & \dh[k - 1]\deltaML[k] + \deltaML[k + 1]\dh[k]
            \end{pmatrix}
        $ &
        $
            \begin{pmatrix}
                \deltaML[1]\dh[0] & \deltaML[1] \\
                0 & \dh[0]\deltaML[1] + \deltaML[2]\dh[1]
            \end{pmatrix}
        $  \\
    \end{tabular}
    \caption{$r$-transformations and transformed systems for the problems of interest.}
    \label{tab:transformations}
\end{table}

Assume now that we have chosen the standard basis of Whitney forms on all the finite element spaces, such that we can identify the operators with matrices, which we denote in bold. Then the smoothing iterations in \Cref{tab:smoothing_iterations} were used, where $\tril, \triu$ denote the lower and upper triangular part of the matrix, respectively.
\begin{table}[H]
    \centering
    \small
    \setlength{\tabcolsep}{4pt}
    \renewcommand{\arraystretch}{1.25}
    \begin{tabular}{@{}p{0.18\linewidth}p{0.78\linewidth}@{}}
        \hline
        \textbf{Problem} & \textbf{Smoothing iteration} \\
        \hline
        Hodge--Dirac &
        \(\displaystyle
        \begin{aligned}[t]
            \mathbf{u} \mapsto {}& \mathbf{u} + \diracMatOpDEC[][h_\ell] \transformedSystemSplitMat[h_\ell]^{-1}\left( \mathbf{f} - \diracMatOpDEC[][h_\ell] \mathbf{u} \right),\\
            \transformedSystemSplitMat[h_\ell] = {}& \tril\left[\textnormal{blockdiag}\left(\dhMat[k - 1]\deltaMLMat[k] + \deltaMLMat[k + 1]\dhMat[k]\right)\right].
        \end{aligned}
        \)\\
        \hline
        Hodge--Laplace &
        \(\displaystyle
        \begin{aligned}[t]
            \mathbf{u} \mapsto {}& \mathbf{u} +
            \begin{pmatrix}
                -\IdMat & \deltaMLMat[k] \\
                \dhMat[k - 1] & \IdMat
            \end{pmatrix}
            \transformedSystemSplitMat[h_\ell]^{-1}\left( \mathbf{f} - \laplaceMatOpDEC[k][h_\ell] \mathbf{u} \right),\\
            \transformedSystemSplitMat[h_\ell] = {}& \tril\left[
                \begin{pmatrix}
                    \IdMat + \deltaMLMat[k]\dhMat[k - 1] & \mathbf{0} \\
                    -\dhMat[k - 1] & \dhMat[k - 1]\deltaMLMat[k] + \deltaMLMat[k + 1]\dhMat[k]
                \end{pmatrix}
            \right].
        \end{aligned}
        \)\\
        \hline
        Magnetostatics &
        \(\displaystyle
        \begin{aligned}[t]
            \mathbf{u} \mapsto {}& \mathbf{u} +
            \begin{pmatrix}
                \mathbf{0} & \deltaMLMat[1] \\
                \dhMat[0] & \IdMat
            \end{pmatrix}
            \transformedSystemSplitMat[h_\ell]^{-1}\left( \mathbf{f} - \magMatOpDEC[][h_\ell] \mathbf{u} \right),\\
            \transformedSystemSplitMat[h_\ell] = {}& \triu\left[
                \begin{pmatrix}
                    \deltaMLMat[1]\dhMat[0] & \deltaMLMat[1] \\
                    \mathbf{0} & \dhMat[0]\deltaMLMat[1] + \deltaMLMat[2]\dhMat[1]
                \end{pmatrix}
            \right].
        \end{aligned}
        \)\\
        \hline
    \end{tabular}
    \caption{Smoothing iterations used for the problems of interest.}
    \label{tab:smoothing_iterations}
\end{table}

\subsection{Multigrid Algorithm}\label{subsection:mg_algorithm}
For completeness' sake, we list the multigrid for $\genericOpDEC$ algorithm here, see \cite{Hackbusch1985} for more details.

The recursive multigrid procedure for the endomorphism $\genericOpDEC[][h_\ell]$ is summarized in \Cref{alg:multigrid}.

The algorithm is governed by the presmoothing and postsmoothing iteration counts ($\nu_1$ and $\nu_2$), and the cycle index $\gamma$, where $\gamma = 1$ corresponds to a $\text{V}$- and $\gamma = 2$ to a $\text{W}$-cycle.

\begin{algorithm}[ht]
\caption{Multigrid Cycle $\text{MG}_{h_\ell}(u, f)$}
\label{alg:multigrid}
\textbf{Parameters:} Presmoothing steps $\nu_1 \ge 0$, postsmoothing steps $\nu_2 \ge 0$, cycle index $\gamma \in \{1, 2\}$.
\begin{algorithmic}[1]
\Require Level index $\ell \ge 1$, initial guess $u \in \genericDiscWDEC[][h_\ell]$, right-hand side $f \in \genericDiscWDEC[][h_\ell]$.
\Ensure Updated solution approximation $u^{\text{new}} \in \genericDiscWDEC[][h_\ell]$.
\Statex
\State \textbf{Presmoothing}
\State $v \gets u$
\For{$k = 1, \dots, \nu_1$}
    \State $v \gets v + \transformR[h_\ell] \transformedSystemSplit[h_\ell]^{-1} \bigl(f - \genericOpDEC[][h_\ell] v\bigr)$
\EndFor
\Statex
\State \textbf{Coarse-Grid Correction}
\State $r_{h_{\ell-1}} \gets \restrict[h_\ell] \bigl(f - \genericOpDEC[][h_\ell] v\bigr)$ \Comment{Restrict smooth residual}
\If{$\ell = 1$}
    \State $e_{h_0} \gets \bigl(\genericOpDEC[][h_0]\bigr)^{-1} r_{h_0}$ \Comment{Exact solve on coarsest grid}
\Else
    \State $e_{h_{\ell-1}} \gets \text{MG}_{h_{\ell-1}}^{\gamma}(0, r_{h_{\ell-1}})$ \Comment{Recursive coarse-grid solve}
\EndIf
\State $v \gets v + \prolong[h_{\ell-1}] e_{h_{\ell-1}}$ \Comment{Prolongate and apply correction}
\Statex
\State \textbf{Postsmoothing}
\For{$k = 1, \dots, \nu_2$}
    \State $v \gets v + \transformR[h_\ell] \transformedSystemSplit[h_\ell]^{-1} \bigl(f - \genericOpDEC[][h_\ell] v\bigr)$
\EndFor
\Statex
\State \Return $u^{\text{new}} \gets v$
\end{algorithmic}
\end{algorithm}

\subsection{Error Operator}
Here, we briefly introduce the error operator and describe how it relates to the convergence of the method. For more details, see \cite{Hackbusch1985}.

Because the multigrid algorithm (\Cref{alg:multigrid}) constitutes a linear stationary iterative process, its convergence can be characterized by a single error propagation operator (or iteration matrix), denoted by $\mathsf{E}_{h_\ell}$. Let $u$ be the exact solution to the discrete system $\genericOpDEC[][h_\ell] u = f$. For an approximation $u^{(m)}$ obtained after $m$ multigrid cycles, the error is defined as $e^{(m)} = u - u^{(m)}$. Due to the linearity of the method, this error evolves according to the relation $e^{(m)} = \mathsf{E}_{h_\ell}^m e^{(0)}$, where $e^{(0)}$ is the error of the initial guess. 

Following standard multigrid theory (cf.\ \cite[Lemma~7.1.4]{Hackbusch1985}), the operator $\mathsf{E}_{h_\ell}$ is defined recursively across the grid hierarchy. On the coarsest level, the exact solve eliminates the error entirely, yielding the base case $\mathsf{E}_{h_0} = 0$. For any finer level $\ell \ge 1$, the error propagation operator is given by:
\begin{equation}\label{eq:mg_error_operator}
    \mathsf{E}_{h_\ell} = 
    \underbrace{
        \left[ \Id - \transformR[h_\ell] \transformedSystemSplit[h_\ell]^{-1} \genericOpDEC[][h_\ell] \right]^{\nu_2}
    }_{\text{post-smoothing}}
    \underbrace{
        \left[ \Id - \prolong[h_{\ell-1}] \left( \Id - \mathsf{E}_{h_{\ell-1}}^\gamma \right) \genericOpDEC[-1][h_{\ell-1}] \restrict[h_\ell] \genericOpDEC[][h_\ell] \right]
    }_{\text{coarse-grid correction}}
    \underbrace{
        \left[ \Id - \transformR[h_\ell] \transformedSystemSplit[h_\ell]^{-1} \genericOpDEC[][h_\ell] \right]^{\nu_1}
    }_{\text{pre-smoothing}}.
\end{equation}

The long-term convergence behavior of the algorithm is strictly governed by the spectral radius $\rho(\mathsf{E}_{h_\ell})$. The method achieves optimal, $h$-uniform asymptotic convergence if this radius is bounded away from unity independently of the mesh size, i.e., $\rho(\mathsf{E}_{h_\ell}) \le \rho_{\text{asymp}} < 1$.

\section{Numerical Experiments}\label{section:results}
\subsection{Setup}
To assess the effectiveness of the methods, we implemented the methods in \texttt{C++} using \texttt{MFEM} \cite{mfem} and computed the spectral radii of \eqref{eq:mg_error_operator} for the problems in both 2D and 3D for the V- and W-cycles with one pre- and post-smoothening step using \cite{Spectra}. The code can be found at \url{https://github.com/rdabetic/mg_feec_transforming}.

For the mass-lumping, we investigated the following options:
\begin{enumerate}
    \item Scaled identity: $\rieszDEC[k] = h^{2k - n}\Id$, where $n$ is the dimension of $\Omega\subset\R^n$. By a straightforward scaling argument and assuming a shape-regular family of meshes, one expects $\rieszFE[k] \sim h^{2k - n} \Id$.
    \item Row sums: Construction of a diagonal mass matrix by taking, for each row, the $\ell^1$-norm of the corresponding row of the consistent mass matrix.
    \item Barycentric duals: Assembly of a non-orthogonal dual mesh by connecting the barycenters of the $k$-simplices to those of the $(k-1)$-simplices contained in each $k$-simplex, and then using the dual-to-primal volume ratios as the mass entries. In 2D, this yields a barycentric dual mesh with bent (kinked) dual edges, where the kinks occur at the intersections of dual and primal edges, i.e., at the midpoints of the primal edges. This construction is inspired by discrete exterior calculus (DEC), see \cite{GUP25,Dabetic2025}\footnote{If the mesh is \emph{well-centered} in 2D, then uniform refinement generates a shape-regular sequence of nested, well-centered meshes. In this setting, one can formulate discrete exterior calculus (DEC) using circumcentric dual meshes, which yields a \emph{consistent} discretization, provided that the exact solution satisfies appropriate regularity assumptions, see \cite{GUP25,Dabetic2025}.}.
\end{enumerate}

The tests were conducted by taking an initial coarse mesh and refining it using uniform refinement, as implemented in \cite{mfem}. We also added a few meshes with non-trivial topologies to see how the method behaves in that case as well.

For non-trivial topologies, by \Cref{lemma:dec_hodge}, we expect the rank deficiency to be determined by the Betti numbers, as the dimension of the mass-lumped harmonic forms equals the dimension of the FEEC harmonic forms, which in turn is determined by the (relative) Betti numbers of the domain, see \cite{Arnold2018} for more. Therefore we should expect to find a fixed number of eigenvectors of $\mathsf{E}_{h}$ with magnitude $\geq 1$, determined solely by the mesh topology.

For a concrete example: in the case of the Dirac operator, we have exactly $\sum_k b_k$ harmonic forms, where $b_k$ is the $k$-th Betti number. For the other problems, the $k$ corresponds to the dimension of the kernel of the operators, which is also given through the Betti numbers, see \cite{Arnold2018}.

In order to estimate the typical iteration counts for solving the \emph{consistent} FEEC linear systems with an iterative method, we applied (right-) preconditioned \texttt{GMRES}\footnote{\texttt{MFEM}\cite{mfem} implements this in the function \texttt{FGMRES}.} with restart parameter \texttt{20} and relative tolerance \texttt{1e-6}, using the corresponding multigrid cycle (V- or W-cycle with one smoothing step) built for the mass-lumped operator as a preconditioner. For each configuration, we generated eight random right-hand-sides\footnote{We generate random solutions, then apply the operator to get a RHS in the range of the operator.}, recorded the \texttt{GMRES} iteration counts, and report their average. For non-trivial topologies, where the system is only solvable up to harmonic forms, we terminate the iteration when the residual is sufficiently small, i.e.\ the solution may contain non-trivial harmonic forms.

The meshes used, together with the names of the meshes, can be found in \Cref{fig:mesh2d} for the 2D and \Cref{fig:mesh3d} for the 3D problems.

Furthermore, for the magnetostatics problem, we also tried only post-smoothening with a varying amount of steps.

\subsection{Results}
The numerical results are summarized in \Cref{fig:dirac_2d_bary,fig:dirac_2d_bary_w,fig:dirac_2d_mass,fig:dirac_2d_mass_w,fig:dirac_2d_scaledid,fig:dirac_2d_scaledid_w} for the Dirac operator, in \Cref{fig:laplace_2d_bary,fig:laplace_2d_bary_w,fig:laplace_2d_mass,fig:laplace_2d_mass_w,fig:laplace_2d_scaledid,fig:laplace_2d_scaledid_w} for the Hodge-Laplacian on 1-forms, and in \Cref{fig:mag_2d_bary,fig:mag_2d_bary_w,fig:mag_2d_mass,fig:mag_2d_mass_w,fig:mag_2d_scaledid,fig:mag_2d_scaledid_w} for the magnetostatics problem.

We analyze the $n_\textnormal{ev}$-th dominant eigenvalue of the error propagation matrix, where $n_\textnormal{ev}$ denotes the number of harmonic forms associated with the continuous problem. In the presence of non-trivial harmonic forms, one expects the error propagation matrix to admit non-convergent eigenmodes corresponding to these harmonic components. This expectation is confirmed by the numerical experiments: we observe eigenvectors whose eigenvalues have unit modulus. Furthermore, the number of such eigenvectors coincides with the rank deficiency of the matrix.

The methods exhibit robust performance in both two and three spatial dimensions, with a consistently improved behavior in the two-dimensional setting. Unsurprisingly, the W-cycle generally yields more robust convergence compared to the V-cycle.

Across all experiments, the different mass-lumping strategies do not seem to lead to any pronounced qualitative differences in the overall convergence behavior.

In both two and three dimensions, the method appears capable of handling non-trivial harmonic forms\footnote{With a caveat: recall that we did not impose that the solution be orthogonal on the harmonic forms.}.

For the magnetostatics problem, cf.~\Cref{fig:mag_2d_bary,fig:mag_2d_bary_w,fig:mag_2d_mass,fig:mag_2d_mass_w,fig:mag_2d_scaledid,fig:mag_2d_scaledid_w}, we see that one post-smoothening step appears to be sufficient for robust convergence. Also, an increased number of post-smoothing steps leads to noticeably improved convergence rates, which is unsurprising.

\begin{figure}[ht]
    \centering

    % 3 per row
    \begin{subfigure}[b]{0.32\textwidth}
        \centering
        \includegraphics[width=\textwidth]{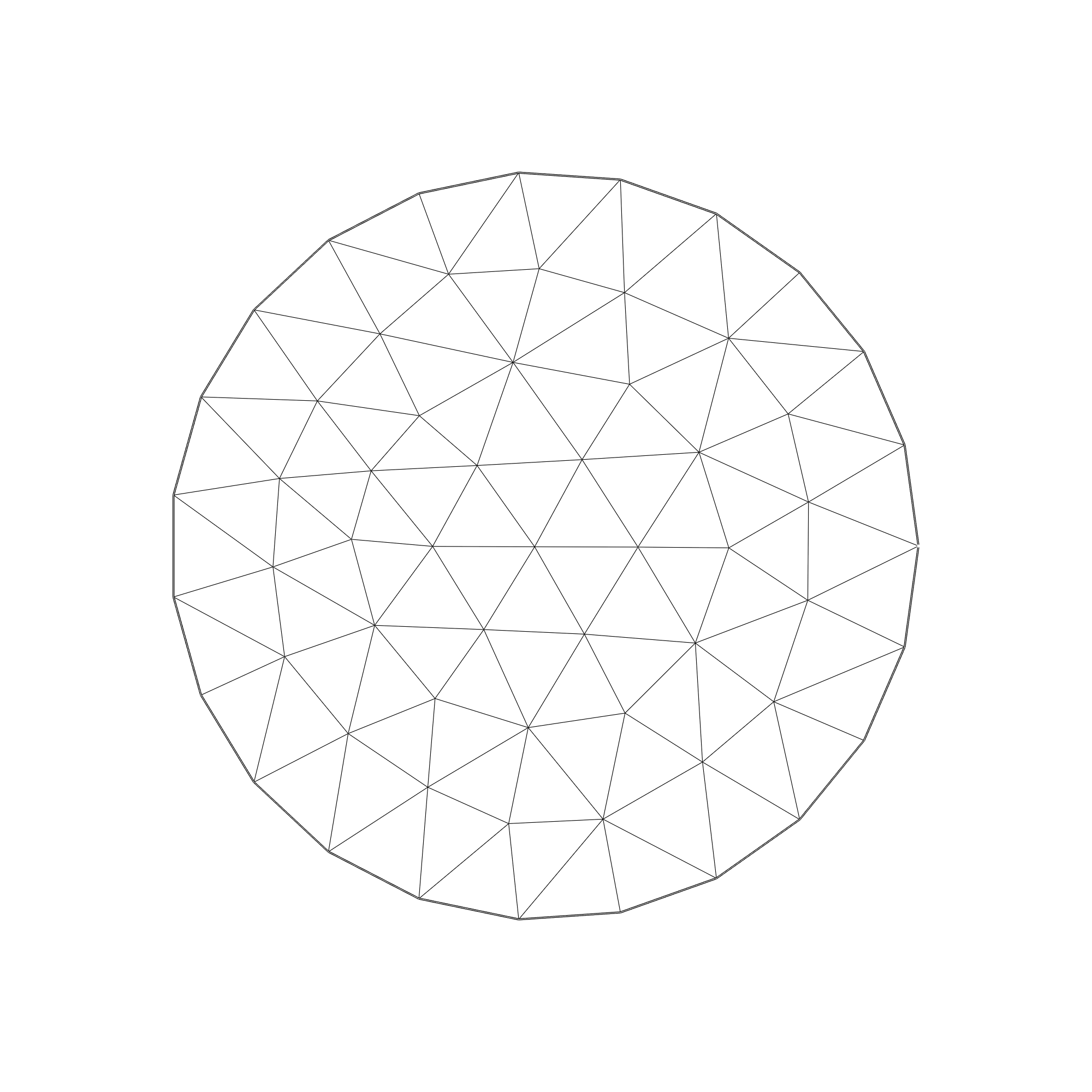}
        \caption{circle}
    \end{subfigure}\hfill
    \begin{subfigure}[b]{0.32\textwidth}
        \centering
        \includegraphics[width=\textwidth]{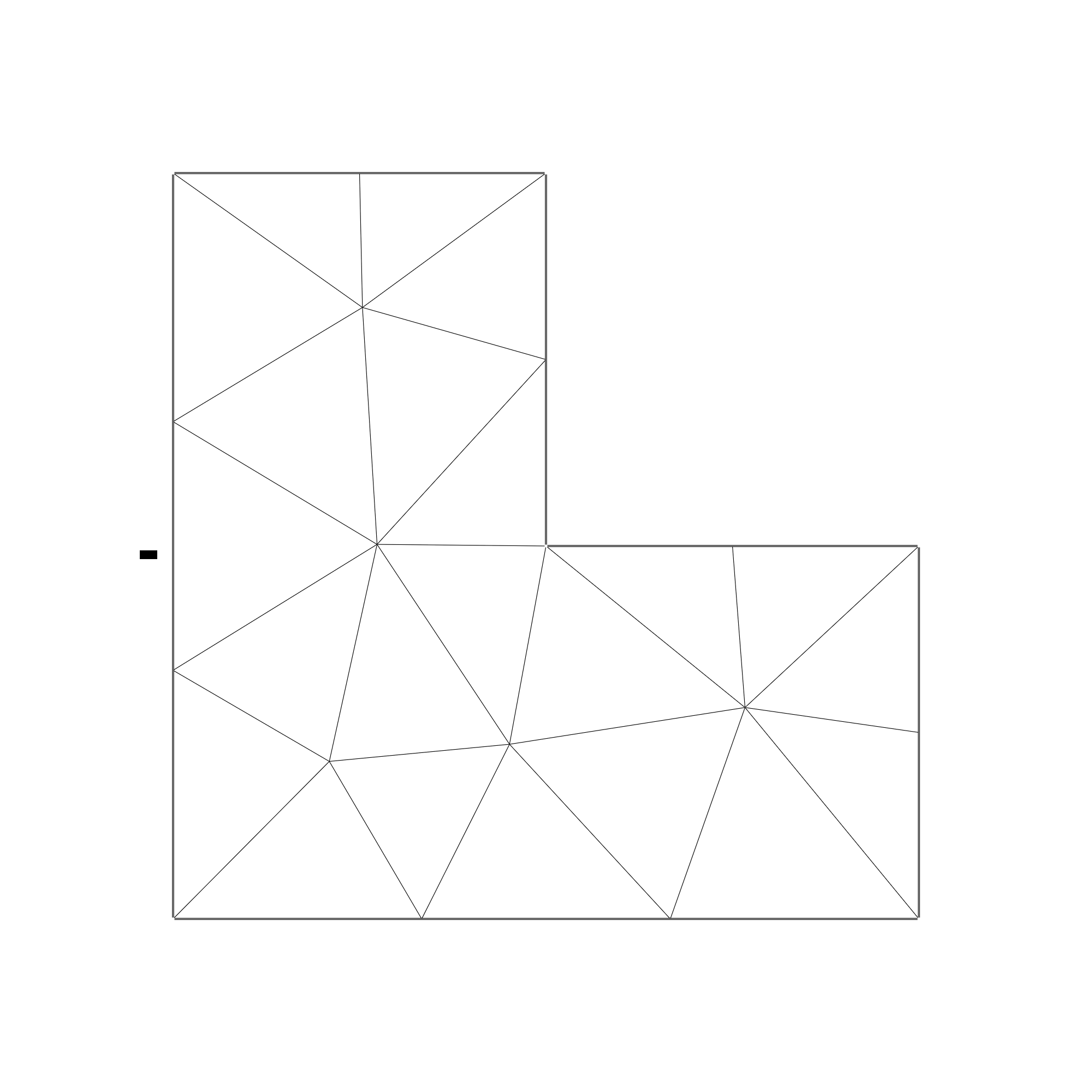}
        \caption{L}
    \end{subfigure}\hfill
    \begin{subfigure}[b]{0.32\textwidth}
        \centering
        \includegraphics[width=\textwidth]{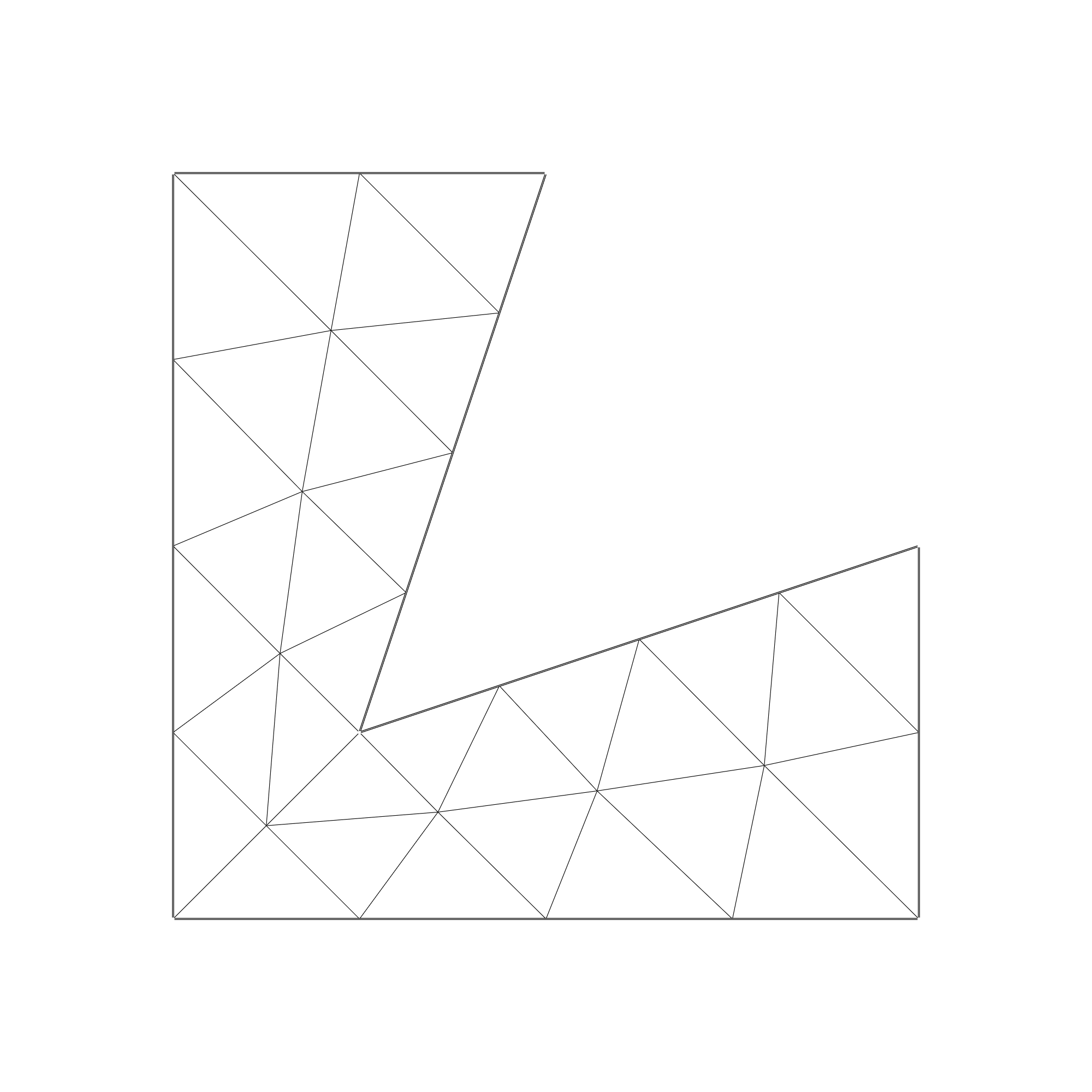}
        \caption{reentrant}
    \end{subfigure}

    \vspace{0.5em}
    \begin{subfigure}[b]{0.32\textwidth}
        \centering
        \includegraphics[width=\textwidth]{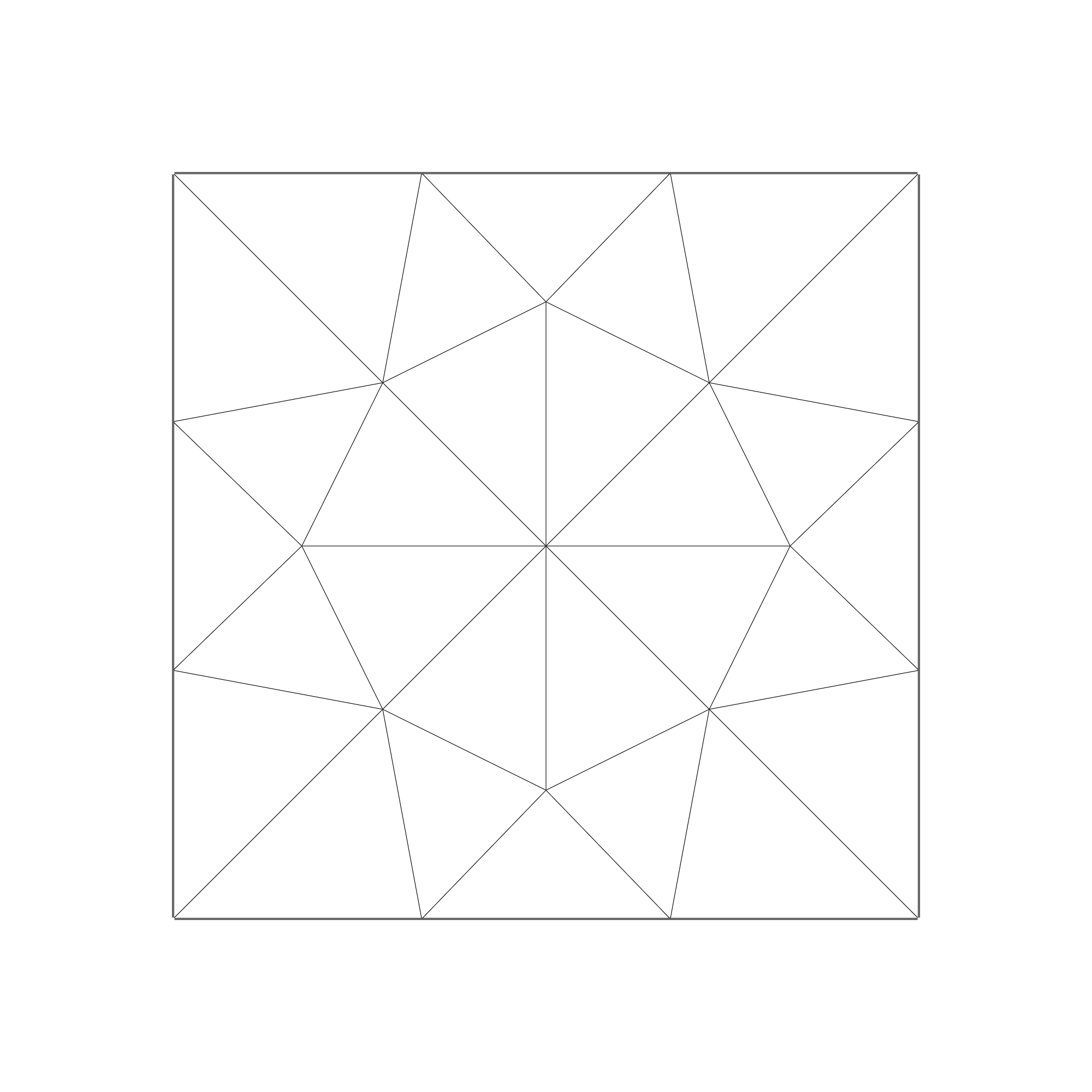}
        \caption{square\_centered}
    \end{subfigure}\hfill
    \begin{subfigure}[b]{0.32\textwidth}
        \centering
        \includegraphics[width=\textwidth]{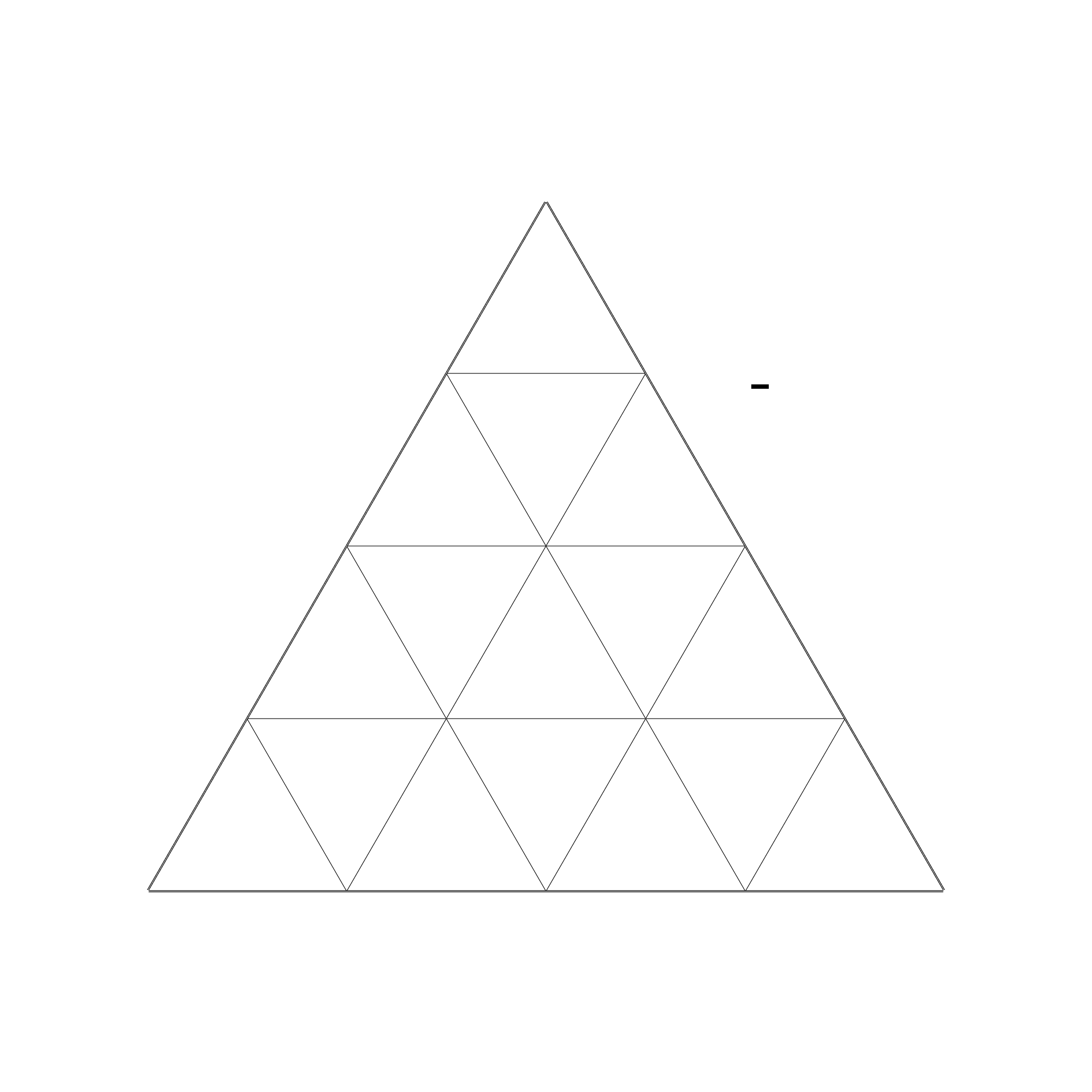}
        \caption{tria}
    \end{subfigure}\hfill
    \begin{subfigure}[b]{0.32\textwidth}
        \centering
        \includegraphics[width=\textwidth]{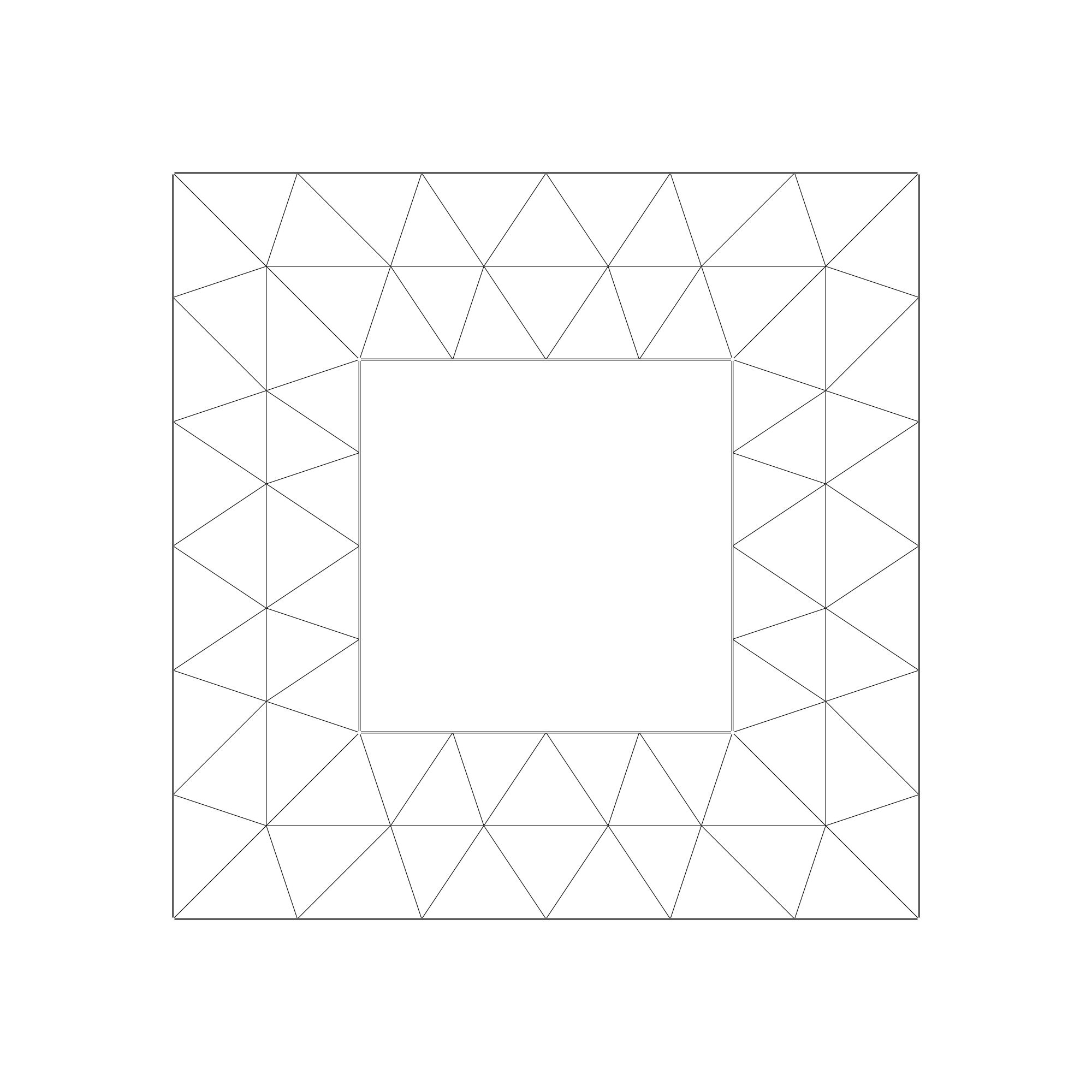}
        \caption{square\_hole}
    \end{subfigure}
\end{figure}
\begin{figure}[ht]
    \ContinuedFloat
    \centering

    \begin{subfigure}[b]{0.32\textwidth}
        \centering
        \includegraphics[width=\textwidth]{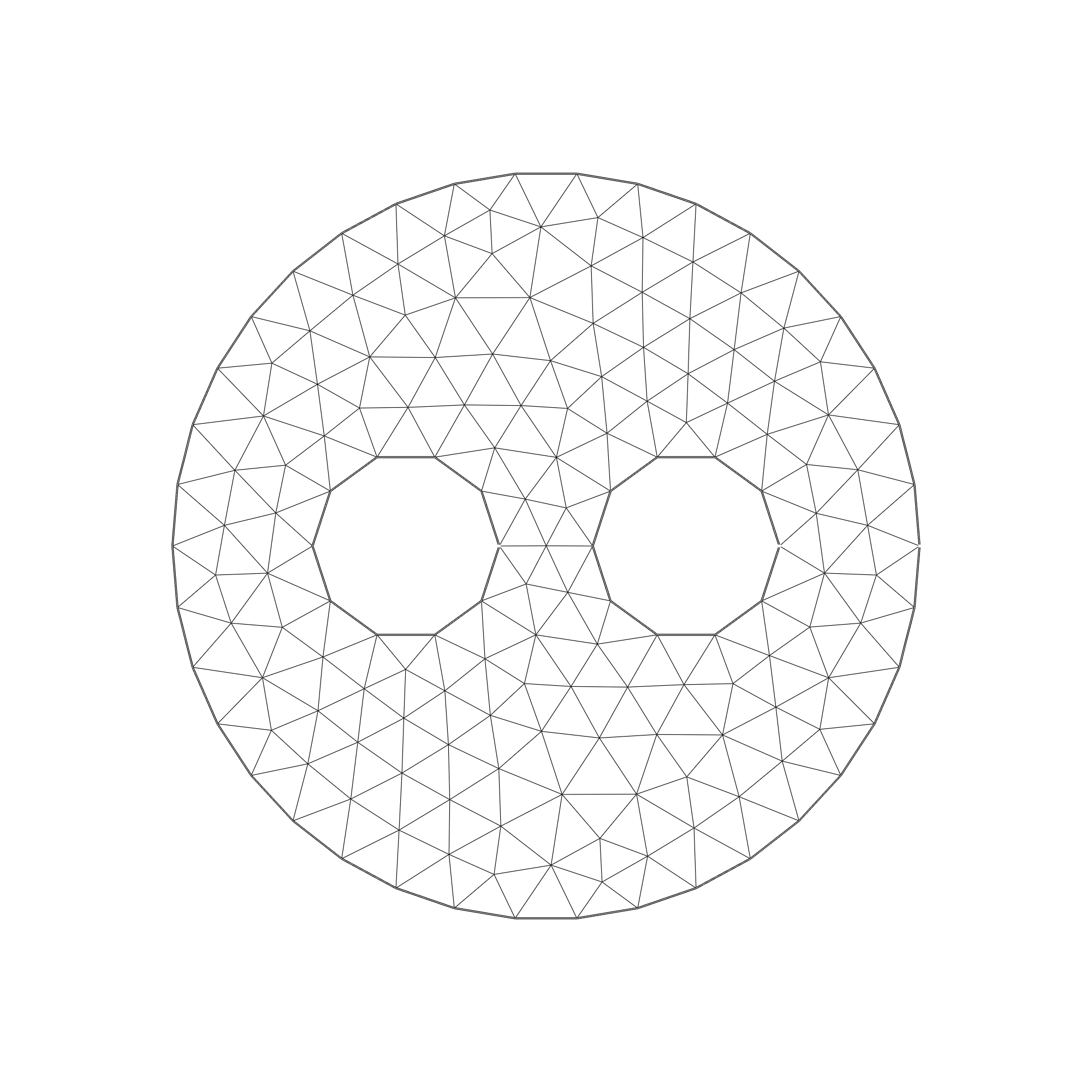}
        \caption{two\_holes}
    \end{subfigure}

    \caption{Coarse meshes used in the 2D tests.}
    \label{fig:mesh2d}
\end{figure}

\begin{figure}[ht]
    \centering

    % 3 per row (except hole_void)
    \begin{subfigure}[b]{0.32\textwidth}
        \centering
        \includegraphics[width=\textwidth]{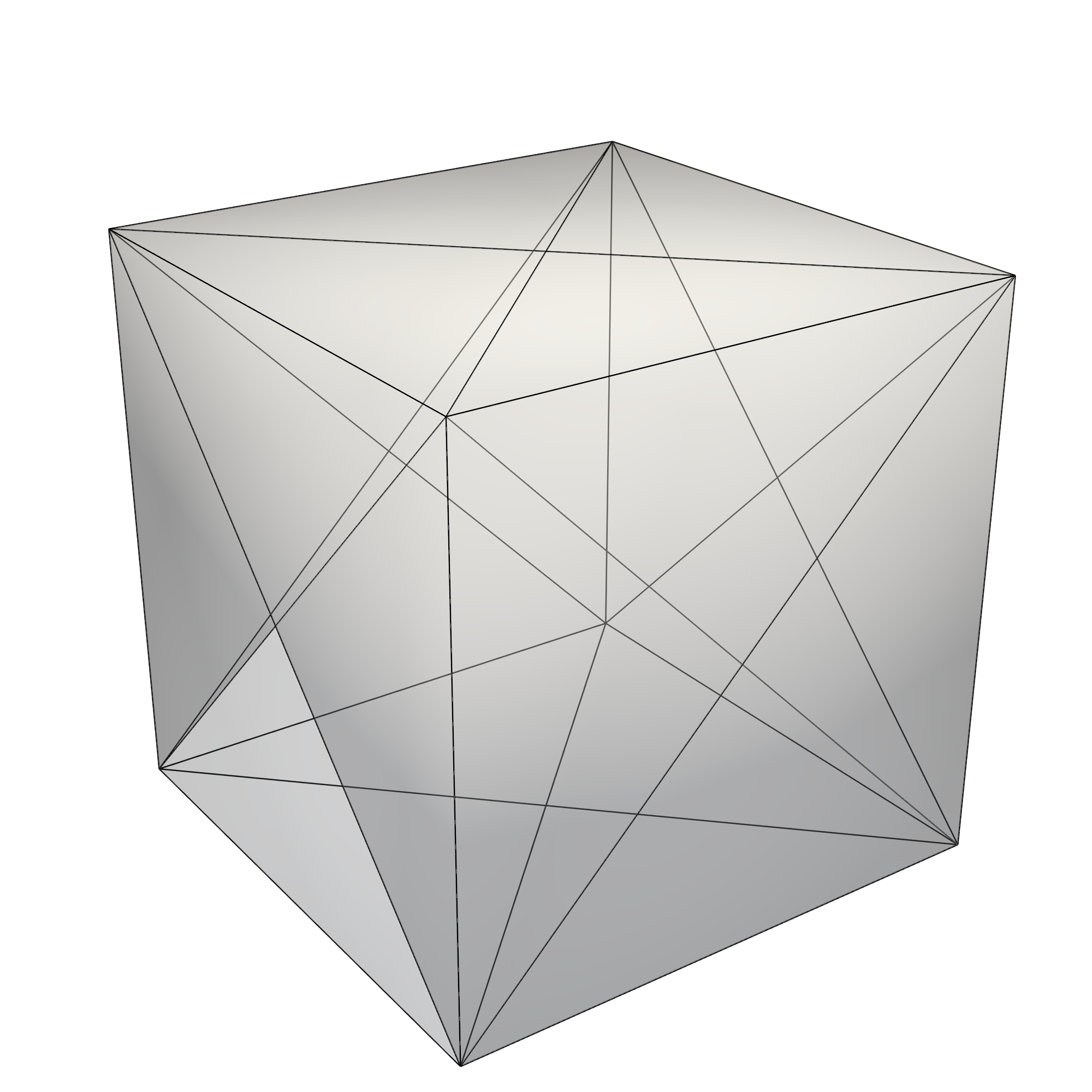}
        \caption{cube}
    \end{subfigure}\hfill
    \begin{subfigure}[b]{0.32\textwidth}
        \centering
        \includegraphics[width=\textwidth]{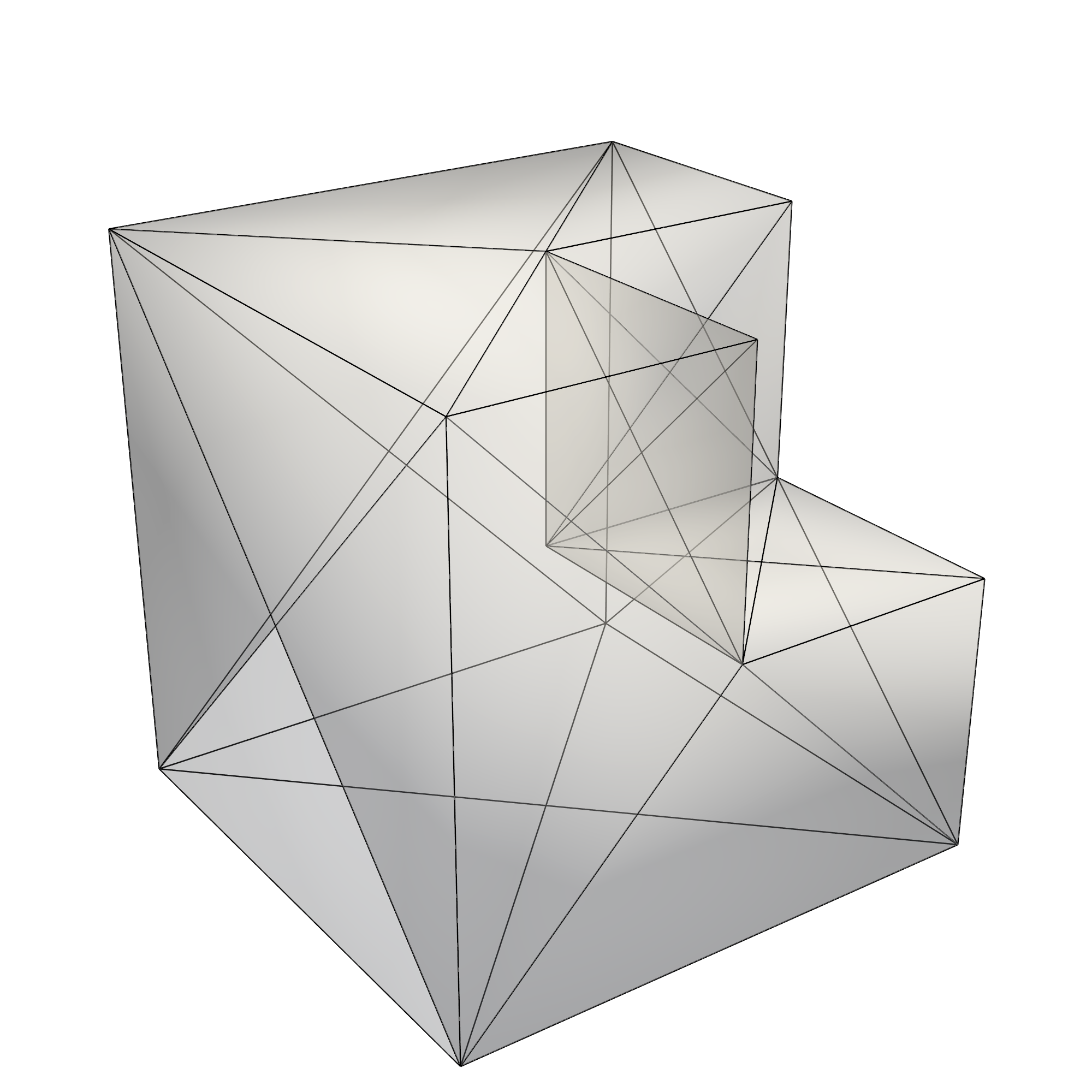}
        \caption{corner}
    \end{subfigure}\hfill
    \begin{subfigure}[b]{0.32\textwidth}
        \centering
        \includegraphics[width=\textwidth]{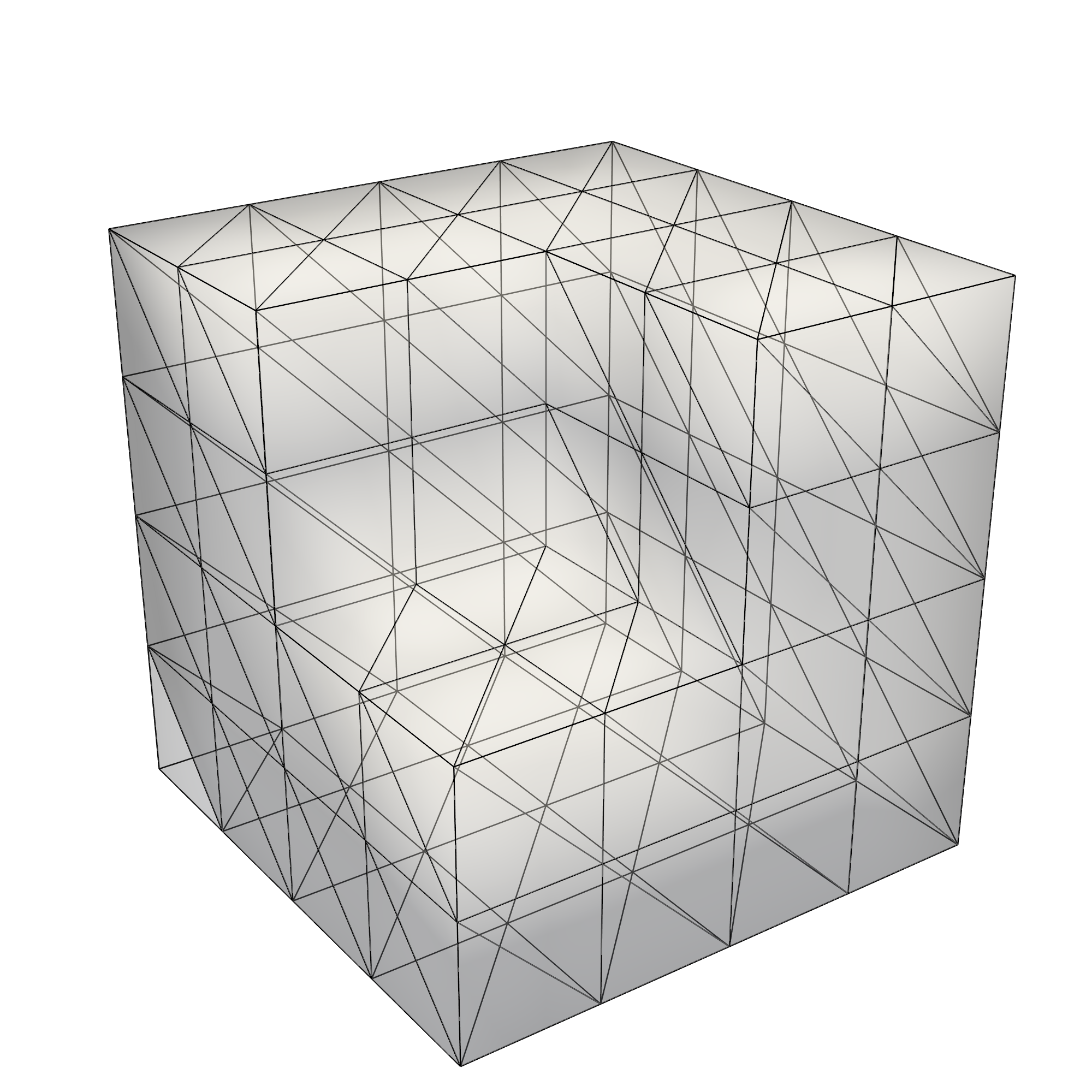}
        \caption{corner\_structured}
    \end{subfigure}

    \vspace{0.5em}
    % 2 on this row (fills the line nicer)
    \begin{subfigure}[b]{0.49\textwidth}
        \centering
        \includegraphics[width=\textwidth]{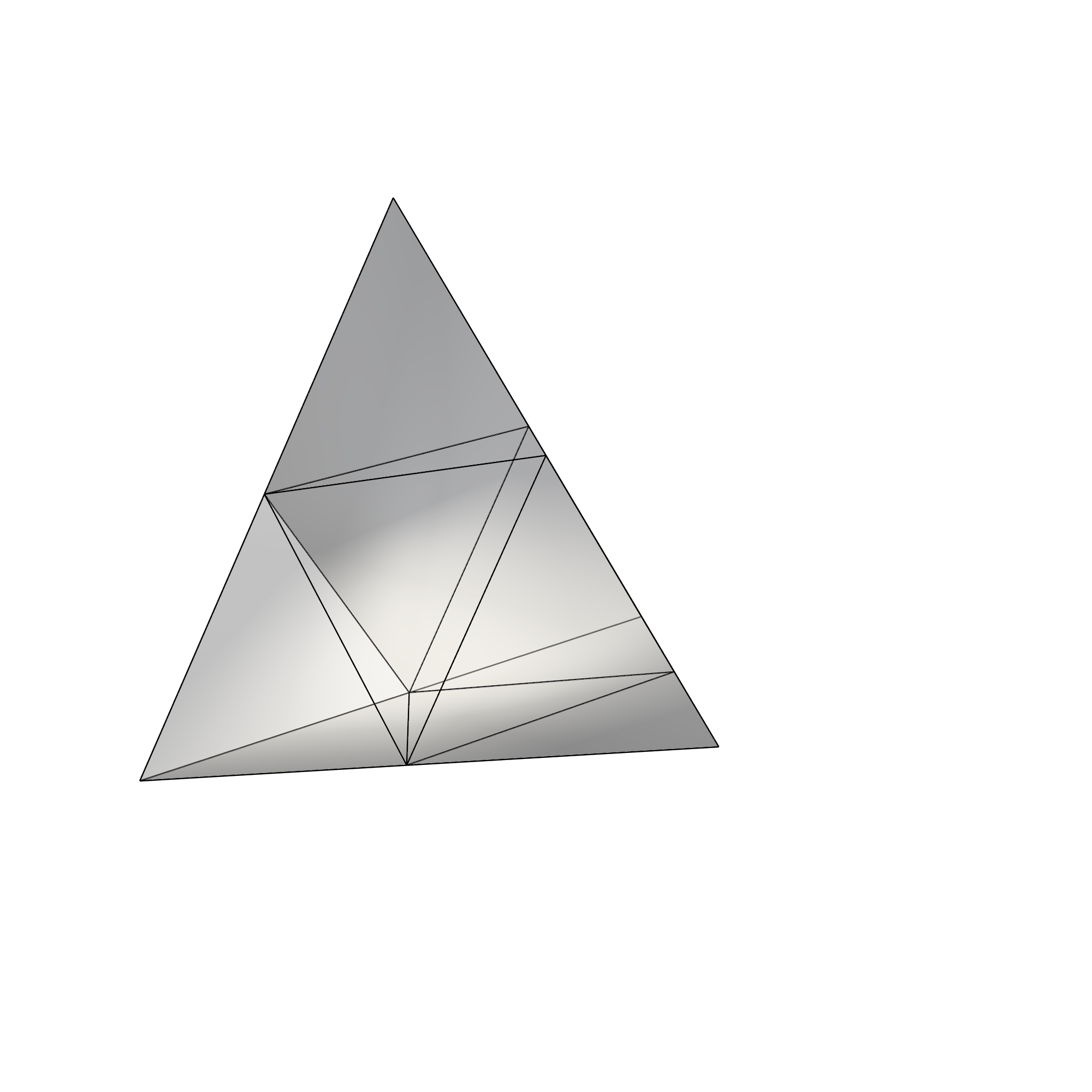}
        \caption{tetra}
    \end{subfigure}\hfill
    \begin{subfigure}[b]{0.49\textwidth}
        \centering
        \includegraphics[width=\textwidth]{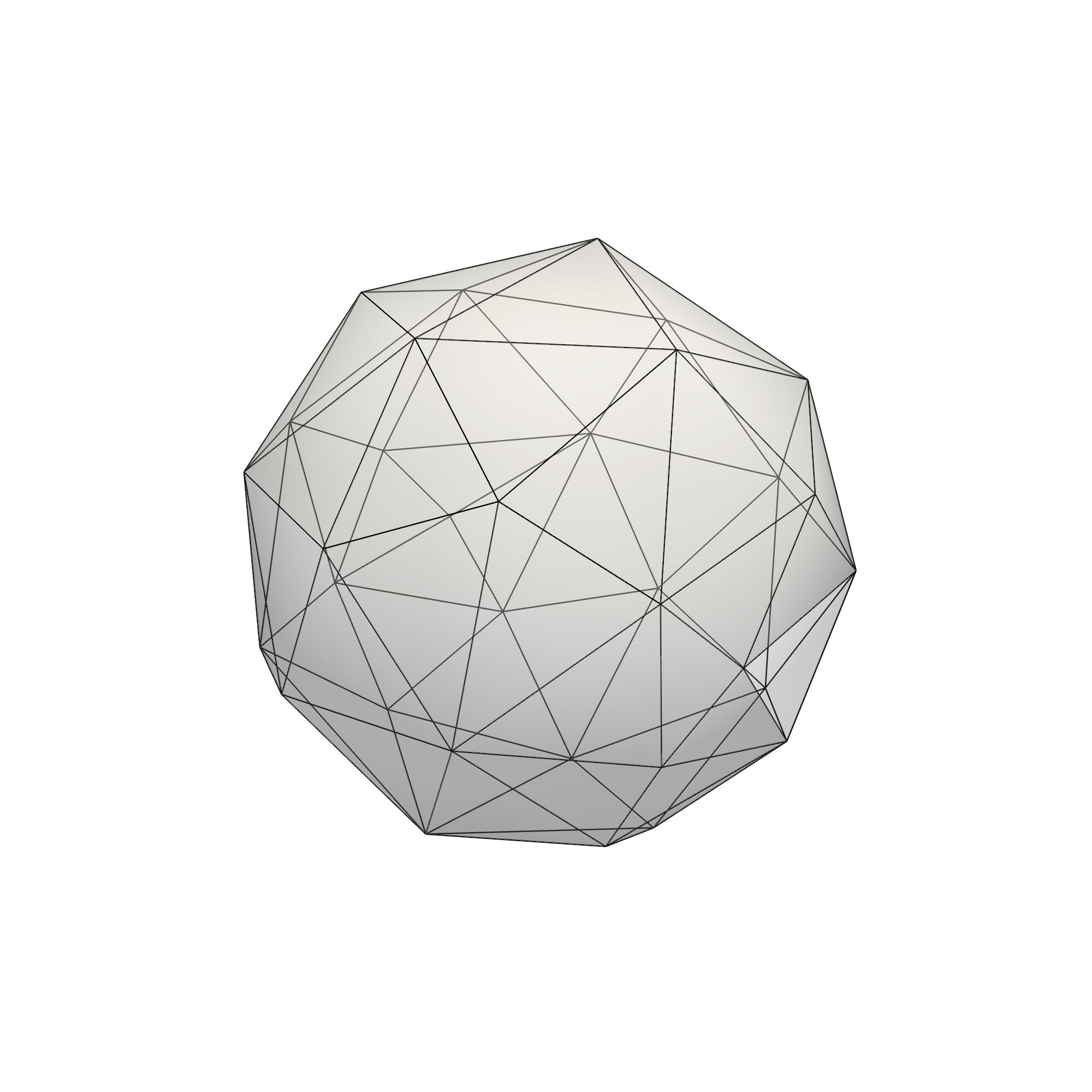}
        \caption{ball}
    \end{subfigure}

    \vspace{0.5em}
    \makebox[\textwidth][c]{%
        \begin{subfigure}[b]{0.49\textwidth}
            \centering
            \includegraphics[width=\textwidth]{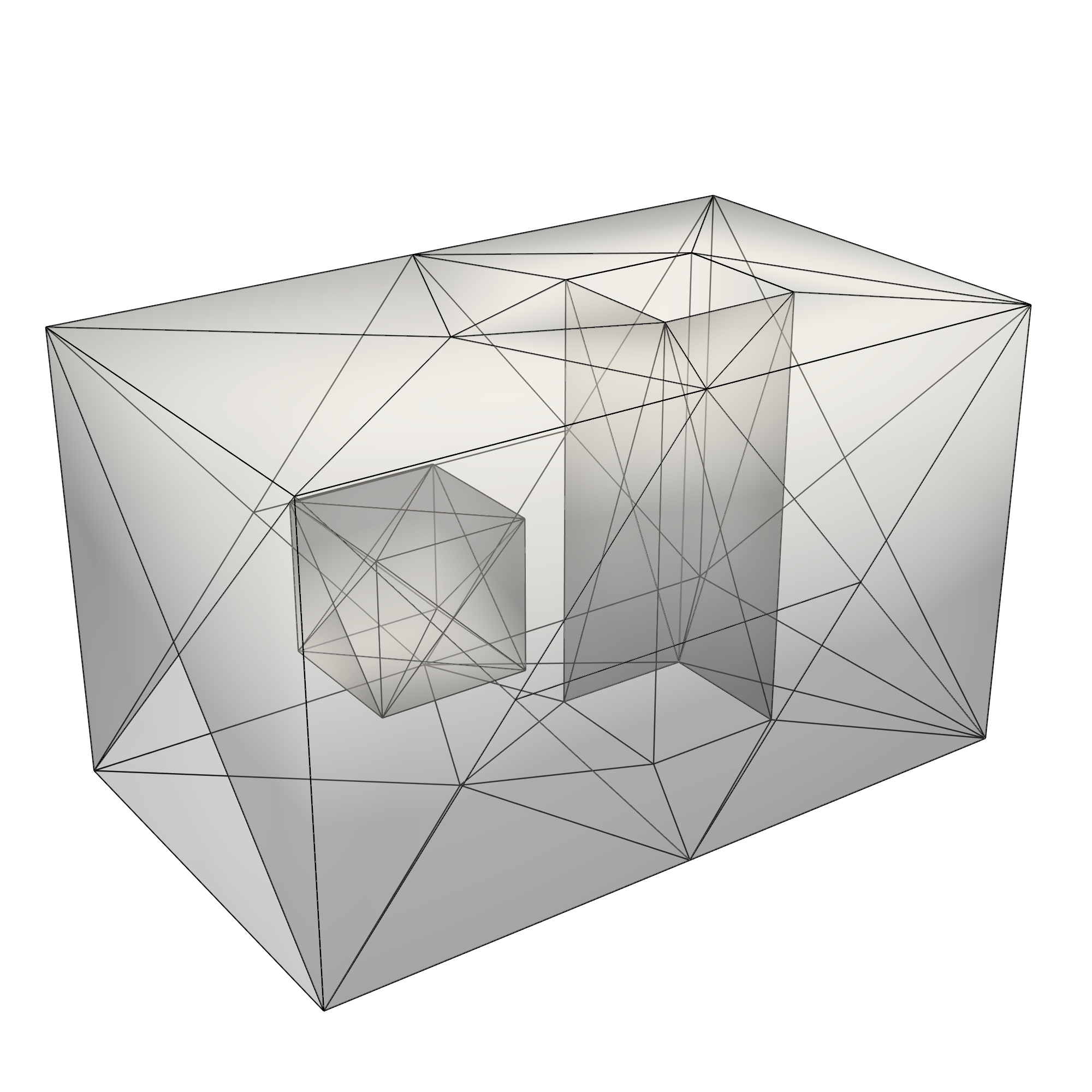}
            \caption{hole\_void}
        \end{subfigure}%
    }

    \caption{Coarse meshes used in the 3D tests.}
    \label{fig:mesh3d}
\end{figure}

\begin{figure}[p]
    \centerfloat
    \centering
    \fullplot{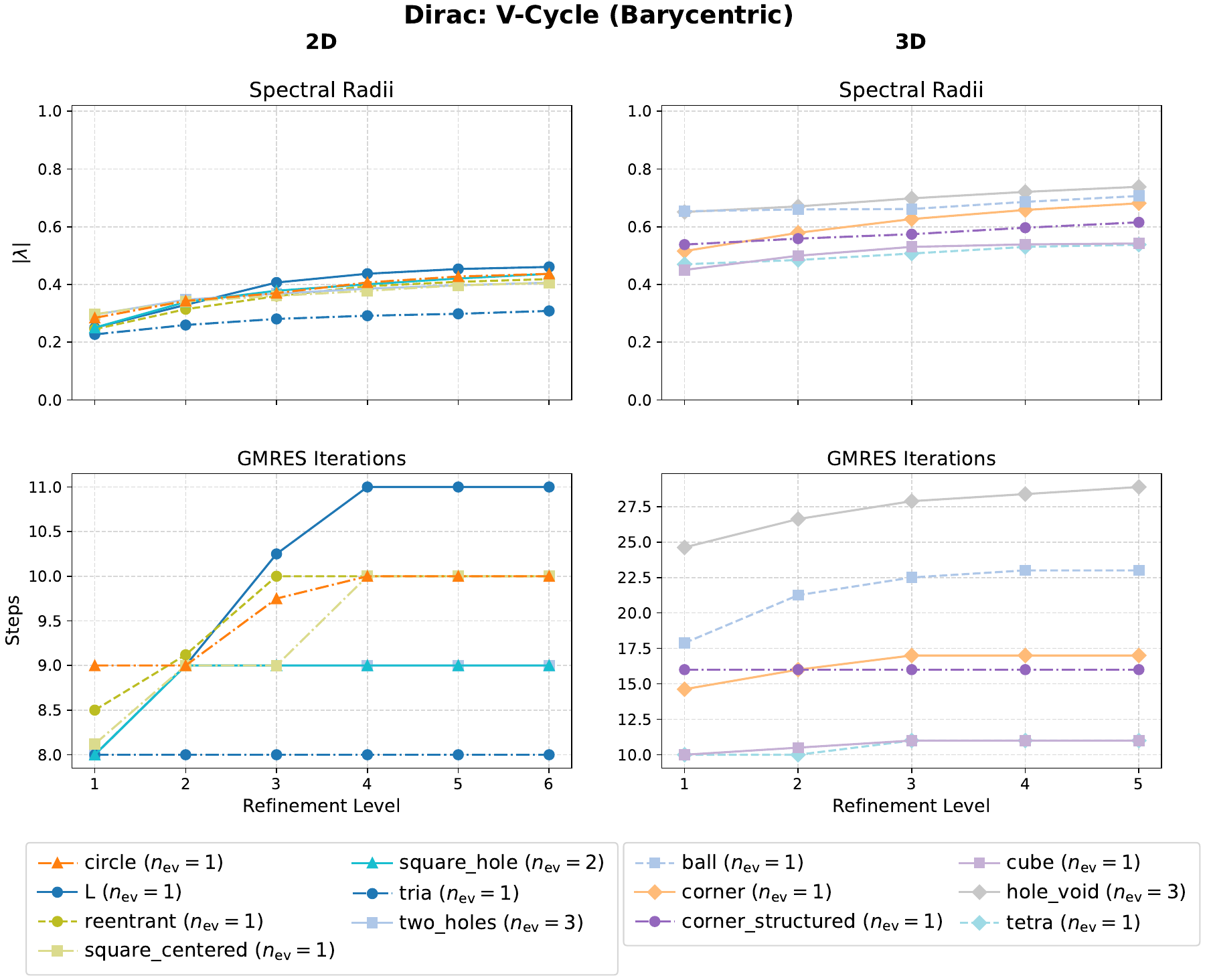}

    \caption{Convergence of the multigrid method for the Dirac operator in 2D and 3D with barycentric duals (V-cycle).}
    \label{fig:dirac_2d_bary}
\end{figure}

\begin{figure}[p]
    \centerfloat
    \centering
    \fullplot{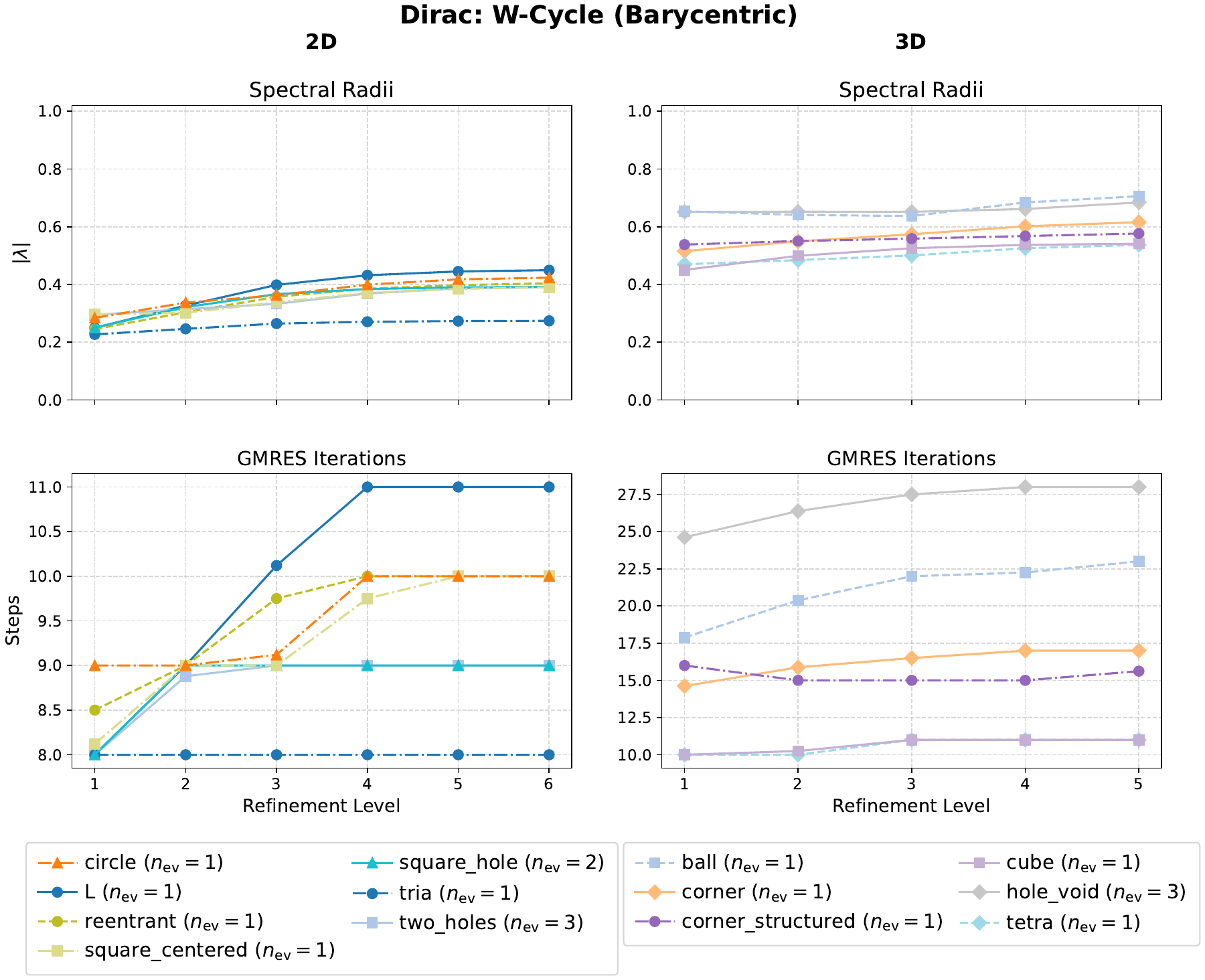}

    \caption{Convergence of the multigrid method for the Dirac operator in 2D and 3D with barycentric duals (W-cycle).}
    \label{fig:dirac_2d_bary_w}
\end{figure}

\begin{figure}[p]
    \centerfloat
    \centering
    \fullplot{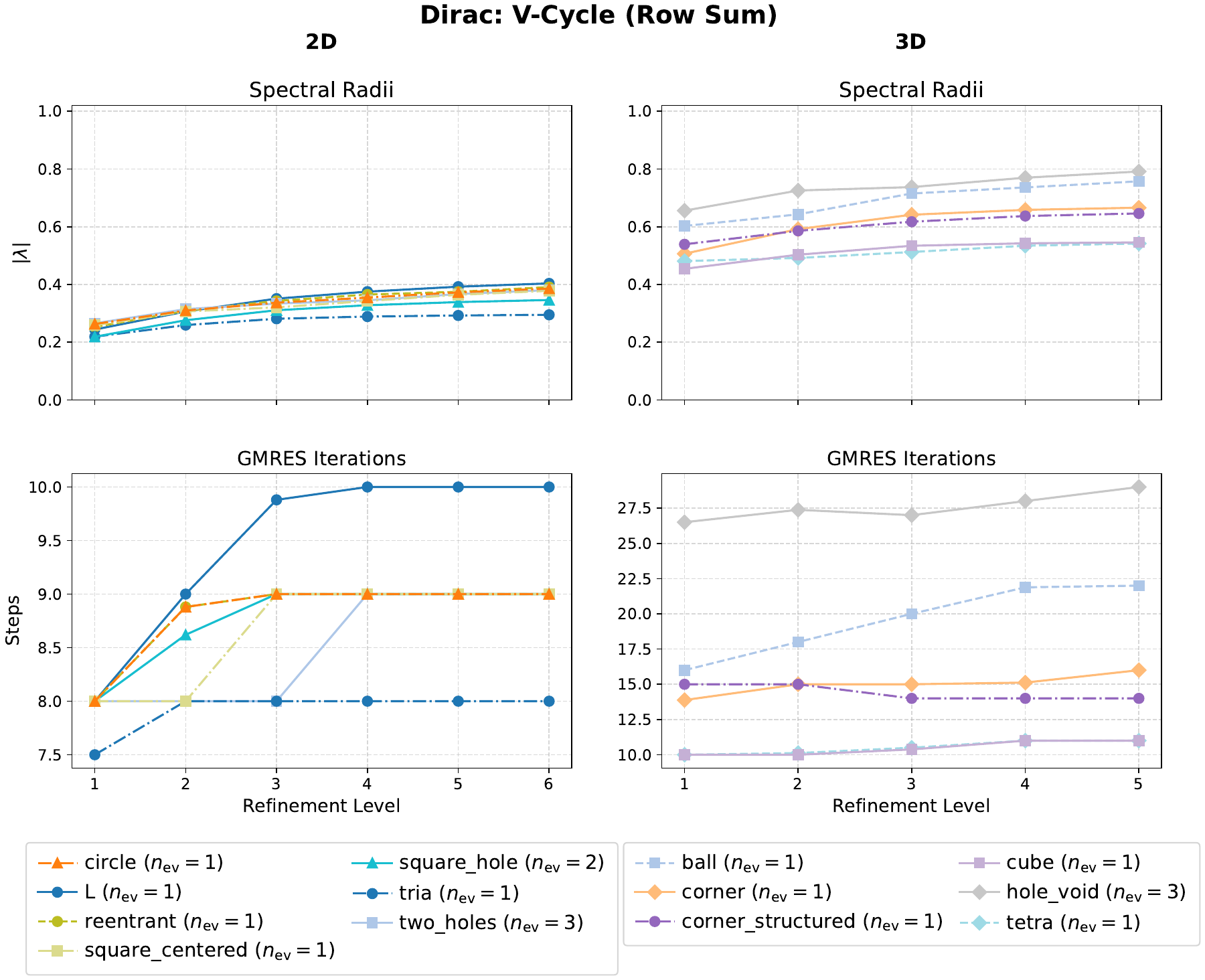}

    \caption{Convergence of the multigrid method for the Dirac operator in 2D and 3D with row-sum mass-lumping (V-cycle).}
    \label{fig:dirac_2d_mass}
\end{figure}

\begin{figure}[p]
    \centerfloat
    \centering
    \fullplot{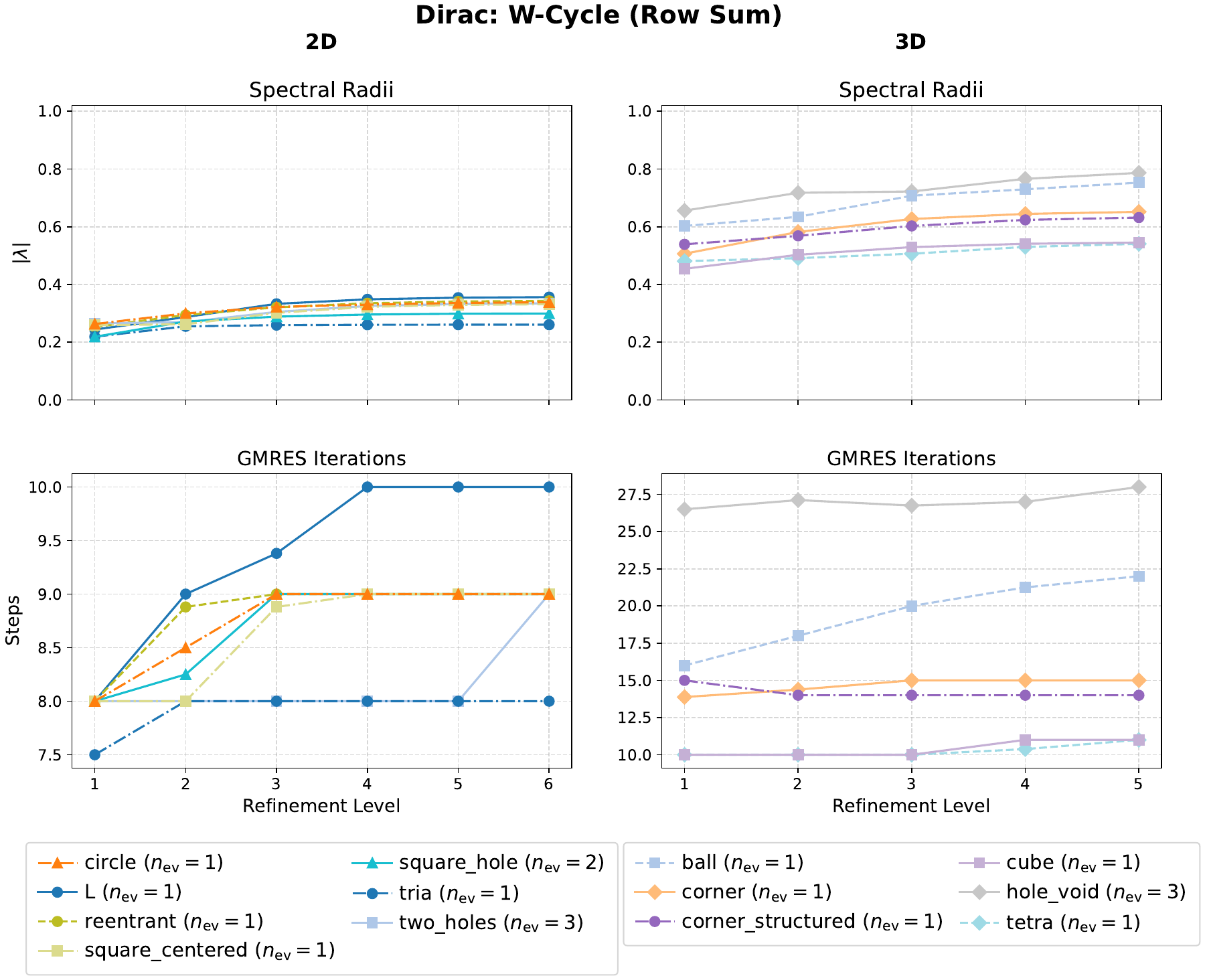}

    \caption{Convergence of the multigrid method for the Dirac operator in 2D and 3D with row-sum mass-lumping (W-cycle).}
    \label{fig:dirac_2d_mass_w}
\end{figure}

\begin{figure}[p]
    \centerfloat
    \centering
    \fullplot{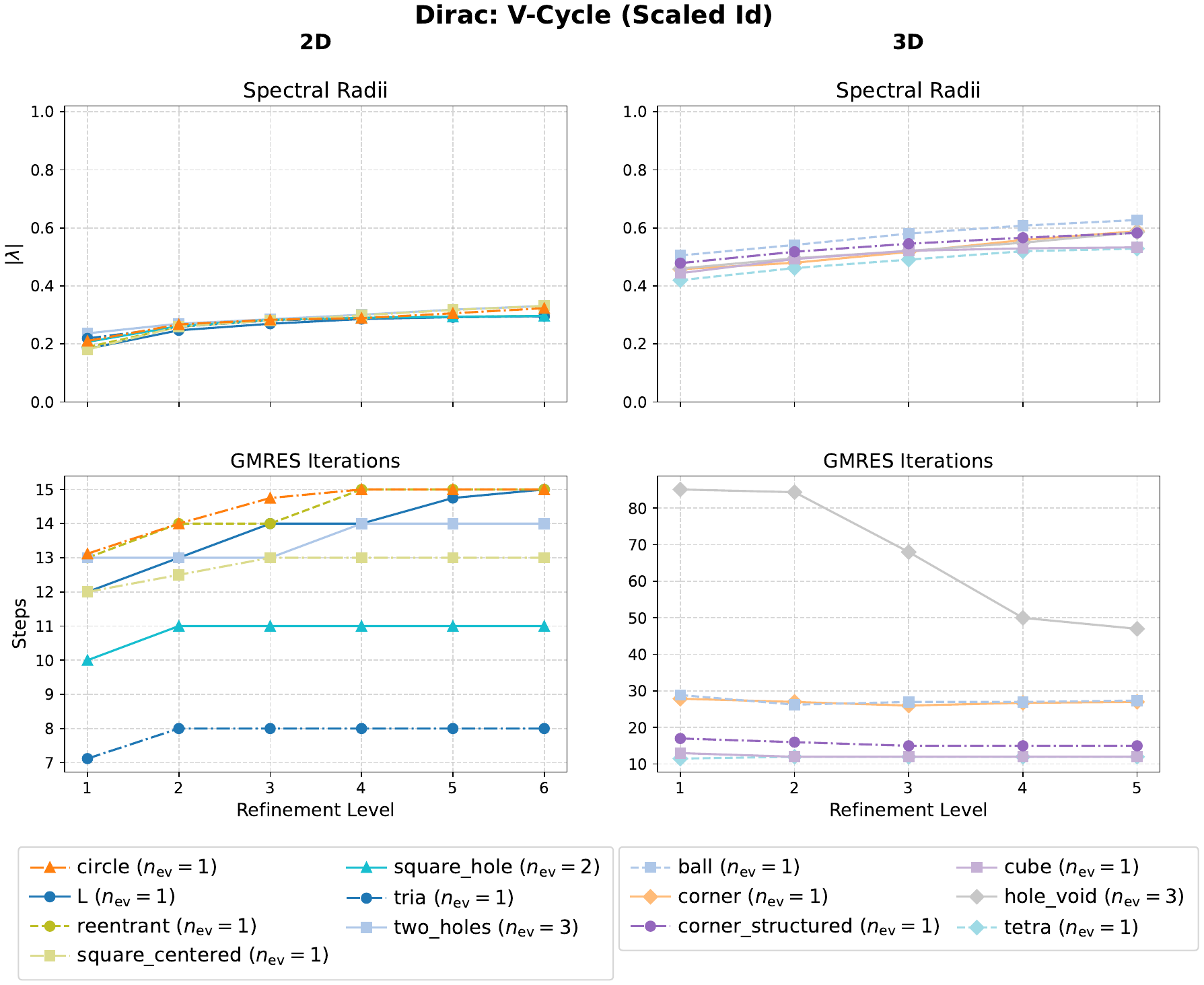}

    \caption{Convergence of the multigrid method for the Dirac operator in 2D and 3D with the scaled identity as the mass-lumping (V-cycle).}
    \label{fig:dirac_2d_scaledid}
\end{figure}

\begin{figure}[p]
    \centerfloat
    \centering
    \fullplot{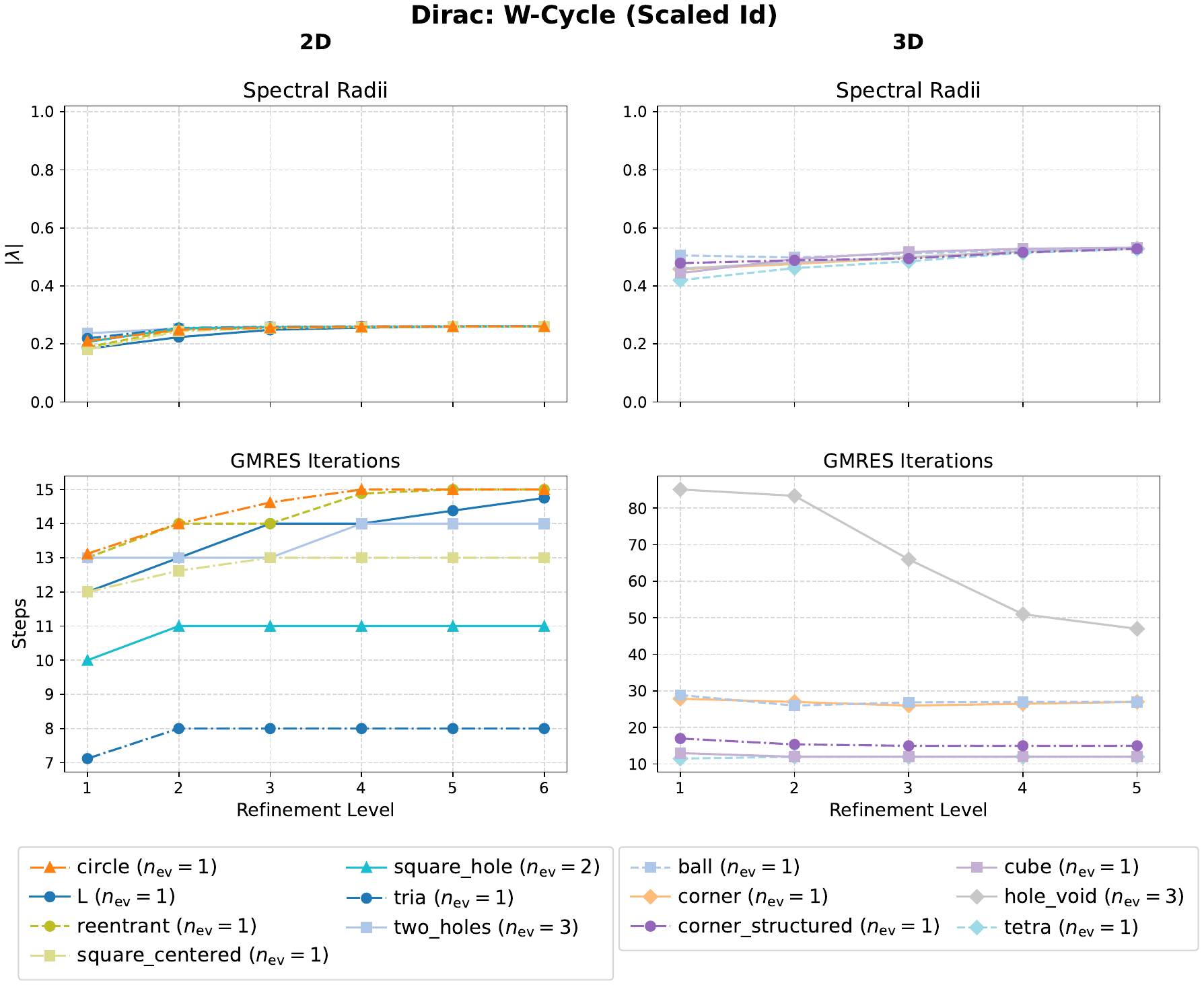}

    \caption{Convergence of the multigrid method for the Dirac operator in 2D and 3D with the scaled identity as the mass-lumping (W-cycle).}
    \label{fig:dirac_2d_scaledid_w}
\end{figure}

\begin{figure}[p]
    \centerfloat
    \centering
    \fullplot{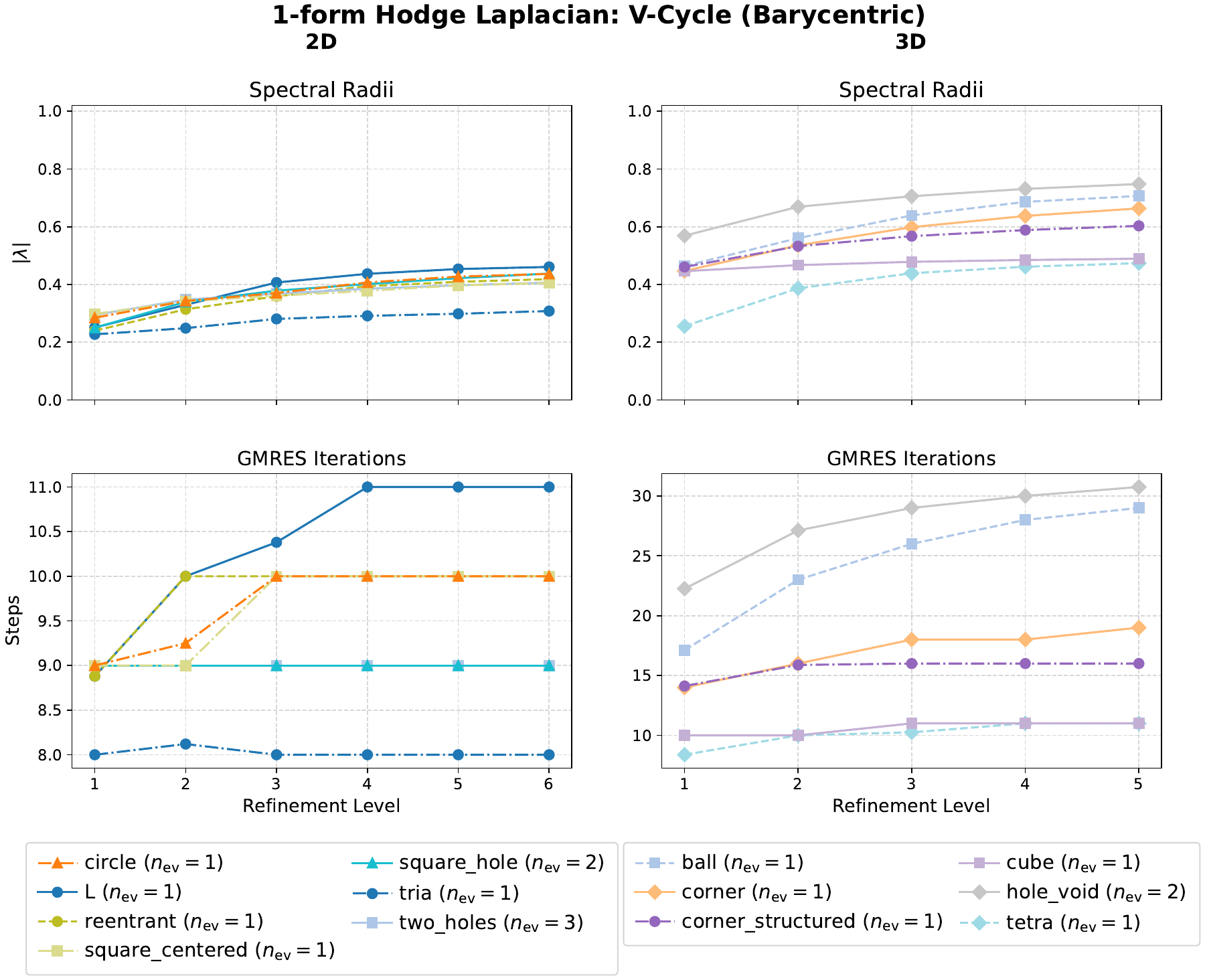}

    \caption{Convergence of the multigrid method for the Hodge-Laplacian on 1-forms in 2D and 3D with barycentric duals (V-cycle).}
    \label{fig:laplace_2d_bary}
\end{figure}

\begin{figure}[p]
    \centerfloat
    \centering
    \fullplot{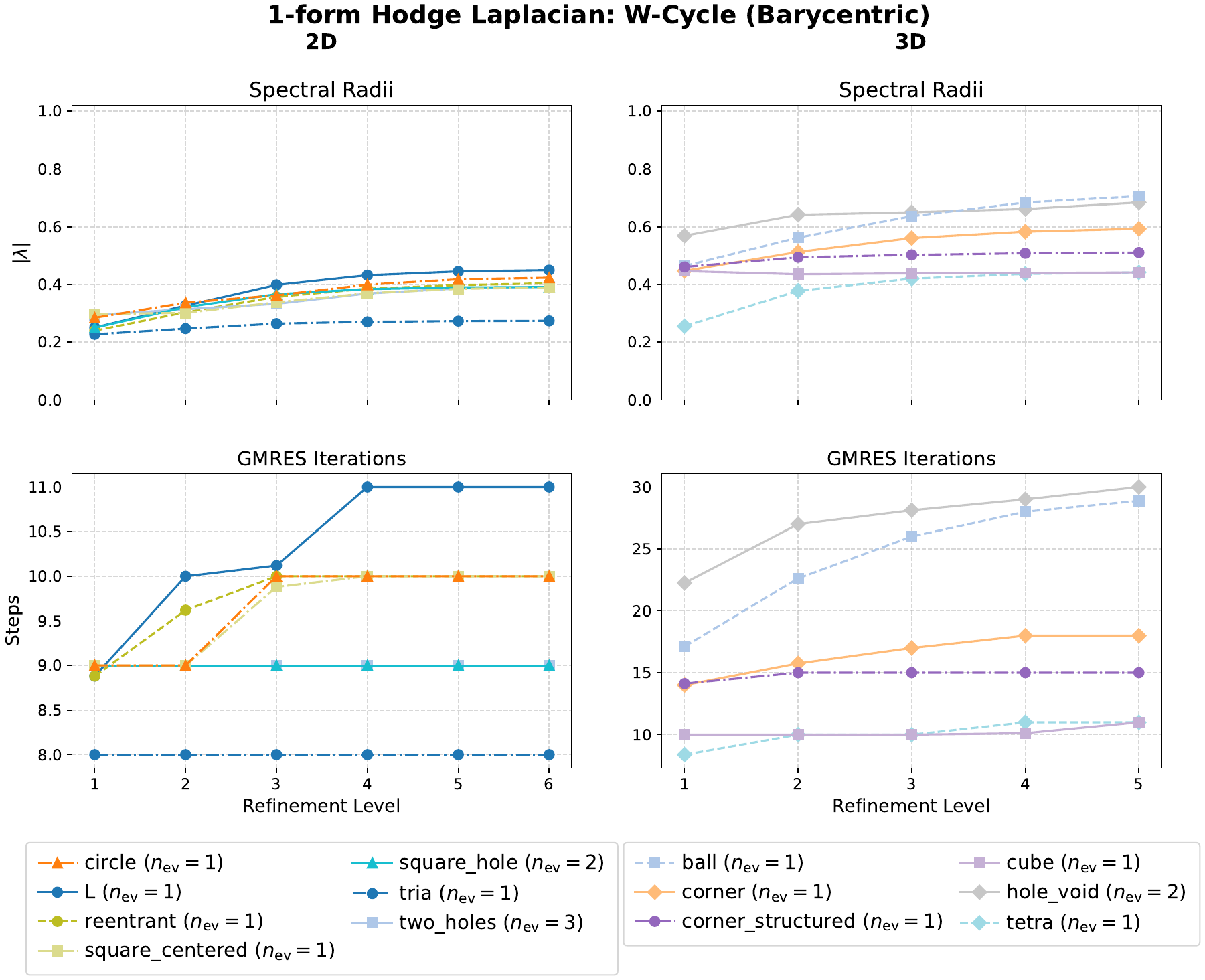}

    \caption{Convergence of the multigrid method for the Hodge-Laplacian on 1-forms in 2D and 3D with barycentric duals (W-cycle).}
    \label{fig:laplace_2d_bary_w}
\end{figure}

\begin{figure}[p]
    \centerfloat
    \centering
    \fullplot{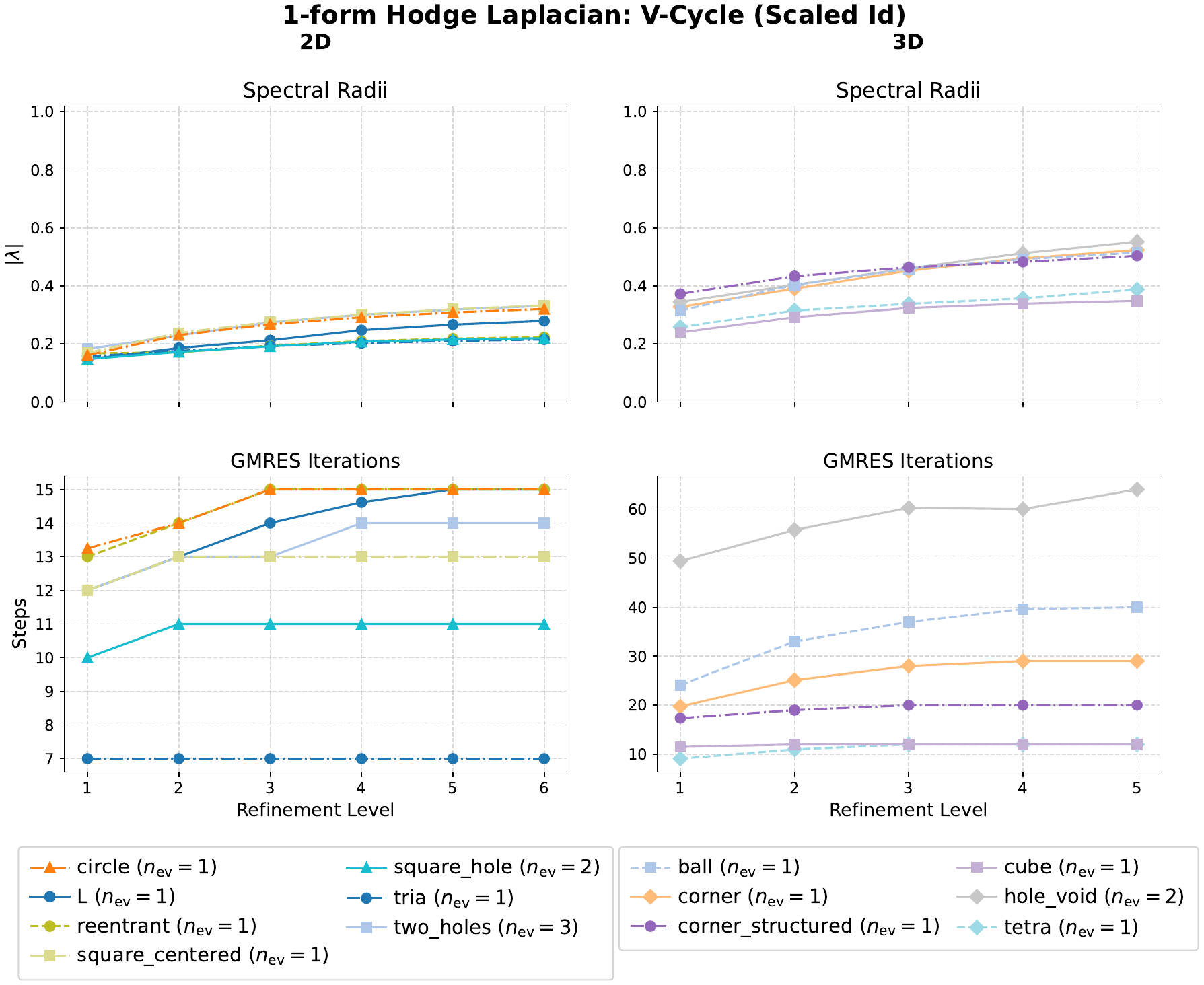}

    \caption{Convergence of the multigrid method for the Hodge-Laplacian on 1-forms in 2D and 3D with the scaled identity as the mass-lumping (V-cycle).}
    \label{fig:laplace_2d_scaledid}
\end{figure}

\begin{figure}[p]
    \centerfloat
    \centering
    \fullplot{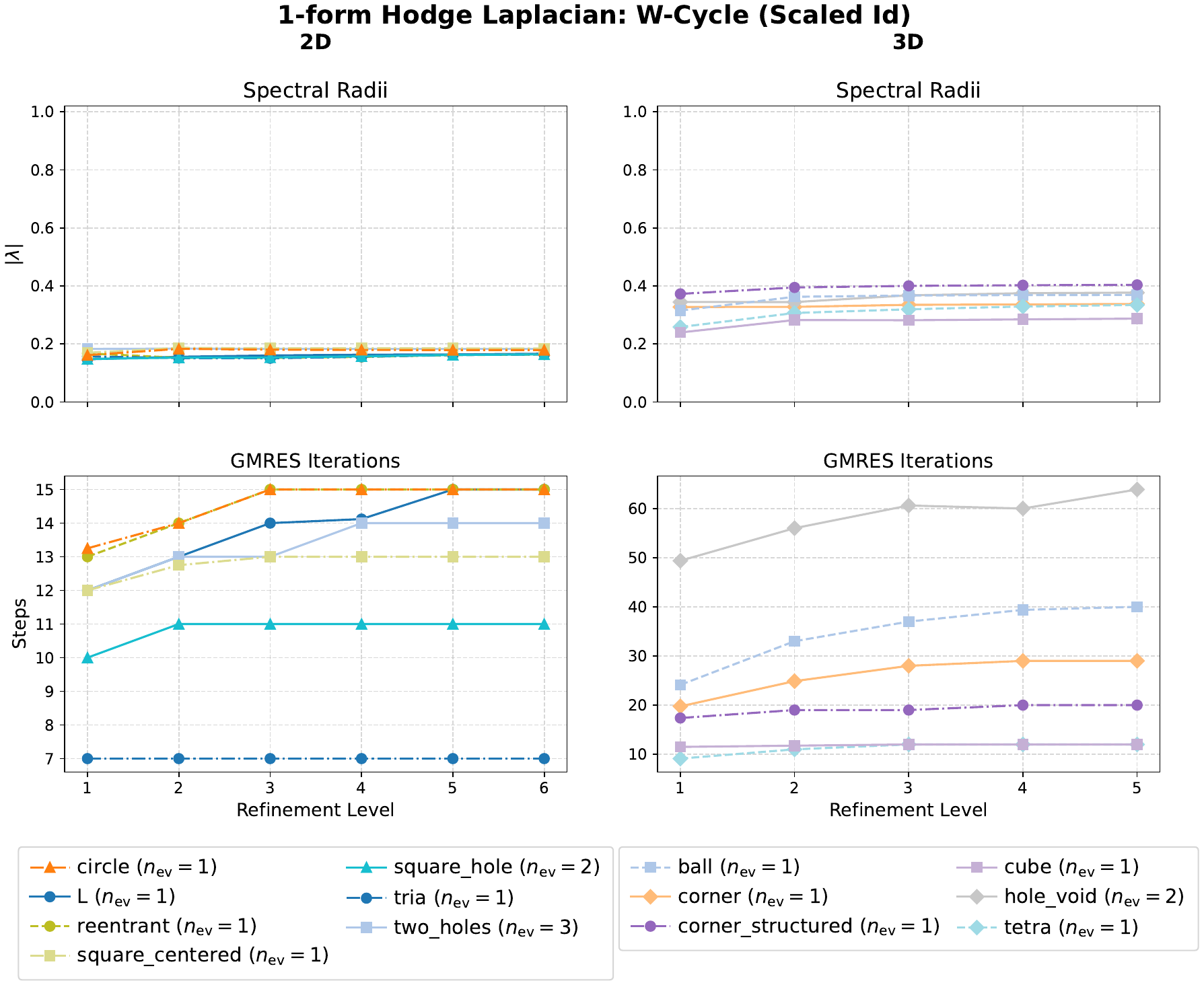}

    \caption{Convergence of the multigrid method for the Hodge-Laplacian on 1-forms in 2D and 3D with the scaled identity as the mass-lumping (W-cycle).}
    \label{fig:laplace_2d_scaledid_w}
\end{figure}

\begin{figure}[p]
    \centerfloat
    \centering
    \fullplot{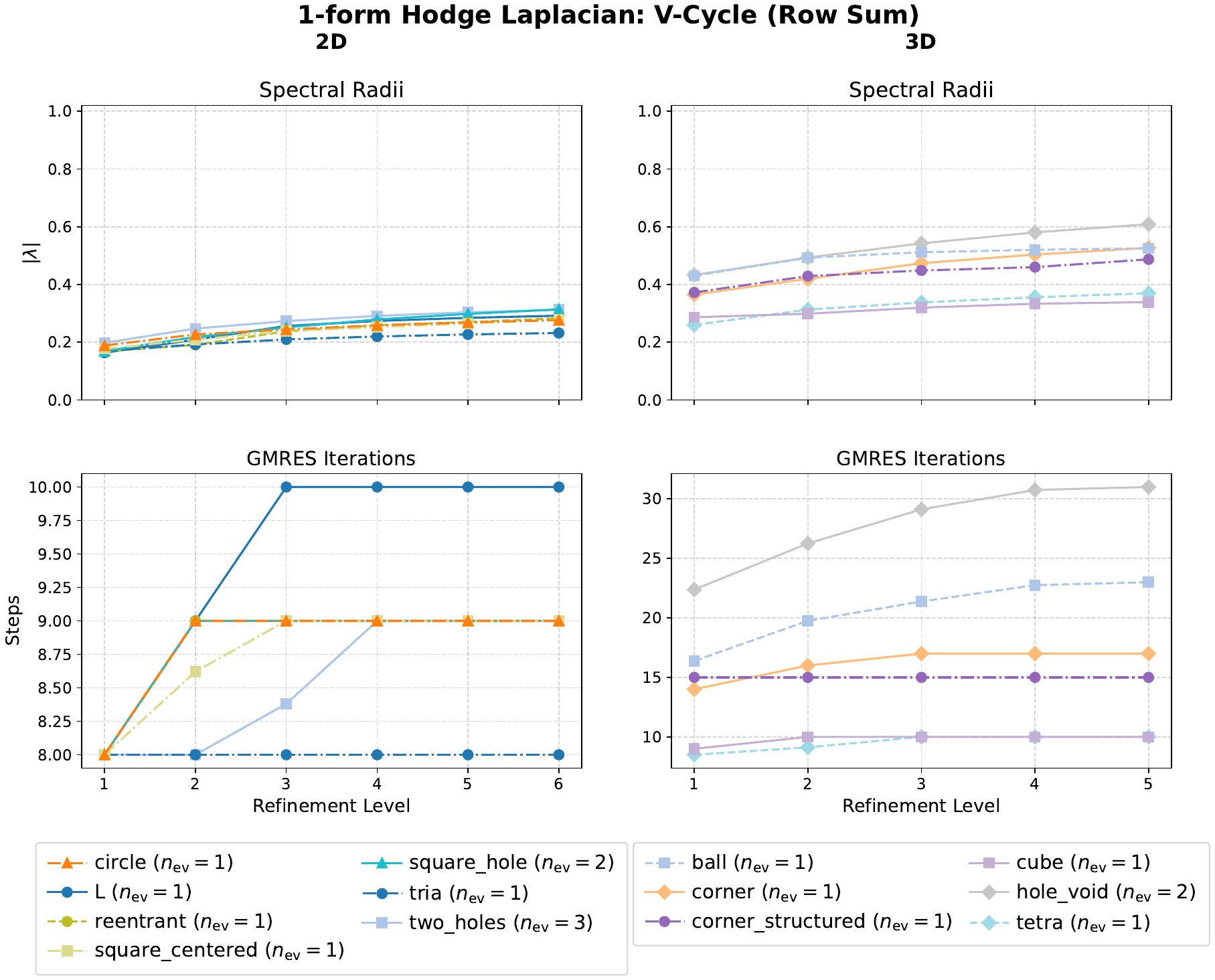}

    \caption{Convergence of the multigrid method for the Hodge-Laplacian on 1-forms in 2D and 3D with row-sum mass-lumping (V-cycle).}
    \label{fig:laplace_2d_mass}
\end{figure}

\begin{figure}[p]
    \centerfloat
    \centering
    \fullplot{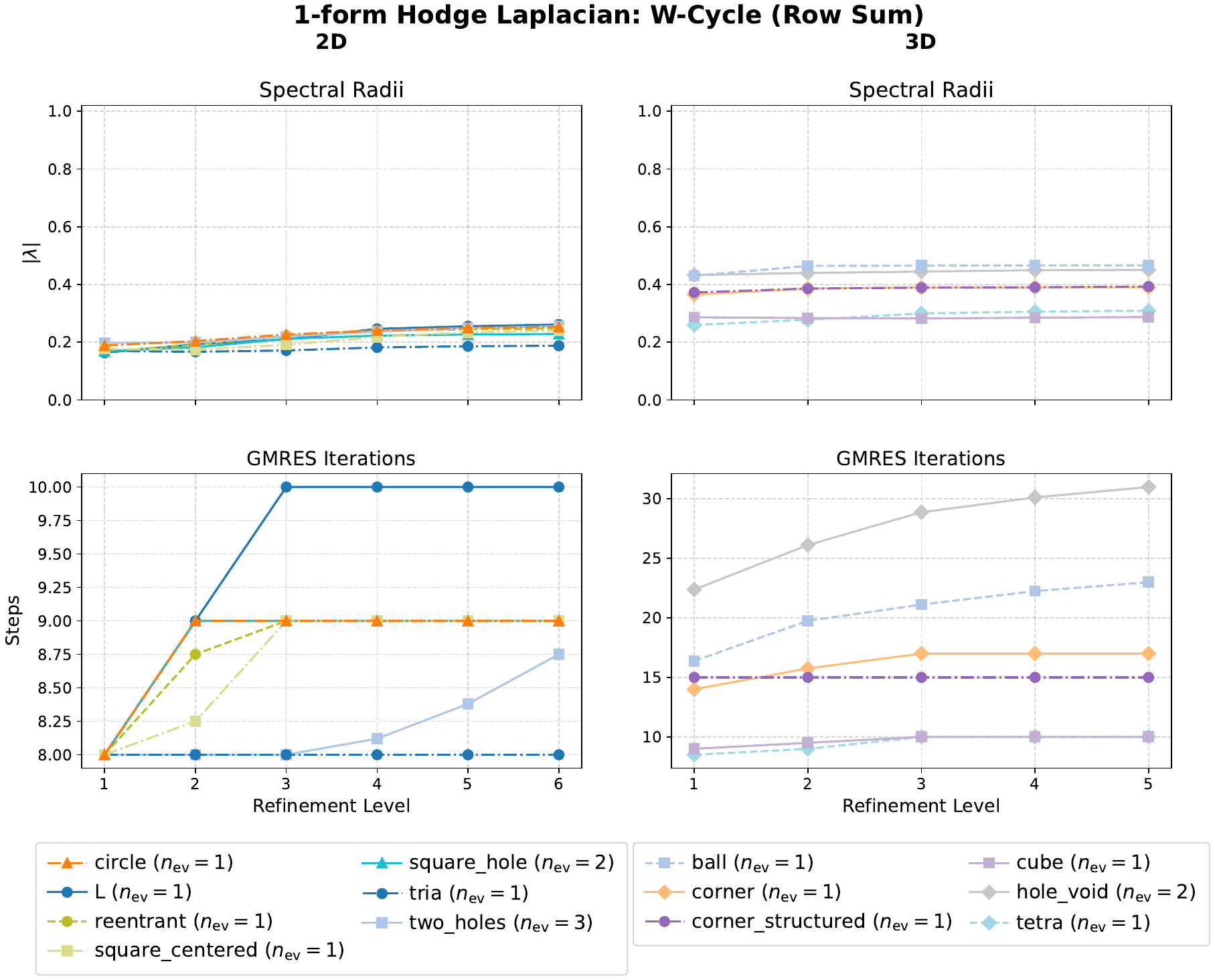}

    \caption{Convergence of the multigrid method for the Hodge-Laplacian on 1-forms in 2D and 3D with row-sum mass-lumping (W-cycle).}
    \label{fig:laplace_2d_mass_w}
\end{figure}

\begin{figure}[p]
    \centerfloat
    \centering
    \fullplot{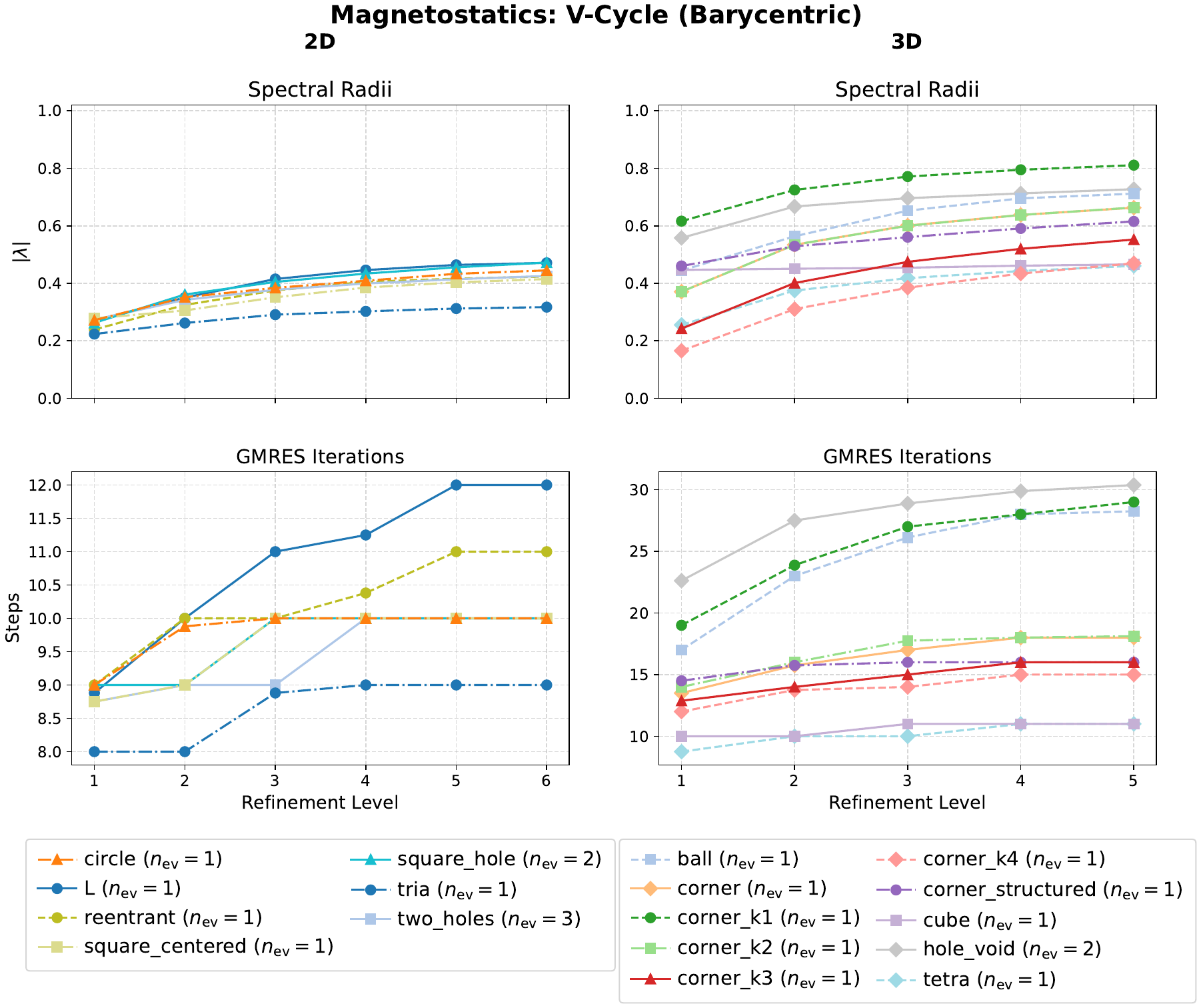}

    \caption{Convergence of the multigrid method for the magnetostatics problem in 2D and 3D with barycentric duals (V-cycle). Here, $k_S$ denotes $S$ post-smoothening steps.}
    \label{fig:mag_2d_bary}
\end{figure}

\begin{figure}[p]
    \centerfloat
    \centering
    \fullplot{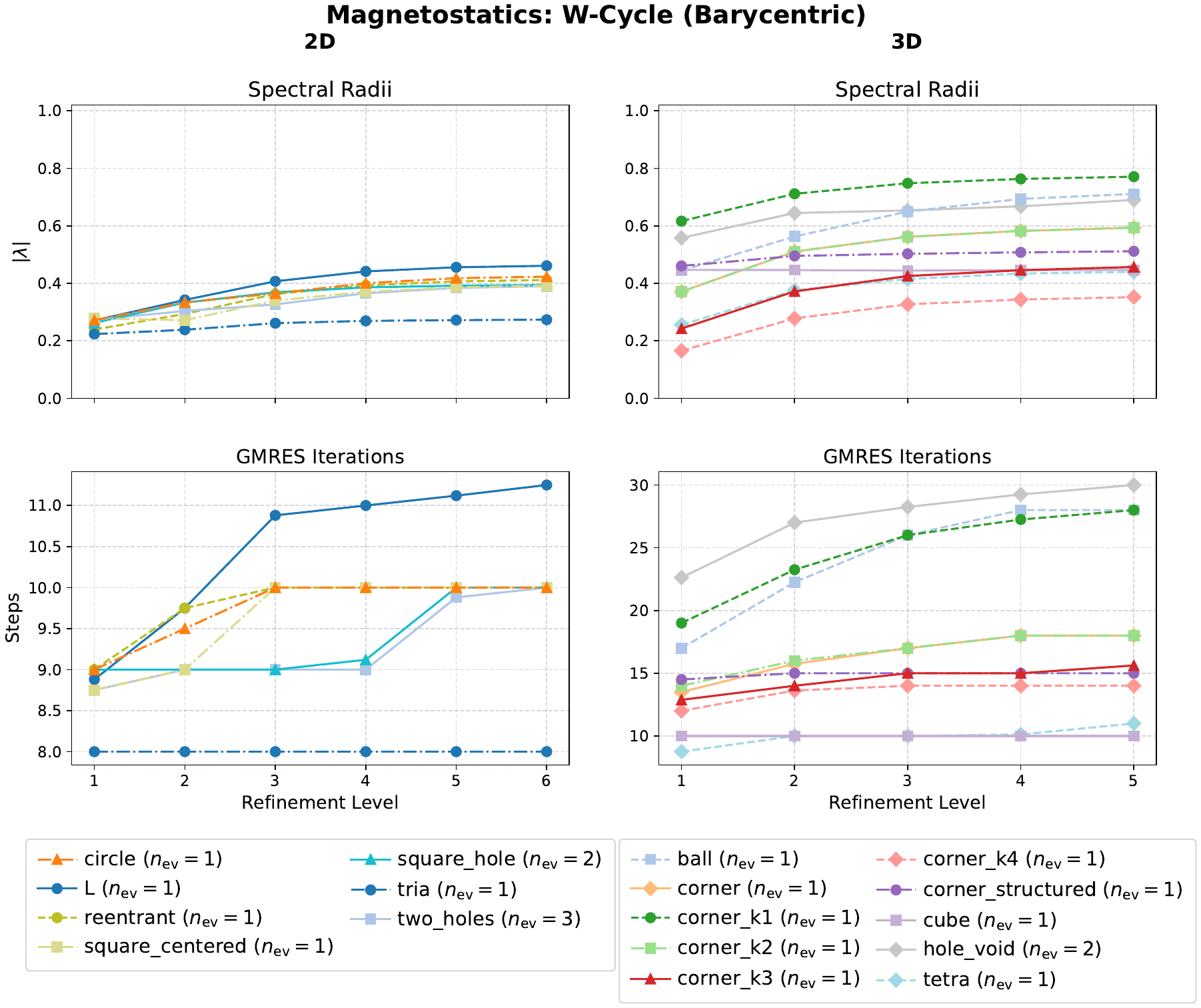}

    \caption{Convergence of the multigrid method for the magnetostatics problem in 2D and 3D with barycentric duals (W-cycle). Here, $kS$ denotes $S$ post-smoothening steps.}
    \label{fig:mag_2d_bary_w}
\end{figure}

\begin{figure}[p]
    \centerfloat
    \centering
    \fullplot{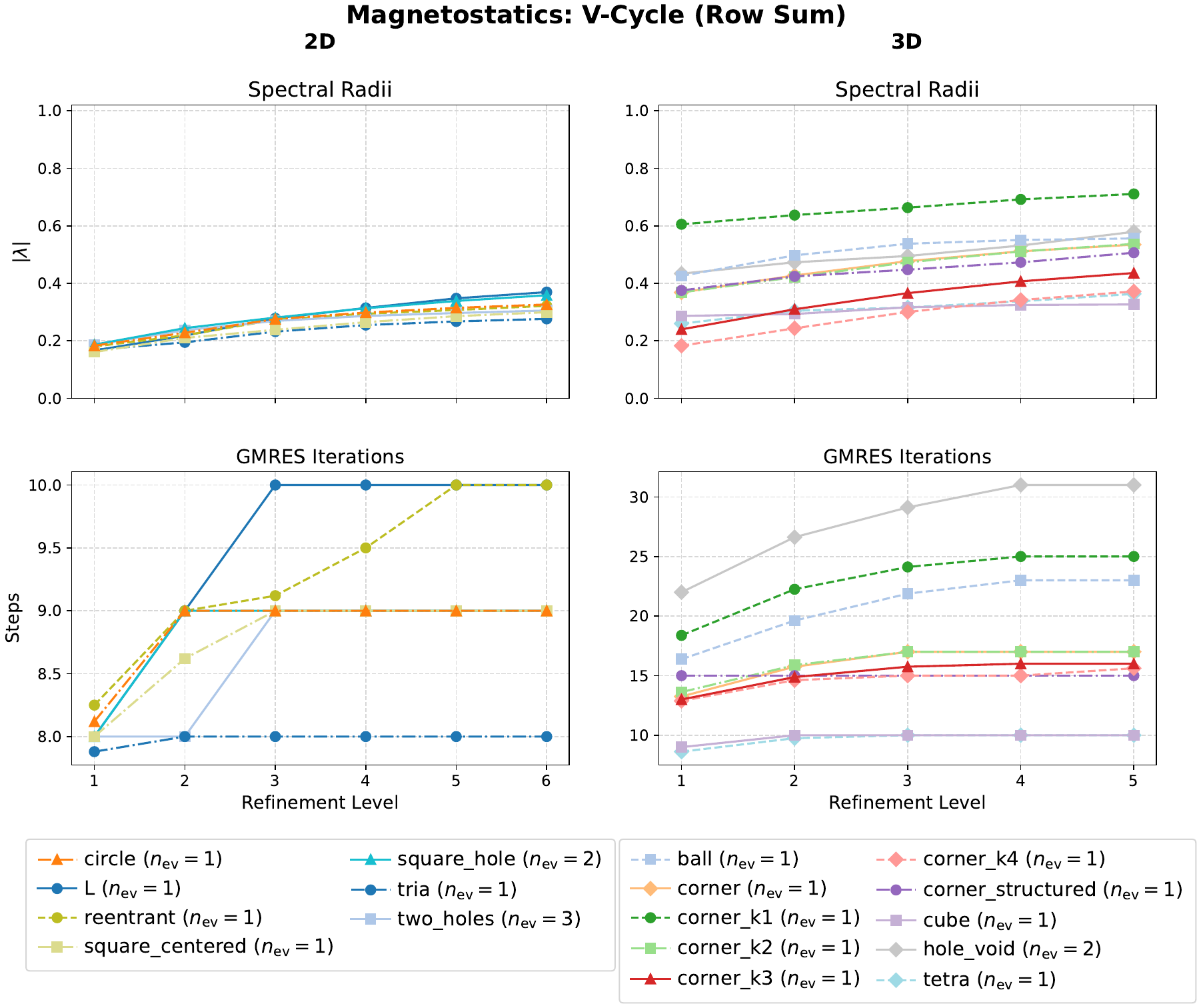}

    \caption{Convergence of the multigrid method for the magnetostatics problem in 2D and 3D with row-sum mass-lumping (V-cycle). Here, \_kS denotes $S$ post-smoothening steps.}
    \label{fig:mag_2d_mass}
\end{figure}

\begin{figure}[p]
    \centerfloat
    \centering
    \fullplot{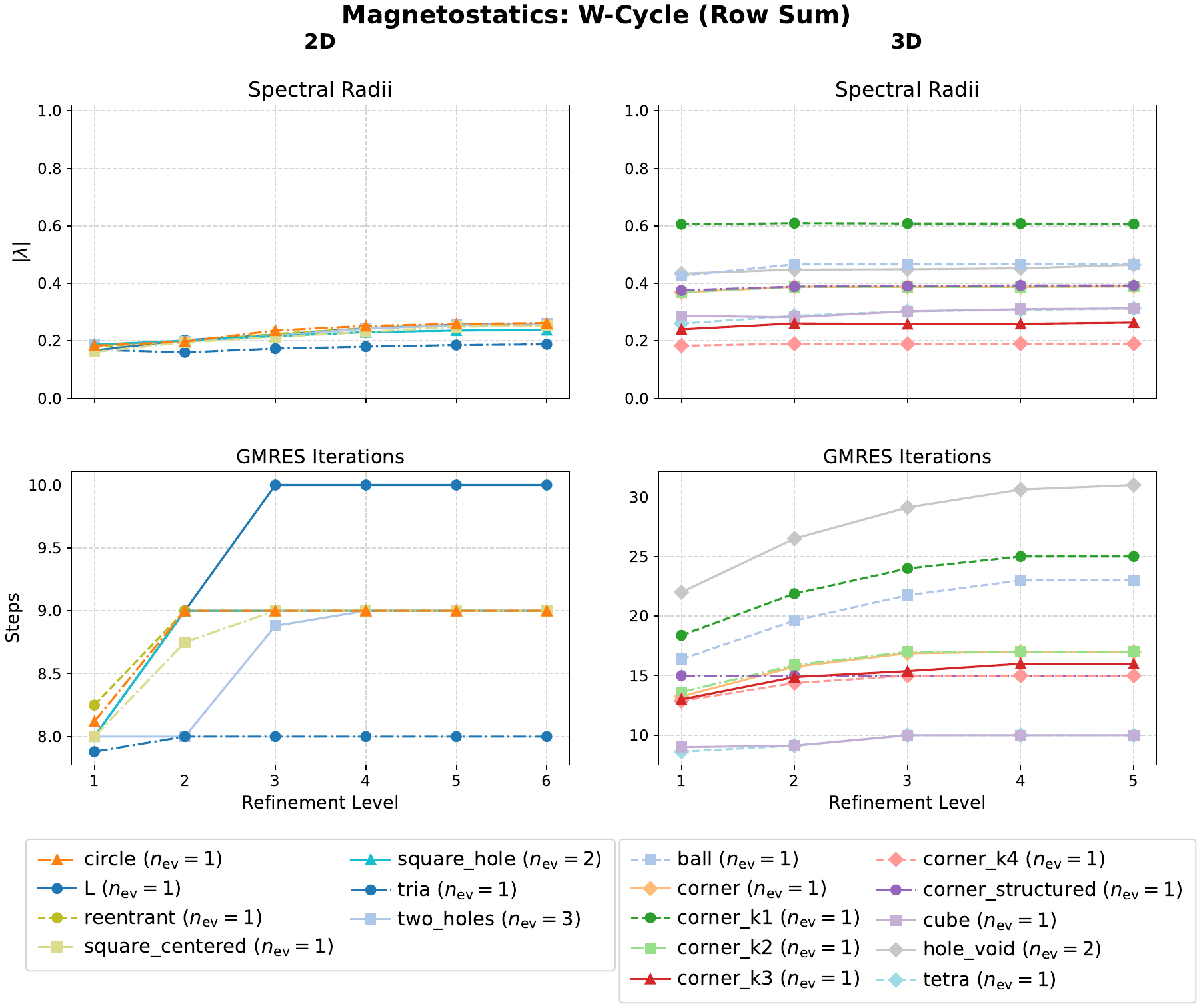}

    \caption{Convergence of the multigrid method for the magnetostatics problem in 2D and 3D with row-sum mass-lumping (W-cycle). Here, \_kS denotes $S$ post-smoothening steps.}
    \label{fig:mag_2d_mass_w}
\end{figure}

\begin{figure}[p]
    \centerfloat
    \centering
    \fullplot{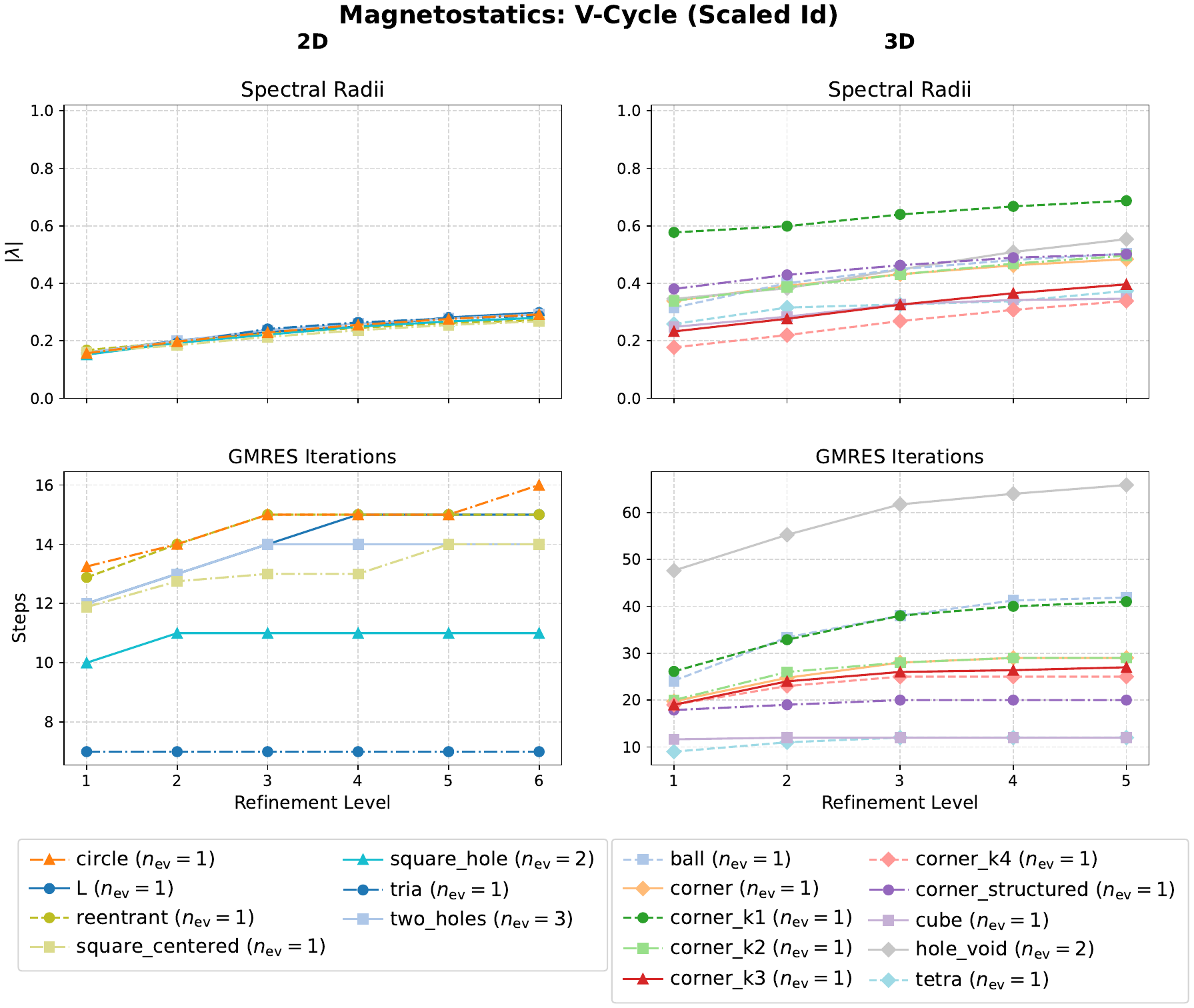}

    \caption{Convergence of the multigrid method for the magnetostatics problem in 2D and 3D with the scaled identity as the mass-lumping (V-cycle). Here, \_kS denotes $S$ post-smoothening steps.}
    \label{fig:mag_2d_scaledid}
\end{figure}

\begin{figure}[p]
    \centerfloat
    \centering
    \fullplot{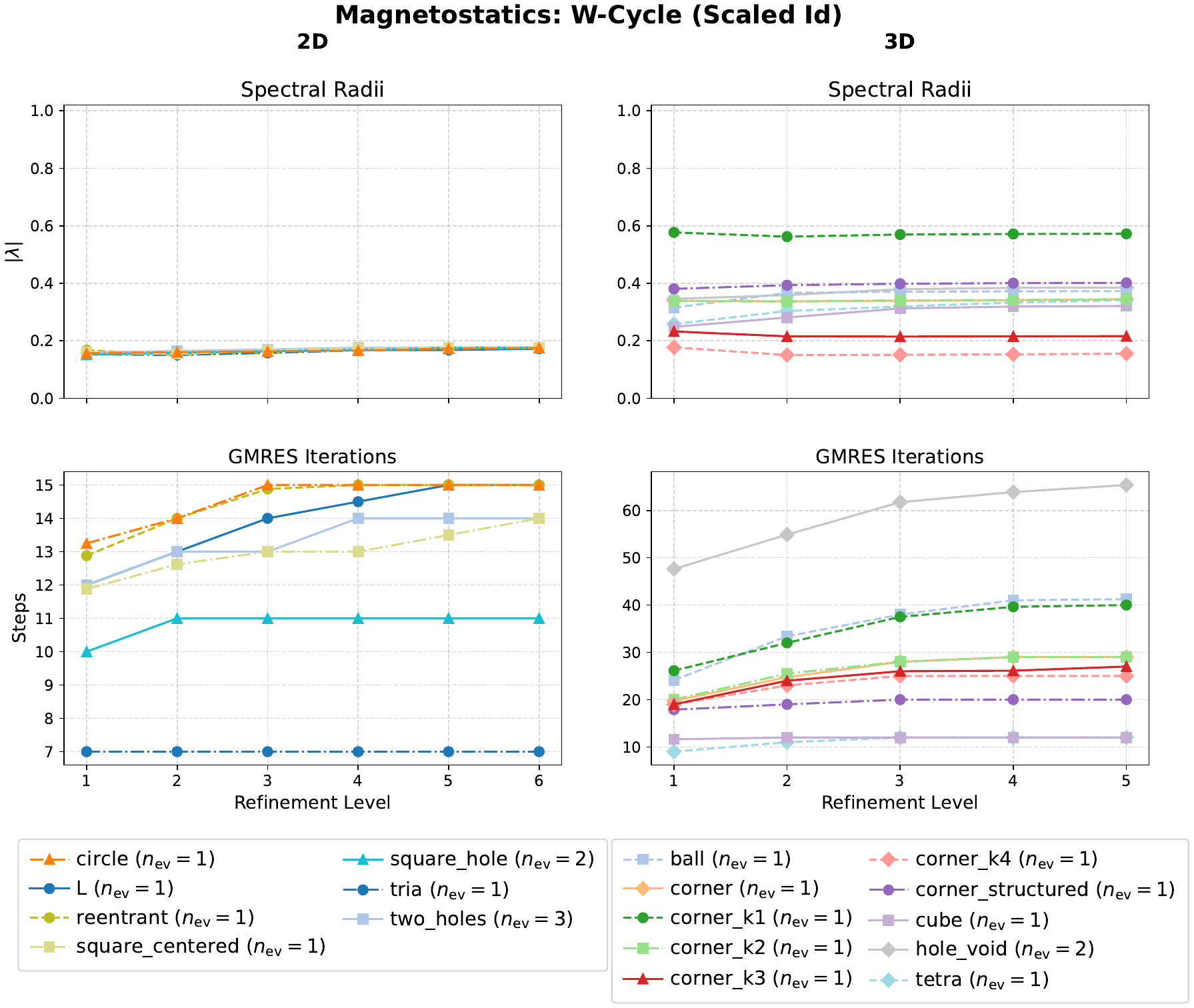}

    \caption{Convergence of the multigrid method for the magnetostatics problem in 2D and 3D with the scaled identity as the mass-lumping (W-cycle). Here, \_kS denotes $S$ post-smoothening steps.}
    \label{fig:mag_2d_scaledid_w}
\end{figure}

\clearpage

\section{Conclusion}
In conclusion, we have presented a new multigrid approach for Hodge-Dirac operators and mixed formulations of the Hodge-Laplacian, as well as a saddle-point problem arising from magnetostatics within the framework of differential forms. 

Its central feature is the use of transforming smoothers applied to mass-lumped systems, whose mass matrices possess explicit inverses and whose algebraic structure replicates desirable properties on the discrete level. The resulting transformed systems exhibit a block-triangular structure with positive definite diagonal blocks, making them amenable to smoothing, in contrast to the original systems, where standard schemes like Jacobi or Gauss-Seidel are not applicable.

Under mild conditions on the mass-lumping, we established stability of the mass-lumped problem, and by extension spectral equivalence to the FEEC problem.

Numerical experiments demonstrate that the proposed multigrid methods, when employed as preconditioners for \texttt{GMRES}, perform well in practice. Moreover, tests on meshes with nontrivial topologies show that the method remains robust in the presence of harmonic forms. These observations motivate further investigation of extensions (for example, to Stokes problems or problems with coefficient jumps) as well as a rigorous theoretical study encompassing a convergence theory for the multigrid scheme and a proof establishing uniform bounds on the \texttt{GMRES} iteration counts.

\bibliographystyle{plainnat}
\bibliography{sources}

\appendix

\end{document}

%% file: macros.tex
\newcommand{\R}{\ensuremath{\mathbb{R}}}

\DeclarePairedDelimiter{\Paren}{(}{)}
\DeclarePairedDelimiter{\Bracket}{[}{]}
\DeclarePairedDelimiter{\Brace}{\{}{\}}

\makeatletter
\newcommand*{\centerfloat}{%
  \parindent \z@
  \leftskip \z@ \@plus 1fil \@minus \textwidth
  \rightskip\leftskip
  \parfillskip \z@skip}
\makeatother

\NewDocumentCommand{\Oh}{som}{%
  \ensuremath{\mathcal{O}\IfBooleanTF{#1}{%
    \Paren*{#3}%          % *-version => auto-size
  }{%
    \IfNoValueTF{#2}{%
      \Paren{#3}%         % no optional => normal size
    }{%
      \Paren[#2]{#3}%     % optional => \big, \Big, etc.
    }%
  }%
}}

\NewDocumentCommand{\Om}{som}{%
\ensuremath{\Omega\IfBooleanTF{#1}{%
    \Paren*{#3}%
  }{%
    \IfNoValueTF{#2}{%
      \Paren{#3}%
    }{%
      \Paren[#2]{#3}%
    }%
  }%
}}

\NewDocumentCommand{\Th}{som}{%
\ensuremath{\Theta\IfBooleanTF{#1}{%
    \Paren*{#3}%
  }{%
    \IfNoValueTF{#2}{%
      \Paren{#3}%
    }{%
      \Paren[#2]{#3}%
    }%
  }%
}}

\NewDocumentCommand{\oh}{som}{%
\ensuremath{o\IfBooleanTF{#1}{%
    \Paren*{#3}%
  }{%
    \IfNoValueTF{#2}{%
      \Paren{#3}%
    }{%
      \Paren[#2]{#3}%
    }%
  }%
}}

\NewDocumentCommand{\om}{som}{%
\ensuremath{\omega\IfBooleanTF{#1}{%
    \Paren*{#3}%
  }{%
    \IfNoValueTF{#2}{%
      \Paren{#3}%
    }{%
      \Paren[#2]{#3}%
    }%
  }%
}}

\NewDocumentCommand{\Prob}{som}{%
\ensuremath{\mathbb{P}% or \Pr, if you prefer
  \IfBooleanTF{#1}{%
    \Bracket*{#3}%
  }{%
    \IfNoValueTF{#2}{%
      \Bracket{#3}%
    }{%
      \Bracket[#2]{#3}%
    }%
  }%
}}

\NewDocumentCommand{\E}{som}{%
\ensuremath{\mathbb{E}
  \IfBooleanTF{#1}{%
    \Bracket*{#3}%
  }{%
    \IfNoValueTF{#2}{%
      \Bracket{#3}%
    }{%
      \Bracket[#2]{#3}%
    }%
  }%
}}

\NewDocumentCommand{\Max}{s o m m}{%
\ensuremath{\max
  \IfBooleanTF{#1}{%
    \Brace*{#3,\, #4}%
  }{%
    \IfNoValueTF{#2}{%
      \Brace{#3,\, #4}%
    }{%
      \Brace[#2]{#3,\, #4}
    }%
  }%
}}

\NewDocumentCommand{\Min}{s o m m}{%
\ensuremath{\min
  \IfBooleanTF{#1}{%
    \Brace*{#3,\, #4}%
  }{%
    \IfNoValueTF{#2}{%
      \Brace{#3,\, #4}%
    }{%
      \Brace[#2]{#3,\, #4}
    }%
  }%
}}

\NewDocumentCommand{\Exp}{som}{%
\ensuremath{\exp
  \IfBooleanTF{#1}{%
    \Paren*{#3}%
  }{%
    \IfNoValueTF{#2}{%
      \Paren{#3}%
    }{%
      \Paren[#2]{#3}%
    }%
  }%
}}

\NewDocumentCommand{\Log}{som}{%
\ensuremath{\log
  \IfBooleanTF{#1}{%
    \Paren*{#3}%
  }{%
    \IfNoValueTF{#2}{%
      \Paren{#3}%
    }{%
      \Paren[#2]{#3}%
    }%
  }%
}}

\NewDocumentCommand{\Var}{som}{%
\ensuremath{\mathrm{Var}
  \IfBooleanTF{#1}{%
    \Bracket*{#3}%
  }{%
    \IfNoValueTF{#2}{%
      \Bracket{#3}%
    }{%
      \Bracket[#2]{#3}%
    }%
  }%
}}

\NewDocumentCommand{\SetBuilder}{s o m m}{%
\ensuremath{\IfBooleanTF{#1}{%
    \Brace*{\,#3 : #4\,}%
  }{%
    \IfNoValueTF{#2}{%
      \Brace{\,#3 : #4\,}%
    }{%
      \Brace[#2]{\,#3 : #4\,}%
    }%
  }%
}}

\renewcommand{\phi}{\ensuremath{\varphi}}

\DeclareMathOperator{\GRAD}{{\textnormal{grad}}}
\DeclareMathOperator{\CURL}{{\textnormal{curl}}}
\DeclareMathOperator{\DIV}{{\textnormal{div}}}
\DeclareMathOperator{\ran}{{\textnormal{ran}}}
\DeclareMathOperator{\tril}{{\textnormal{tril}}}
\DeclareMathOperator{\triu}{{\textnormal{triu}}}
\DeclareMathOperator{\Id}{{\textnormal{Id}}}
\newcommand{\DD}{{\textnormal{D}}}

\newcommand{\inner}[2]{\left\langle #1, #2 \right\rangle}

\NewDocumentCommand{\XLambdak}{ O{} O{} }{\ensuremath{{#1}\Lambda^{#2}\left(\Omega\right)}}
\NewDocumentCommand{\ringWhitneyForms}{ o O{h} }{\ensuremath{\mathring{V}\IfValueT{#1}{_{#1}}\IfValueT{#2}{^{#2}}}}

\NewDocumentCommand{\hilbertSpaceFE}{ o O{L^2} O{h} }{\ensuremath{\left(\ringWhitneyForms[#1][#3], \inner{\cdot}{\cdot}_{\XLambdak[#2][#1]}\right)}}
\NewDocumentCommand{\hilbertSpaceDEC}{ o o O{h} }{\ensuremath{\left(\ringWhitneyForms[#1][#3], \innerc{\cdot}{\cdot}[\IfValueTF{#2}{#2}{\IfValueTF{#1}{#1}{}}]\right)}}
\NewDocumentCommand{\IsoFEDEC}{o O{h} }{\ensuremath{\mathsf{I}\IfValueT{#1}{^{#1}}\IfValueT{#2}{_{#2}}}}
\NewDocumentCommand{\IsoSobolevFEDEC}{o O{h} }{\ensuremath{\bar{\mathsf{I}}\IfValueT{#1}{^{#1}}\IfValueT{#2}{_{#2}}}}

\RenewDocumentCommand{\d}{ O{} }{\ensuremath{\mathsf{d}\IfValueT{#1}{^{#1}}}}
\RenewDocumentCommand{\dh}{ O{} }{\ensuremath{\mathsf{d}\IfValueT{#1}{^{#1}}_h}}
\NewDocumentCommand{\deltaML}{ O{} }{\ensuremath{\tilde{\delta}^h\IfValueT{#1}{_{#1}}}}

\NewDocumentCommand{\IdMat}{}{\ensuremath{\mathbf{I}}}
\NewDocumentCommand{\dhMat}{ O{} }{\ensuremath{\mathbf{d}\IfValueT{#1}{^{#1}}_h}}
\NewDocumentCommand{\deltaMLMat}{ O{} }{\ensuremath{\widetilde{\boldsymbol{\delta}}^h\IfValueT{#1}{_{#1}}}}
\NewDocumentCommand{\deltaFE}{ O{} }{\ensuremath{\delta^h\IfValueT{#1}{_{#1}}}}

\NewDocumentCommand{\rieszFE}{ o O{h} }{\ensuremath{\mathsf{R}\IfValueT{#1}{^{#1}}\IfValueT{#2}{_{#2}}}}
\NewDocumentCommand{\rieszDEC}{ o O{h} }{\ensuremath{\tilde{\mathsf{R}}\IfValueT{#1}{^{#1}}\IfValueT{#2}{_{#2}}}}

\NewDocumentCommand{\massFE}{ O{} }{\ensuremath{\mathbf{M}_{#1}}}
\NewDocumentCommand{\massDEC}{ O{} }{\ensuremath{\tilde{\mathbf{M}}_{#1}}}

\NewDocumentCommand{\InitOpFamily}{ m m }{%
  \expandafter\NewDocumentCommand\csname #1BFEEC\endcsname{ m m o O{h} }{\ensuremath{\mathcal{#2}\IfNoValueF{##3}{^{##3}}_{##4}\left( ##1, ##2 \right)}}%
  \expandafter\NewDocumentCommand\csname #1BDEC\endcsname{ m m o O{h} }{\ensuremath{\widetilde{\mathcal{#2}}\IfNoValueF{##3}{^{##3}}_{##4}\left( ##1, ##2 \right)}}%
  
  \expandafter\NewDocumentCommand\csname #1GalerkinOpFEEC\endcsname{ o O{h} }{\ensuremath{\mathbb{#2}\IfNoValueF{##1}{^{##1}}_{##2}}}%
  \expandafter\NewDocumentCommand\csname #1GalerkinOpDEC\endcsname{ o O{h} }{\ensuremath{\widetilde{\mathbb{#2}}\IfNoValueF{##1}{^{##1}}_{##2}}}%
  
  \expandafter\NewDocumentCommand\csname #1SobolevOpFEEC\endcsname{ o O{h} }{\ensuremath{\overline{\mathbb{#2}}\IfNoValueF{##1}{^{##1}}_{##2}}}%
  \expandafter\NewDocumentCommand\csname #1SobolevOpDEC\endcsname{ o O{h} }{\ensuremath{\overline{\widetilde{\mathbb{#2}}}\IfNoValueF{##1}{^{##1}}_{##2}}}%
  
  \expandafter\NewDocumentCommand\csname #1OpFEEC\endcsname{ o O{h} }{\ensuremath{\mathsf{#2}\IfNoValueF{##1}{^{##1}}_{##2}}}%
  \expandafter\NewDocumentCommand\csname #1OpDEC\endcsname{ o O{h} }{\ensuremath{\widetilde{\mathsf{#2}}\IfNoValueF{##1}{^{##1}}_{##2}}}%
}
\NewDocumentCommand{\InitMatOpFamily}{ m m }{%
  \expandafter\NewDocumentCommand\csname #1MatOpFEEC\endcsname{ o O{h} }{\ensuremath{\mathbf{#2}\IfNoValueF{##1}{^{##1}}_{##2}}}%
  \expandafter\NewDocumentCommand\csname #1MatOpDEC\endcsname{ o O{h} }{\ensuremath{\widetilde{\mathbf{#2}}\IfNoValueF{##1}{^{##1}}_{##2}}}%
}

\InitMatOpFamily{dirac}{D}
\InitMatOpFamily{laplace}{L}
\InitMatOpFamily{mag}{M}
\InitMatOpFamily{generic}{G}

\NewDocumentCommand{\transformedSystemMat}{ O{h} }{\ensuremath{\mathbf{T}_{#1}}}
\NewDocumentCommand{\transformedSystemSplitMat}{ O{h} }{\ensuremath{\mathbf{Q}_{#1}}}

\InitOpFamily{dirac}{D}
\InitOpFamily{laplace}{L}
\InitOpFamily{mag}{M}
\InitOpFamily{generic}{G}

\NewDocumentCommand{\InitDiscSpaceFamily}{ m m }{%
  \expandafter\NewDocumentCommand\csname #1DiscWFEEC\endcsname{ o O{h} }{\ensuremath{#2_{##2}\IfValueT{##1}{^{##1}}}}%
  \expandafter\NewDocumentCommand\csname #1DiscWSobolevFEEC\endcsname{ o O{h} O{H} }{\ensuremath{#2_{##2,##3}\IfValueT{##1}{^{##1}}}}%
  \expandafter\NewDocumentCommand\csname #1DiscWDEC\endcsname{ o O{h} }{\ensuremath{\tilde{#2}_{##2}\IfValueT{##1}{^{##1}}}}%
  \expandafter\NewDocumentCommand\csname #1DiscWSobolevDEC\endcsname{ o O{h} O{H} }{\ensuremath{\tilde{#2}_{##2,##3}\IfValueT{##1}{^{##1}}}}%
}

\InitDiscSpaceFamily{laplace}{W}
\InitDiscSpaceFamily{mag}{U}
\InitDiscSpaceFamily{generic}{H}

\NewDocumentCommand{\genericIso}{O{h}}{%
  \ensuremath{\mathsf{J}\IfValueT{#1}{_{#1}}}%
}
\NewDocumentCommand{\genericRieszFE}{O{h}}{%
  \ensuremath{\mathsf{H}\IfValueT{#1}{_{#1}}}%
}
\NewDocumentCommand{\genericRieszDEC}{O{h}}{%
  \ensuremath{\tilde{\mathsf{H}}\IfValueT{#1}{_{#1}}}%
}

\NewDocumentCommand{\prolong}{ O{h} }{\ensuremath{\hat{\mathsf{P}}_{#1}}}
\NewDocumentCommand{\restrict}{ O{h} }{\ensuremath{\check{\mathsf{P}}_{#1}}}

\NewDocumentCommand{\transformSymbol}{}{\ensuremath{\mathsf{T}}}
\NewDocumentCommand{\transformR}{ O{h} }{\ensuremath{{\transformSymbol}^{\mathrm{R}}_{#1}}}
\NewDocumentCommand{\transformL}{ O{h} }{\ensuremath{{\transformSymbol}^{\mathrm{L}}_{#1}}}
\NewDocumentCommand{\transformedSystem}{ O{h} }{\ensuremath{\mathsf{T}_{#1}}}
\NewDocumentCommand{\transformedSystemSplit}{ O{h} }{\ensuremath{\mathsf{Q}_{#1}}}

\NewDocumentCommand{\cochain}{ O{} O{\mathcal{T}} O{} }{\ensuremath{C_{#3}^{#1}\left(#2\right)}}
\NewDocumentCommand{\ringCochain}{ O{} O{\mathcal{T}} O{} }{\ensuremath{\mathring{C}_{#3}^{#1}\left(#2\right)}}

\NewDocumentCommand{\cobound}{ O{} }{\ensuremath{\partial^{#1}}}

\NewDocumentCommand{\whitneyForms}{ O{} }{\ensuremath{V_{#1}^h}}

\NewDocumentCommand{\deRham}{ O{} }{\ensuremath{\mathcal{R}^{#1}}}
\NewDocumentCommand{\whitneyMap}{ O{} }{\ensuremath{\mathcal{W}^{#1}}}

\NewDocumentCommand{\normm}{m O{} }{\ensuremath{{\left\vert\kern-0.25ex\left\vert\kern-0.25ex\left\vert #1 
    \right\vert\kern-0.25ex\right\vert\kern-0.25ex\right\vert}}_{#2}}
\NewDocumentCommand{\innerc}{m m O{} }{\ensuremath{\left\llbracket #1, #2 \right\rrbracket}_{#3}}
\NewDocumentCommand{\innerforms}{ m m O{L^2} O{} }{\ensuremath{\inner{#1}{#2}_{\XLambdak[#3][#4]}}}
\NewDocumentCommand{\formNorm}{ m O{L^2} O{} }{\ensuremath{\norm{#1}_{\XLambdak[#2][#3]}}}

\NewDocumentCommand{\perpML}{}{\ensuremath{\perp_\textnormal{ML}}}
\NewDocumentCommand{\perpFEEC}{}{\ensuremath{\perp_\textnormal{FEEC}}}

\NewDocumentCommand{\cycles}{ o O{h} }{\ensuremath{\mathfrak{Z}\IfValueT{#1}{^{#1}}\IfValueT{#2}{_{#2}}}}
\NewDocumentCommand{\boundaries}{ o O{h} }{\ensuremath{\mathfrak{B}\IfValueT{#1}{^{#1}}\IfValueT{#2}{_{#2}}}}
\NewDocumentCommand{\harmonics}{ o O{h} }{\ensuremath{\mathfrak{H}\IfValueT{#1}{^{#1}}\IfValueT{#2}{_{#2}}}}

\NewDocumentCommand{\cyclesML}{ o O{h} }{\ensuremath{\tilde{\mathfrak{Z}}\IfValueT{#1}{^{#1}}\IfValueT{#2}{_{#2}}}}
\NewDocumentCommand{\boundariesML}{ o O{h} }{\ensuremath{\tilde{\mathfrak{B}}\IfValueT{#1}{^{#1}}\IfValueT{#2}{_{#2}}}}
\NewDocumentCommand{\harmonicsML}{ o O{h} }{\ensuremath{\tilde{\mathfrak{H}}\IfValueT{#1}{^{#1}}\IfValueT{#2}{_{#2}}}}

\NewDocumentCommand{\projH}{ o O{h} }{\ensuremath{\Pi}_{#2}\IfValueT{#1}{^{#1}}}

%% file: theoremsetup.tex
\numberwithin{equation}{section}

\usepackage[svgnames, dvipsnames]{xcolor}
\usepackage[amsmath,thmmarks]{ntheorem}
\usepackage{amsmath,amssymb,amsfonts,mathrsfs}
\usepackage{thmtools}
\theoremstyle{break}
\theorempreskipamount=\parskip
\theorempostskipamount=\parskip
\declaretheorem[title = Theorem, numberwithin = section]{theorem}

\declaretheorem[title = Lemma, sibling = theorem, style = break]{lemma}
\declaretheorem[title = Definition, sibling = theorem, style = break]{definition}
\declaretheorem[title = Proposition, sibling = theorem, style = break]{proposition}
\declaretheorem[title = Assumption, numberwithin = section, style = break]{assumption}
\declaretheorem[title = Remark, sibling = theorem, style = plain]{remark}

\theoremstyle{nonumberplain}
\theorembodyfont{\normalfont}
\theoremsymbol{\ensuremath{\square}}
\newtheorem{proof}{Proof}
\crefname{theorem}{theorem}{theorems}
\Crefname{theorem}{Theorem}{Theorems}

\crefname{lemma}{lemma}{lemmas}
\Crefname{lemma}{Lemma}{Lemmas}

\crefname{corollary}{corollary}{corollaries}
\Crefname{corollary}{Corollary}{Corollaries}

\crefname{proposition}{proposition}{propositions}
\Crefname{proposition}{Proposition}{Propositions}

\crefname{claim}{claim}{claims}
\Crefname{claim}{Claim}{Claims}

\crefname{definition}{definition}{definitions}
\Crefname{definition}{Definition}{Definitions}

\crefname{example}{example}{examples}
\Crefname{example}{Example}{Examples}

\crefname{remark}{remark}{remarks}
\Crefname{remark}{Remark}{Remarks}

\crefname{assumption}{assumption}{assumptions}
\Crefname{assumption}{Assumption}{Assumptions}

\crefname{Rule}{rule}{rules}
\Crefname{Rule}{Rule}{Rules}

%% file: sources.bib
@book{BB13,
  title = {Mixed Finite Element Methods and Applications},
  ISBN = {9783642365195},
  ISSN = {0179-3632},
  url = {http://dx.doi.org/10.1007/978-3-642-36519-5},
  DOI = {10.1007/978-3-642-36519-5},
  journal = {Springer Series in Computational Mathematics},
  publisher = {Springer Berlin Heidelberg},
  author = {Boffi,  Daniele and Brezzi,  Franco and Fortin,  Michel},
  year = {2013}
}

@Book{Arnold2018,
  author    = {Arnold, Douglas N.},
  publisher = {Society for Industrial and Applied Mathematics},
  title     = {Finite Element Exterior Calculus},
  year      = {2018},
  isbn      = {9781611975543},
  month     = dec,
  doi       = {10.1137/1.9781611975543},
}

@Article{Leopardi2016,
  author    = {Leopardi, Paul and Stern, Ari},
  journal   = {SIAM Journal on Numerical Analysis},
  title     = {The Abstract Hodge--Dirac Operator and Its Stable Discretization},
  year      = {2016},
  issn      = {1095-7170},
  month     = jan,
  number    = {6},
  pages     = {3258--3279},
  volume    = {54},
  doi       = {10.1137/15m1047684},
  publisher = {Society for Industrial & Applied Mathematics (SIAM)},
}

@Misc{GUP25,
  author        = {Johnny Guzmán and Pratyush Potu},
  title         = {A Framework for Analysis of DEC Approximations to Hodge-Laplacian Problems using Generalized Whitney Forms},
  year          = {2025},
  archiveprefix = {arXiv},
  eprint        = {2505.08934},
  primaryclass  = {math.NA},
  url           = {https://arxiv.org/abs/2505.08934},
}

@article{mfem,
  title = {MFEM: A modular finite element methods library},
  volume = {81},
  ISSN = {0898-1221},
  url = {http://dx.doi.org/10.1016/j.camwa.2020.06.009},
  DOI = {10.1016/j.camwa.2020.06.009},
  journal = {Computers \& Mathematics with Applications},
  publisher = {Elsevier BV},
  author = {Anderson,  Robert and Andrej,  Julian and Barker,  Andrew and Bramwell,  Jamie and Camier,  Jean-Sylvain and Cerveny,  Jakub and Dobrev,  Veselin and Dudouit,  Yohann and Fisher,  Aaron and Kolev,  Tzanio and Pazner,  Will and Stowell,  Mark and Tomov,  Vladimir and Akkerman,  Ido and Dahm,  Johann and Medina,  David and Zampini,  Stefano},
  year = {2021},
  month = jan,
  pages = {42–74}
}

@MastersThesis{FRI23,
  author = 	 {S. Fritschi},
  title = 	 {Discontinuous Galerkin methods for the Dirac operator},
  school = 	 {D-MATH, ETH Zürich},
  year = 	 2023,
  type = 	 {MSc Thesis in CSE},
  url = {https://people.math.ethz.ch/~hiptmair/StudentProjects/Fritschi.Severin/Thesis_DG_Dirac_Curl.pdf}
}

@article{CHR18,
    AUTHOR = {Christiansen, Snorre H.},
     TITLE = {On eigenmode approximation for {D}irac equations: differential
              forms and fractional {S}obolev spaces},
   JOURNAL = {Math. Comp.},
  FJOURNAL = {Mathematics of Computation},
    VOLUME = 87,
      YEAR = 2018,
    NUMBER = 310,
     PAGES = {547--580},
      ISSN = {0025-5718},
   MRCLASS = {65N30 (65N25 81Q05)},
  MRNUMBER = 3739210,
MRREVIEWER = {Christian Wieners},
       DOI = {10.1090/mcom/3233},
       URL = {https://doi.org/10.1090/mcom/3233},
}

@article{Hiptmair2006,
  title = {Operator Preconditioning},
  volume = {52},
  ISSN = {0898-1221},
  url = {http://dx.doi.org/10.1016/j.camwa.2006.10.008},
  DOI = {10.1016/j.camwa.2006.10.008},
  number = {5},
  journal = {Computers \& Mathematics with Applications},
  publisher = {Elsevier BV},
  author = {Hiptmair,  R.},
  year = {2006},
  month = sep,
  pages = {699–706}
}

@Article{Wittum1990,
  author    = {Wittum, Gabriel},
  journal   = {Numerische Mathematik},
  title     = {On the convergence of multi-grid methods with transforming smoothers: Theory with applications to the Navier-Stokes equations},
  year      = {1990},
  issn      = {0945-3245},
  month     = dec,
  number    = {1},
  pages     = {15--38},
  volume    = {57},
  doi       = {10.1007/bf01386394},
  publisher = {Springer Science and Business Media LLC},
}

@mastersthesis{Dabetic2025,
  author       = {Radovan Dabeti{\'c}},
  title        = {Multigrid Solver for Boundary Value Problems for the Dirac Operator},
  school       = {ETH Zurich, Department of Mathematics (D-MATH)},
  year         = {2025},
  type         = {Master's thesis},
  address      = {Z{\"u}rich, Switzerland},
  month        = sep,
  day          = {11},
  url          = {https://people.math.ethz.ch/~hiptmair/StudentProjects/Dabetic.Radovan/MScThesis_RadovanDabetic.pdf}
}

@article{CW18,
  title = {MultiGrid Preconditioners for Mixed Finite Element Methods of the Vector Laplacian},
  volume = {77},
  ISSN = {1573-7691},
  url = {http://dx.doi.org/10.1007/s10915-018-0697-7},
  DOI = {10.1007/s10915-018-0697-7},
  number = {1},
  journal = {Journal of Scientific Computing},
  publisher = {Springer Science and Business Media LLC},
  author = {Chen,  Long and Wu,  Yongke and Zhong,  Lin and Zhou,  Jie},
  year = {2018},
  month = mar,
  pages = {101–128}
}

@Book{Hackbusch1985,
  author    = {Hackbusch, Wolfgang},
  publisher = {Springer Berlin Heidelberg},
  title     = {Multi-Grid Methods and Applications},
  year      = {1985},
  isbn      = {9783662024270},
  doi       = {10.1007/978-3-662-02427-0},
  issn      = {0179-3632},
  journal   = {Springer Series in Computational Mathematics},
}

@article{Wittum1989,
  title = {Multi-grid methods for stokes and navier-stokes equations: Transforming smoothers: algorithms and numerical results},
  volume = {54},
  ISSN = {0945-3245},
  url = {http://dx.doi.org/10.1007/BF01396361},
  DOI = {10.1007/bf01396361},
  number = {5},
  journal = {Numerische Mathematik},
  publisher = {Springer Science and Business Media LLC},
  author = {Wittum,  G.},
  year = {1989},
  month = sep,
  pages = {543–563}
}

@inbook{BD79,
  title = {Multigrid Solutions to Elliptic Flow Problems},
  ISBN = {9780125460507},
  url = {http://dx.doi.org/10.1016/B978-0-12-546050-7.50008-3},
  DOI = {10.1016/b978-0-12-546050-7.50008-3},
  booktitle = {Numerical Methods for Partial Differential Equations},
  publisher = {Elsevier},
  author = {Brandt,  Achi and Dinar,  Nathan},
  year = {1979},
  pages = {53–147}
}

@Misc{Spectra,
  author       = {Qiu, Yixuan},
  title        = {Spectra: a {C}++ library for large scale eigenvalue problems},
  howpublished = {GitHub repository},
  url          = {https://github.com/yixuan/spectra},
  note         = {Accessed 2026-06-07},
}
